\documentclass[11pt,a4paper]{article}
\usepackage{amsmath,amssymb,amsthm,amsfonts,latexsym,graphicx,subfigure}
\usepackage[all,import]{xy}
\usepackage{indentfirst,cite}
\usepackage{tabularx,booktabs}
\usepackage{caption,amscd}
\usepackage{extarrows}
\usepackage{hyperref}

\usepackage{fancyhdr,enumerate}
\usepackage{bbm,color}
\usepackage{tikz}
\usetikzlibrary{shapes.geometric, arrows}
\usepackage{accents,cases}
\usepackage{multirow}
\usepackage{euscript}\usepackage{wasysym}\usepackage{pdfsync}\usepackage{comment}

\usepackage{array,float}
\newcommand{\PreserveBackslash}[1]{\let\temp=\\#1\let\\=\temp}
\newcolumntype{C}[1]{>{\PreserveBackslash\centering}p{#1}}
\newcolumntype{R}[1]{>{\PreserveBackslash\raggedleft}p{#1}}
\newcolumntype{L}[1]{>{\PreserveBackslash\raggedright}p{#1}}

\DeclareMathOperator*{\argmin}{\ensuremath{arg\,min}}

\DeclareMathOperator*{\Sgn}{\ensuremath{Sgn}}

\DeclareMathOperator*{\vol}{\ensuremath{vol}}

\newcommand{\gen}{\mathrm{genus}}

\newcommand{\B}{\ensuremath{\mathcal{B}}}
\makeatletter
\def\wbar{\accentset{{\cc@style\underline{\mskip8mu}}}}

\makeatother

\renewcommand{\vec}[1]{\mbox{\boldmath \small $#1$}}

\newcommand{\power}{\ensuremath{\mathcal{P}}}
\DeclareMathSymbol{\mdot}{\mathord}{symbols}{"01}

\newcommand{\supp}{\ensuremath{\mathrm{supp}}}
\newcommand{\R}{\ensuremath{\mathbb{R}}}
\newcommand{\A}{\ensuremath{\mathcal{A}}}

\theoremstyle{plain}
\newtheorem{theorem}{Theorem}[section]

\newtheorem{defn}{Definition}[section]

\newtheorem{lemma}{Lemma}[section]
\newtheorem{remark}{Remark}

\newtheorem{cor}{Corollary}[section]
\newtheorem{pro}{Proposition}[section]
\newtheorem{example}{Example}[section]

\newtheorem{Conj}{Conjecture}

\def\D{{\mathcal D}}

\usepackage{comment}
\usepackage[colorinlistoftodos]{todonotes}

\allowdisplaybreaks[4]
\begin{document}
\bibliographystyle{unsrt}
\title{Discrete-to-Continuous Extensions: piecewise multilinear extension, min-max theory and spectral theory} 
\author{J\"urgen Jost\footnotemark[1], \and Dong Zhang\footnotemark[1]}
\footnotetext[1]{Max Planck Institute for Mathematics in the Sciences, Inselstrasse 22, 04103 Leipzig,
Germany. \\Email addresses:
{\tt  jost@mis.mpg.de}  (J\"urgen Jost),\; {\tt dzhang@mis.mpg.de}\, and {\tt 13699289001@163.com} (Dong Zhang).
}
\date{}
\maketitle

\begin{abstract}
\small

We introduce 
the 
homogeneous and piecewise multilinear extensions and the eigenvalue problem for locally Lipschitz function pairs, in order to develop a systematic framework for relating discrete and continuous min-max problems. 
This  also enables us to investigate spectral properties for pairs of $p$-homogeneous functions and to propose a critical point theory for  zero-homogeneous functions.
The main contributions are:

\begin{enumerate}[(1)]
    \item We provide several min-max relations between an  original discrete formulation and its
piecewise multilinear 
extension. We introduce  the concept of perfect domain pairs to view  comonotonicity on vectors  as an extension of inclusion chains on sets.   The piecewise multilinear extension is  (slice-)rank preserving, which closely relates to Tao's lemma on diagonal tensors.  More discrete-to-continuous equalities are obtained, including a general form involving log-concave polynomials. 
  And by employing  these fundamental  correspondences,   we get further results and applications on  tensors,  Tur\'an's  problem, signed (hyper-)graphs, etc. 
    \item 
    We derive the mountain pass characterization, linking theorems, nodal domain inequalities,  inertia bounds,    duality theorems
   and distribution of eigenvalues for pairs of $p$-homogeneous functions. We
   establish a new property on the subderivative of a convex function which
   relates to the Gauss map of the graph of the convex function.   Based on
   these fundamental results, we can analyze the  structure of
   eigenspaces in depth. For example, we show a simple one-to-one correspondence between
   the nonzero  eigenvalues of the  vertex  $p$-Laplacian  
and the 
 edge 
$p^*$-Laplacian of a graph.  We can also apply the theory to  Cheeger inequalities and  $p$-Laplacians on  oriented   hypergraphs and simplicial complexes.  
    Also, the first nonlinear analog of Huang's approach for   hypergraphs is provided. 
    
\end{enumerate}

%

\vspace{0.2cm}

 
 \vspace{0.2cm}

\end{abstract}

\tikzstyle{startstop} = [rectangle, rounded corners, minimum width=1cm, minimum height=1cm,text centered, draw=black, fill=red!0]
\tikzstyle{io1} = [rectangle, trapezium left angle=80, trapezium right angle=100, minimum width=1cm, minimum height=1cm, text centered, draw=black, fill=blue!0]
\tikzstyle{io2} = [trapezium,  rounded corners, trapezium left angle=100, trapezium right angle=100, minimum width=1cm, minimum height=1cm, text centered, draw=black, fill=yellow!0]
\tikzstyle{process} = [rectangle, minimum width=1cm, minimum height=1cm, text centered, draw=black, fill=orange!0]
\tikzstyle{decision} = [circle, minimum width=1cm, minimum height=1cm, text centered, draw=black, fill=green!0]
\tikzstyle{decision2} = [ellipse, rounded corners=10mm, minimum width=2cm, minimum height=2cm, text centered, draw=black, fill=green!0]
\tikzstyle{arrow} = [thick,->,>=stealth]
\tableofcontents
\section{Introduction}\label{sec:introduction} 


In his millennium paper  \cite{Lovasz00},  Lov\'asz wrote: \textit{Connections
  between discrete and continuous may be the subject of mathematical study on
  their own right.} In fact, over the last few decades, many firm
  bridges between the discrete data world and the field of continuous
  mathematics have been established, and they are not only interesting in
  themselves, but also  helpful and stimulating in both pure and applied mathematics. 

\vspace{0.2cm}

A natural and important idea for connecting discrete  and continuous problems is by considering the discrete space as  a subset of some Euclidean space and then extending or interpolating discrete functions to continuous ones. A systematic scheme for such extension was provided in the  fundamental    works of Choquet and Lov\'asz  \cite{Choquet54,Lovasz}. It was further developed in  a series of works of the machine learning group of Hein  \cite{HS11,TVhyper-13}, and in  recent works by the authors \cite{CSZZ,JostZhang-Morse,JostZhang-PL}. For example,  concerning optimization, there have been various schemes to solve combinatorial optimization problems
by means of continuous optimization methods,  including continuous reformulations \cite{Lovasz,HS11,SheraliAdams90,Burer09} and continuous relaxations \cite{LawlerWood66,GoemansWilliamson95,BillionnetElloumiLambert12}, which turn out to be powerful. In this work, we also present some results for both the reformulations  and relaxations via our extension theory. 

\vspace{0.2cm}

In the continuous case, the  functions that are  best suited for optimization are the convex ones. In fact, convex analysis  has become a well developed and important subjects, with ramifications in most areas of mathematics.  In the discrete case, a  role that is in some sense analogous to that of convex functions is played by the submodular functions \cite{Choquet54,Lovasz}, and their analysis 
  was systematically
developped by many mathematicians (such as Fujishige \cite{F05-book}, Murota
\cite{Murota98,Murota03book}, Dress   et al \cite{DressWenzel92})  from
different viewpoints. In fact,  submodularity is  a kind of `discrete convexity', and the Lov\'asz extension
turns  submodular functions into  convex ones.  This scheme has been
applied in many areas like game theory, matroid theory, stochastic processes,
electrical networks, computer vision and machine learning.  In this
  direction, we have systematically investigated  submodularity  via multi-way
Lov\'asz extensions in \cite{JostZhang-PL}. 

\vspace{0.2cm}

Extension theory is useful, however, also beyond the convex or submodular setting. It can be more widely applied to  optimization
\cite{HS11,TVhyper-13,JostZhang-PL}, critical point theory and Morse theory
\cite{JostZhang-Morse}, to cite just some examples. It is therefore natural to systematically consider min-max theory from the perspective of extension theory. Such a min-max theory  includes, for instance, 
saddle point problems, von Neumann’s minimax theorem
\cite{Neumann28,Sion58,Krantz14} or  Lusternik-Schnirelmann theory. A main   contribution of this paper therefore is to reveal the connections between discrete  min-max problems and continuous min-max reformulations from  different  viewpoints.

\vspace{0.2cm}

In particular, in graph theory and its extensions, many problems can be considered as saddle point problems. For example, eigenvalues are characterized as critical values of Rayleigh quotients, and more generally, problems related to eigenvectors, such as properties of nodal domains, can also be seen from that perspective. The same holds for other fundamental graph theoretical problems, like Cheeger cuts. Such problems find their natural place in nonlinear spectral theory. 
Previous research  indicates
that we can explore the corresponding nonlinear spectral graph theory with
the help of the corresponding  continuous objective function.  Accordingly, besides the practical need of designing continuous optimization algorithms
for combinatorial problems,  
these  continuous versions enable 
us to reconsider the combinatorial  problems from the viewpoint of 
spectral theory. Thus, we develop a systematic  spectral theory for a pair of homogeneous and locally  Lipschitz functions,  as a  solid foundation of extension methods. 

\vspace{0.2cm}

Based on the extension theory and the spectral theory, we provide general, yet
user-friendly tools  that can be used when attacking discrete models arising
in applications. More concretely, continuing the study in \cite{JostZhang-PL},
we systematically  develop further   applications,  such as the nodal domain theorem and inertia bounds involving adjacency tensors of uniform hypergraphs, inertia bounds for the  graph $p$-Laplacian, the $k$-way Cheeger inequality on oriented hypergraphs, the eigenvalues of tensors, 
 Cheeger-type inequalities for Hodge Laplacians on simplicial complexes, and  spectral estimates for signed hypergraphs.  These results  indicate that the extension theory might be an excellent universal approach to understand the
discrete problems via their continuous extensions. 

\vspace{0.2cm}

The general structure that we are exploring can be compactly represented in the following diagram:

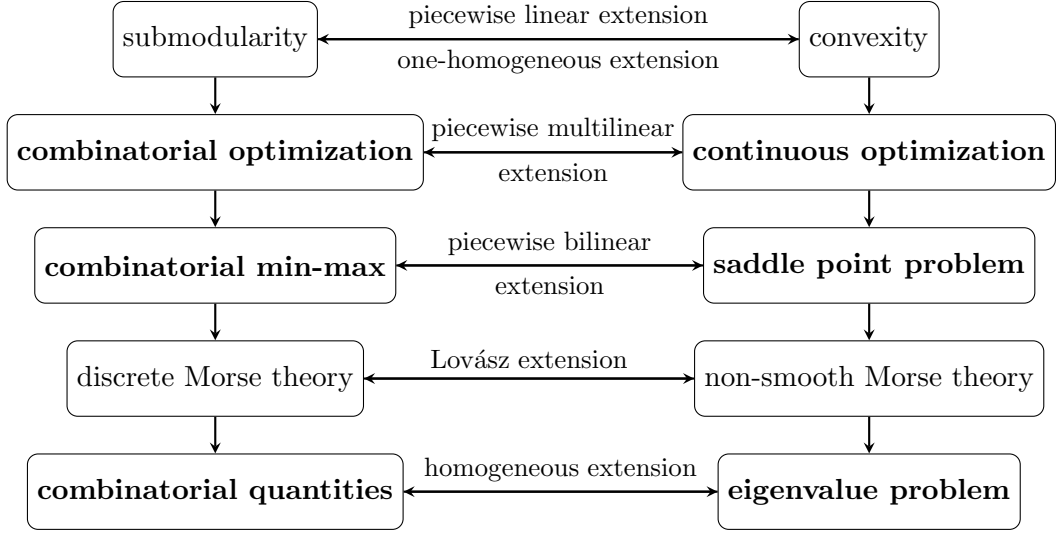
\begin{center}
\begin{tikzpicture}[node distance=3.6cm]

\node (submodular) [startstop] {  submodularity };

\node (optimal) [startstop, below of=submodular, xshift=0cm, yshift=2.1cm]  { \textbf{combinatorial optimization} };

\node (minmax) [startstop, below of=optimal, xshift=0cm, yshift=2.1cm]  { \color{black} 
\textbf{combinatorial min-max}};

\node (critical) [startstop, below of=minmax, xshift=0cm, yshift=2.1cm]  { discrete Morse theory };

\node (quantity) [startstop, below of=critical, xshift=0cm, yshift=2.1cm]  { \textbf{combinatorial  quantities} };

\node (convex) [startstop, right of=submodular, xshift=5cm, yshift=0cm] {  convexity };

\node (optimal-) [startstop, below of=convex, xshift=0cm, yshift=2.1cm]  { \textbf{continuous optimization} };

\node (minmax-) [startstop, below of=optimal-, xshift=0cm, yshift=2.1cm]  { \color{black} 
\textbf{saddle point problem} };

\node (Morse) [startstop, below of=minmax-, xshift=0cm, yshift=2.1cm]  { non-smooth Morse theory };

\node (spectral) [startstop, below of=Morse, xshift=0cm, yshift=2.1cm]  { \color{black} \textbf{eigenvalue problem} }; 

\draw [arrow](submodular) --node[anchor=south] { \small  } (optimal);
\draw [arrow](optimal) --node[anchor=south] { \small  } (minmax);
\draw [arrow](minmax) --node[anchor=south] { \small  }(critical);
\draw [arrow](critical) --node[anchor=south] { \small  }(quantity);

\draw [arrow](convex) --node[anchor=south] { \small  } (optimal-);
\draw [arrow](optimal-) --node[anchor=south] { \small  } (minmax-);
\draw [arrow](minmax-) --node[anchor=south] { \small  }(Morse);
\draw [arrow](Morse) --node[anchor=south] { \small  }(spectral);

\draw [arrow](submodular) --node[anchor=south] { \small  piecewise linear extension  } (convex);
\draw [arrow](convex) --node[anchor=north] { \small  one-homogeneous  extension 
} (submodular);
\draw [arrow](optimal) --node[anchor=south] { \small  piecewise multilinear  } (optimal-);
\draw [arrow](optimal-) --node[anchor=north] { \small extension
} (optimal);
\draw [arrow](minmax) --node[anchor=south] { \small piecewise bilinear  }(minmax-);
\draw [arrow](minmax-) --node[anchor=north] { \small extension }(minmax);
\draw [arrow](critical) --node[anchor=south] { \small Lov\'asz extension 
}(Morse);
\draw [arrow](Morse) --node[anchor=north] { \small  }(critical);
\draw [arrow](quantity) --node[anchor=south] { \small homogeneous extension }(spectral);
\draw [arrow](spectral) --node[anchor=north] { \small 
}(quantity);
\end{tikzpicture}
\end{center}

Some of the above relations are discussed and investigated in \cite{JostZhang-Morse,JostZhang-PL}, and in this paper, we will complete the above framework in particular by relating discrete and continuous saddle points and eigenvalue problems. 


For simplicity, we begin with the following {\sl \textbf{piecewise bilinear extension}}:

Given $V=\{1,\cdots,n\}$ and its power set $\mathcal{P}(V)$, for $\vec x=(x_1,\cdots,x_n)$ and $\vec y=(y_1,\cdots,y_n)$ in $\R^n$, let  $\sigma,\tau:V\cup\{0\}\to V\cup\{0\}$ be  permutations such that $ x_{\sigma(1)}\le x_{\sigma(2)} \le \cdots\le x_{\sigma(n)}$,  $y_{\tau(1)}\le \cdots\le y_{\tau(n)}$ and $\sigma(0)=\tau(0)=0$ with $x_0:=y_0:=0$. 

For a discrete function $f:\mathcal{P}(V)\times \mathcal{P}(V)\to \R$, the {\sl piecewise bilinear  extension} of $f$ at $(\vec x,\vec y)$ is  
\begin{equation}\label{eq:PQ-extension}
f^{Q}(\vec x,\vec y)=\sum_{i,j=0}^{n-1}(x_{\sigma(i+1)}-x_{\sigma(i)})(y_{\tau(j+1)}-y_{\tau(j)})f(V^{\sigma(i)}(\vec x),V^{\tau(j)}(\vec y)),
\end{equation}
where  
 $V^{\sigma(i)}(\vec x):=\{j\in V: x_{j}> x_{\sigma(i)}\},\; i=1,\cdots,n-1,\; V^0(\vec x)=V$, and the definition of $V^{\tau(j)}(\vec y)$ is analogous. 
We can rewrite   \eqref{eq:PQ-extension}
 in an integral form as \begin{align*}\label{eq:PQ-extension2}
 f^Q(\vec x,\vec y)=&\int_{\min\vec y}^{\max\vec y}\int_{\min\vec x}^{\max\vec x}f(V^{t}(\vec x),V^{s}(\vec y))dt ds+\min \vec x\int_{\min\vec y}^{\max\vec y}f(V,V^{s}(\vec y))ds
\\&\; +\min \vec y\int_{\min\vec x}^{\max\vec x}f(V^{t}(\vec x),V)dt+\min \vec x\min \vec yf(V,V),
 \end{align*}
 where $\min\vec x:=\min\limits_{i=1,\cdots,n}x_i$, $\max\vec x:=\max\limits_{i=1,\cdots,n}x_i$, and $V^t(\vec x):=\{i\in V:x_i>t\}$. For a function $f:\power(V_1)\times \power(V_2)\to\R$, one can define $f^Q(\vec x,\vec y)$ in the same way. It is therefore not necessary to write  the details here. 
 
Clearly,  the piecewise bilinear extension 
is  $2$-homogeneous, and it constitutes a generalization of the original Lov\'asz extension. In fact, taking  $\vec y=\vec 1$, we have $f^Q(\vec x,\vec 1)=\tilde{f}^L(\vec x)$, where $\tilde{f}(A):=f(A,V)$ for any $A\in\mathcal{P}(V)$.  
The bilinear  extension  and its generalizations (see Section   \ref{sec:extension}) possess 
many connections with various fields like optimization, saddle point problems, critical point theory and spectral graph theory. 

\vspace{0.16cm}

\textbf{Connections with saddle point problems}
 
\begin{theorem}[Theorem \ref{thm:min-max-ABxy} and Proposition \ref{pro:min-max-submodular}]
 \label{thm:quadratic-saddle} 
Let $f:\power(V_1)\times \power(V_2)\to \R$, $g:\power(V_1)\times \power(V_2)\to \R_{\ge 0}$,  $n=\#V_1$ and $m=\#V_2$.  Denote by $\R^m_+=(0,+\infty)^m$ and $\R^n_{\ge0}=[0,+\infty)^n$. Then  \begin{equation}\label{eq:continuous-ex}
\inf\limits_{\vec x\in\R^n_{\ge0}\setminus\{\vec0\}}\sup\limits_{\vec y\in\R^m_+}\frac{f^Q(\vec x,\vec y)}{g^Q(\vec x,\vec y)}=\sup\limits_{\vec y\in\R^m_+}\inf\limits_{\vec x\in\R^n_{\ge0}\setminus\{\vec0\}}\frac{f^Q(\vec x,\vec y)}{g^Q(\vec x,\vec y)}
\end{equation}
if either of the followings holds. 
\begin{itemize}
    \item[(a)]  $g$ is positive, and  \begin{equation}\label{eq:discrete-ex}
\min\limits_{A\in\power(V_1)\setminus\{\varnothing\}}\max\limits_  {B\in\power(V_2)\setminus\{\varnothing\}}\frac{f(A,B)}{g(A,B)}=\max\limits_{B\in\power(V_2)\setminus\{\varnothing\}} \min\limits_{A\in\power(V_1)\setminus\{\varnothing\}}\frac{f(A,B)}{g(A,B)}    \end{equation}
And in this case, \eqref{eq:discrete-ex} and \eqref{eq:continuous-ex} coincide. Moreover,  $(A^*,B^*)$ is a saddle point of $f/g$ if and only if $(\vec 1_{A^*},\vec 1_{B^*})$ is a saddle point of $f^Q/g^Q$.

\item[(b)]   $g$ is  modular on each component with $g(\{i\},V_2)>0$ and $g(V_1,\{j\})>0$ for any $i\in V_1,j\in V_2$, and $f$ satisfies  the following conditions:
\begin{itemize}
    \item[$\bullet$] $f$ is submodular  in its first  component;
    \item[$\bullet$] $f$ is  supermodular in its  second component.
\end{itemize}
\end{itemize}
\end{theorem} 

Theorem  \ref{thm:quadratic-saddle} also holds when we replace the piecewise bilinear  extension by some other extensions (see Theorem  \ref{inthm:min-max-ABxy} and  similar results in Section  \ref{sec:extension}). The condition (b) in Theorem  \ref{thm:quadratic-saddle} makes contact with Sion's  min-max theorem, 
and it closely relates to the corresponding topics in  game theory. Moreover,
the  formulation \eqref{eq:continuous-ex} allows us to deal with the Collatz-Wielandt formula (see Lemma \ref{lemma:Collatz-Wielandt} and Example \ref{ex:saddle-conterexample}) and von Neumann's minimax theorem  for matrices (see Example \ref{ex:saddle-Two-Person-Zero-Sum}) in a single, 
unifying mathematical framework. 

As a systematic  research on the extension theory,  we  introduce several  homogeneous extensions of a discrete function (see Definitions \ref{defn:piece-multilinear}, \ref{defn:piece-polynomial} and \ref{def:multiple-integral}). Below, we present the {\sl piecewise multilinear   extension}:

For a discrete function $f:\mathcal{P}(V)^k\to \R$, we have the piecewise multilinear function  $f^M:(\R^{n})^k\to\R$ defined by
\begin{equation*}
f^M(\vec x^1,\cdots,\vec x^k)=\sum_{i_1,\cdots,i_k=0}^{n-1}\prod_{l=1}^k(x_{(i_l+1)}^l-x_{(i_l)}^l)f(V^{(i_1)}(\vec x^1),\cdots,V^{(i_k)}(\vec x^k)),
\end{equation*}
where $V^{(i_l)}(\vec x^l):=\{j\in V: x_j^l> x_{(i_l)}^l\}$ and $x_{(1)}^l\le x_{(2)}^l\le\ldots\le x_{(n)}^l$ is a rearrangement of $\vec x^l:=(x_1^l,\ldots,x_n^l)$ in non-decreasing order,  $x_{(0)}^l:=0$ and $V^{(0)}(\vec x^l):=V$, $l=1,\cdots,k$,  
 $\vec x^1,\cdots,\vec x^k\in\R^n$. For $k=2$, this reduces of course to \eqref{eq:PQ-extension}.

\begin{theorem}[Theorem \ref{thm:min-max-ABxy}]
\label{inthm:min-max-ABxy}
Suppose $f^M, g^M$ are piecewise multilinear  extensions of $f,g:\power^{k+l}(V)\to\R$, where $k$ and $l$ are positive integers. 
If
\begin{equation}\label{eq:min-max-k-l}
\min\limits_{B\in(\power(V)\setminus\{\varnothing\})^l}\max\limits_{A\in(\power(V)\setminus\{\varnothing\})^k} \frac{f(A,B)}{g(A,B)}=\max\limits_{A\in(\power(V)\setminus\{\varnothing\})^k} \min\limits_{B\in(\power(V)\setminus\{\varnothing\})^l}\frac{f(A,B)}{g(A,B)},
\end{equation}
then 
\begin{equation*}
\inf\limits_{\vec y\in\R^{ln}_+}\sup\limits_{\vec x\in\R^{kn}_+}\frac{f^M(\vec x,\vec y)}{g^M(\vec x,\vec y)}=\sup\limits_{\vec x\in\R^{kn}_+}\inf\limits_{\vec y\in\R^{ln}_+}\frac{f^M(\vec x,\vec y)}{g^M(\vec x,\vec y)}
\end{equation*}
which coincides with  \eqref{eq:min-max-k-l}. 
Moreover, $(A^*,B^*)$ is a saddle point of $f/g$ if and only if $(\vec 1_{A^*},\vec 1_{B^*})$ is a saddle point of $f^M/g^M$.
\end{theorem}

It should be noted that Theorem \ref{inthm:min-max-ABxy} is a generalization of the equivalence between a combinatorial optimization and the fractional programming  produced by the  multi-way Lov\'asz extension (Theorem A in  \cite{JostZhang-PL}). The detailed reason is shown  in  Remark \ref{remark:von-generalize-optimization}. On the other hand, Theorem \ref{inthm:min-max-ABxy} provides the first relation between a discrete saddle point  problem and its homogeneous extension, which closely relates to von Neumann's minimax theorem.

\vspace{0.16cm}

\textbf{Connections with  Lusternik-Schnirelmann theory }

We set up a min-max relation in the style of  Lusternik-Schnirelmann theory in Section \ref{sec:LS}. It is convenient to state the result in the context of the {\sl multiple integral extension}:

For a function $f:\power_2(V)^{k}\to \R$, we define $f^M:(\R^{n})^k\to\R$ as 
\begin{align*}
&f^M(\vec x^1,\cdots,\vec x^k)\\=~&\int_0^{\|\vec x^k\|_\infty}\cdots\int_0^{\|\vec x^1\|_\infty}f(V^{t_1}_+(\vec x^1),V^{t_1}_-(\vec x^1),\cdots,V^{t_k}_+(\vec x^k),V^{t_k}_-(\vec x^k))dt_1\cdots dt_k,    
\end{align*}
where $\power_2(V)=\{(A_+,A_-):A_+,A_-\subset V,A_+\cap A_-=\varnothing\}$,  and $V^{t_l}_\pm(\vec x^l)=\{j\in V:\pm x^l_j>t_l\}$, $l=1,\cdots,k$. 

\begin{remark}
The multiple integral extension $f^M$ of a function $f:\power_2(V)^k\to\R$ and the previous piecewise  multilinear extension $h^M$ of a function $h:\power(V)^k\to\R$ have the following   relations:
\begin{enumerate}[(a)]
\item  If $f(A_{1+},A_{1-},\cdots,A_{k+},A_{k-})=h(A_{1+},A_{2+},\cdots,A_{k+})$,\\ $\forall (A_{1+},A_{1-},\cdots,A_{k+},A_{k-})\in \power_2(V)^k$, then $f^M(\vec x)=h^M(\vec x)$, $\forall \vec x\in[0,\infty)^{nk}$. 
\item If $f(A_{1+},A_{1-},\cdots,A_{k+},A_{k-})=h(A_{1+}\cup A_{1-},\cdots,A_{k+}\cup A_{k-})$,\\ $\forall (A_{1+},A_{1-},\cdots,A_{k+},A_{k-})\in \power_2(V)^k$, then  $f^M(\vec x)=h^M(|\vec x|)$, $\forall \vec x\in (\R^{n})^k$. 
\end{enumerate}
Moreover, given a function 
$f:\power_2(V)^k\to\R$,   define $\tilde{f}:\power(V\sqcup V')^k\to\R$ by  $\tilde{f}(A_1,\cdots,A_k)=f(A_1\cap V\setminus \phi(A_1\cap V'),\phi(A_1\cap V')\setminus (A_1\cap V),\cdots,A_k\cap V\setminus \phi(A_k\cap V'),\phi(A_k\cap V')\setminus (A_k\cap V))$, where $V'$ is a copy of $V$, and $\phi:V'\to V$ is the bijection satisfying  $i'\mapsto i$, $\forall i'\in V'$. 
Then, $f^M(\vec x^1,\cdots,\vec x^k)=\tilde{f}^M(\vec x^1_+,\vec x^1_-,\cdots,\vec x^k_+,\vec x^k_-)$, where $\vec x^i\in\R^n$ and $\vec x_\pm:=(\pm \vec x)\vee \vec0\in[0,+\infty)^n$.

In summary, we can embed $\power(V)^k$ into $\power_2(V)^k$, and embed $\power_2(V)^k$ into $\power(V\sqcup V')^k$. The multiple integral extension  agrees with the piecewise multilinear extension on the first quadrant, and their relations can be reduced to the correspondences  between the original Lov\'asz extension and the disjoint-pair version. 
\end{remark}  

\begin{theorem}[Section \ref{sec:LS}]\label{thm:Cheeger-type}
Under the notions in Section \ref{sec:extension}, for 
$f,g:\tilde{P}_1(V)\to\R_+$,  we have
\begin{equation}\label{eq:minmax-union}
\min_{\{A^j\}\in \tilde{P}_{m}(V)}\max_{A\in \Sigma\{A^j\}} \frac{f(A)}{g(A)} \ge \inf_{\gen(X)\ge m}\sup\limits_{\vec x\in X} \frac{f^M(\vec x)}{g^M(\vec x)} \ge \max_{\{A^j\}\in \tilde{P}_{n+1-m}(V)}\min_{A\in \Sigma\{A^j\}} \frac{f(A)}{g(A)}.
\end{equation}
If we further assume that  $f$ is  submodular and symmetric as well as $g$ is   supermodular and symmetric,  then 
$$\min_{\{A^j\}\in \tilde{P}_{m}(V)}\max_{i=1,\cdots,m} \frac{f(A^i)}{g(A^i)} \ge \inf_{\gen(X)\ge m}\sup\limits_{\vec x\in X} \frac{f^L(\vec x)}{g^L(\vec x)}:=\lambda_m\ge \min_{\{A^j\}\in \tilde{P}_{k_m}(V)}\max_{i=1,\cdots,m} \frac{f(A^i)}{g(A^i)}$$
where $k_m$ is the largest  number of  nodal domains of  eigenvectors w.r.t. the $m$-th min-max  eigenvalue $\lambda_m$ of the function pair $(f^L,g^L)$, and  $f^L$ represents the disjoint-pair  Lov\'asz extension of $f$. 
\end{theorem}

This is a general version of higher-order Cheeger-type inequalities for the couple of $f$ and $g$. 
And also, taking  $m\in\{1,n\}$  in \eqref{eq:minmax-union}, we get Theorem B in  \cite{JostZhang-PL}.

\vspace{0.16cm}

\textbf{Connections with  combinatorial optimization}

The piecewise multilinear extension also shows a way to get an equivalence between  discrete and continuous optimizations, which enhances   the corresponding  results in \cite{JostZhang-PL}. Basically, in Section \ref{sec:extension}, we introduce the perfect domain pair $(\A,\D)$ w.r.t. a given  homogeneous extension, denoted by `$\sim$' and defined by the property that 
$$\sup\limits_{A\in\A}\frac{f(A)}{g(A)}=\sup\limits_{\vec x\in\D}\frac{\widetilde{f}(\vec x)}{\widetilde{g}(\vec x)}\;\;\text{ and }\;\; \inf\limits_{A\in\A}\frac{f(A)}{g(A)}=\inf\limits_{\vec x\in\D}\frac{\widetilde{f}(\vec x)}{\widetilde{g}(\vec x)}$$
hold for all suitable functions $f,g$ and their extensions $\widetilde{f}, \widetilde{g}$ satisfying suitable properties.  This is our main idea  to realize a continuous reformulation of a discrete  optimization. 
Both the piecewise multilinear  extension and the multiple integral extension are investigated systematically along this direction. 
 For example, we can get a new continuous representation of the maxcut problem on graphs:   $$\max\limits_{S\subset V}|\partial S|=\max\limits_{x,y\in\R^n_{\ge0},x^\top  y=0}\frac{\sum_{i,j=1}^nw_{ij}x_iy_j}{\|\vec x\|_\infty\|\vec y\|_\infty},$$
where $(w_{ij})$ is the weighted adjacency matrix of the graph.  
We also have a new equivalent optimization of the dual Cheeger constant: 
$$\max\limits_{S\cup T\subset V,S\cap T=\varnothing }\frac{\# E(S,T)}{\vol(S)+\vol(T)}=\max\limits_{x,y\in\R^n_{\ge0},x^\top  y=0}\frac{\sum_{i,j=1}^nw_{ij}x_iy_j}{\|\vec x\|_\infty\sum_{i\in V}\deg_iy_i+\|\vec y\|_\infty\sum_{i\in V}\deg_ix_i}.$$

More interestingly, we obtain a  more general equality with the help of 
 log-concave polynomials \cite{AOV18,BH20-Lorentzian}: \begin{pro}\label{pro:log-concave-optimal}
For a log-concave  polynomial $P$ of degree $d$ in $n$ variables,  and for $f_1,\cdots,f_n:\A\to[0,+\infty)$, we have
$$\min_{A\in\A} \frac{P(f_1(A),\cdots,f_n(A))}{(f_1(A)+\cdots+f_n(A))^d}=\inf\limits_{\vec x\in\D}\frac{P(f_1^M(\vec x),\cdots,f_n^M(\vec x))}{(f_1^M(\vec x)+\cdots+f_n^M(\vec x))^d}$$
where $(\A,\D)$ forms a perfect  domain pair.
\end{pro}

In addition,  based on the equivalence and the extension approach, we can obtain some useful continuous relaxations like Theorem  \ref{thm:Turan-general} below. 

\vspace{0.16cm}

\textbf{Connections with the Tur\'an problem  and spectral graph theory }

For $f^M:(\R^{n})^k\to\R$, define $f^M_\triangle:\R^n\to\R$ by
$f^M_\triangle(\vec x):=f^M(\vec x,\cdots,\vec x)$, $\forall \vec x\in\R^n$.  
\begin{theorem}\label{thm:Turan-general}
 Given  $f:\power(V)^k\to \R$ and $g:\power(V)^k\to \R_+$,   as well as  their piecewise multilinear  extensions  $f^M$ and $g^M$, denote by  $f_{\triangle}(A)=f(A,\cdots,A)$ and $g_{\triangle}(A)=g(A,\cdots,A)$. Then
 \begin{align*}
\max\limits_{A\subset V}\frac{f_\triangle(A)}{g_\triangle(A)}&\le \max\limits_{\vec x\in \R^n_{\ge0}} \frac{f^M_\triangle(\vec x)}{g^M_\triangle(\vec x)}\le \max\limits_{\text{chain }\{A_1, A_2,\cdots, A_k\}}\frac{f(A_1,\cdots,A_k)}{g(A_1,\cdots,A_k)}\\&=
 \max\limits_{\text{comonotonic }\vec x^1,\cdots,\vec x^k\in\R^n_{\ge0}}
 \frac{f^M(\vec x^1,\cdots,\vec x^k)}{g^M(\vec x^1,\cdots,\vec x^k)}
 \end{align*}
where the chain is in the sense of inclusion, and the  vectors $\vec x$ and $\vec y$ are comonotonic if $(x_i-x_j)(y_i-y_j)\ge0$, $i,j\in V$. 
All  `$\le$' become  `$\ge$' if we change all `$\max$' to `$\min$'. 
\end{theorem}

It is also a generalization of Theorem A in  \cite{JostZhang-PL} by taking $k=1$. In addition, it shows a way to rediscover the   Motzkin-Straus theorem and the Lagrangian method on  Tur\'an's problem (see Section \ref{sec:Turan}). Importantly, the identity in Theorem \ref{thm:Turan-general} indicates that  the comonotonicity on  vectors/functions can be regarded as an extension of the  inclusion relation.  
Roughly speaking, $$(\{\text{inclusion chains}\},\{\text{pairwise comonotonic vectors/functions}\})\text{ is a perfect domain pair.}$$

Next we give an example for the application of the above results to tensors.  
\begin{example}\label{ex:tensor}
An order-$k$  $n$-dimensional tensor  $(c_{i_1,\cdots,i_k})$ is a set of $n^k$ entries. It is nonnegative if $c_{i_1,\cdots,i_k}\ge 0$, and it is symmetric if $c_{i_1,\cdots,i_k}=c_{\sigma(i_1),\cdots,\sigma(i_k)}$ for any permutation $\sigma\in S_k$. Now we define a  function $f:\mathcal{P}^k(V)\to \R$ by $f(V_1,\cdots,V_k)=\sum_{i_1\in V_1,\cdots,i_k\in V_k}c_{i_1,\cdots,i_k}$ for any $V_1,\cdots,V_k\subset V$. Then $f^M_\triangle(\vec x)=\sum_{i_1,\cdots,i_k\in V}c_{i_1,\cdots,i_k}x_{i_1}\cdots x_{i_k}$. The tensor  $(c_{i_1,\cdots,i_k})$ is positive definite if 
$f^M_\triangle(\vec x)>0$ whenever $\vec x\ne\vec 0$. Now, let  $(c_{i_1,\cdots,i_k})$ be a symmetric tensor and  $(d_{i_1,\cdots,i_k})$ be a  symmetric and  positive definite tensor. Then all the classical results  on H-eigenvalues of tensors (see Qi \cite{Qi05}, Lim \cite{Lim05}, and Chang et al \cite{CPZ08,CPZ09}) can be obtained  directly  by our {\sl spectral extension theory}. 
Moreover, we  get some new relations on eigenvalues of symmetric tensors (see Proposition \ref{pro:inde-nodal-tensor} and  Theorem \ref{thm:generalized-inertia-hyper-Huang} 
for details), 
and we also apply Theorem \ref{thm:Turan-general}
 to the Turan problem (see Section \ref{sec:Turan}).
\end{example}

\vspace{0.16cm}

\textbf{Connections with inertia bounds }

The inertia bound for independence numbers is a basic result in algebraic graph theory \cite{GR01},  which appeared first in Cvetkovic's PhD thesis \cite{Cvetkovic71}. Its stronger variants have  been used to give a proof of the Sensitivity Conjecture \cite{Huang19}. We find that nodal domain theorems and inertia bounds for independence numbers   can be  absorbed into the following result. 
 Indeed, they are essentially the estimates of the size of the eigenspace of the function pair  $(f^M_\triangle,g^M_\triangle)$ which relate to the distribution of the  eigenvalues (see Section  \ref{sec:spectrum} for related concepts). 

\begin{defn}[independence number]\label{def:indep-fg}
The $\lambda$-level  independence number of the function pair  $(f,g)$ is  
$\alpha_\lambda:=\max\{\# A:f^M_\triangle(\vec x)/g^M_\triangle(\vec x)=\lambda,\forall \vec x \text{ satisfying } \mathrm{supp}(\vec x)\subset A\}$. For $\lambda=0$, the definition is independent of $g$, and then we denote the independence number of $f$ as  $\alpha_0:=\max\{\# A:f^M_\triangle(\vec x)=0,\forall \vec x \text{ satisfying } \mathrm{supp}(\vec x)\subset A\}$.
\end{defn}

For example, on a graph $(V,E)$, if we take  $f(A,B)=\#E(A,B)$, then $\alpha_0$ in Definition \ref{def:indep-fg} is the usual  independence number.

\begin{theorem}[Theorem \ref{thm:nodal-inertia-bound}]\label{thm:inertia-nodal} Given  $f,g:\power(V)^k\to\R$, denote by $\lambda_i$ the $i$-th min-max eigenvalue of the function pair  $(f^M_\triangle,g^M_\triangle)$, where we refer to Definitions \ref{def:eigenpair} and \ref{def:minmaxpair} for related concepts. 
Then we have the inertia bound  $$\alpha_\lambda\le \min\{\#\{\lambda_i\le \lambda\},\#\{\lambda_i\ge \lambda\}\}.$$
 
For any eigenvector $\vec x$ w.r.t. the eigenvalue $\lambda_k$ whose multiplicity is $r$, we have the nodal domain inequality (see Section  \ref{sec:inertia-nodal} for the related definitions) 
$$N(\vec x)\le \min\{ k+r-1, n-k+r\}.$$
\end{theorem}

Theorem \ref{thm:inertia-nodal} is the first nonlinear version of inertia bounds for the  independence number, and it also shows the first strong nodal domain inequality for general function pairs. 
One can easily apply Theorem \ref{thm:inertia-nodal} to Examples
\ref{ex:tensor} and \ref{ex:spectral-radius} to get an inertia bound on
$k$-uniform hypergraphs  (hypergraphs where each hyperedge contains exactly  $k$
  vertices):
\begin{pro}\label{pro:inertia-k-uniform}
The independence number of a $k$-uniform hypergraph $(V,E)$ is defined as $\alpha=\max\{\#U:U\subset V\text{ s.t. }U\text{ contains no hyperedge}\}$. Let $\lambda_i$ be the $i$-th minimax $H$-eigenvalue of the adjacency tensor of $(V,E)$. Then $\alpha\le\min\{\#\{\lambda_i\le 0\},\#\{\lambda_i\ge 0\}\}$. 

Moreover, for any $H$-eigenvector $\vec x$ w.r.t. $\lambda_i$ whose multiplicity is $r$, the number  of connected components of the support  of $\vec x$ is smaller than or equal to $\min\{ i+r-1, n-i+r\}$. 
\end{pro}
The  definition of H-eigenvalue and the proofs of  Propositions \ref{pro:inertia-k-uniform} and \ref{pro:inertia-p-Lap} are given in Sections  \ref{sec:tensor} and \ref{sec:p-Lap}, respectively. 
By Theorem \ref{thm:nodal-inertia-bound} (a slight variant of Theorem \ref{thm:inertia-nodal}), we have\footnote{A generalized version of Proposition \ref{pro:inertia-p-Lap} in the setting of oriented hypergraphs  is presented in  Theorem \ref{thm:p-Lap-Cheeger}.} 
\begin{pro}\label{pro:inertia-p-Lap}
For a  graph, we have the inertia bound $\alpha\le\min\{\#\{\lambda_i(\Delta_p)\le 1\},\#\{\lambda_i(\Delta_p)\ge 1\}\} $, where $\lambda_i(\Delta_p)$ is the $i$-th minimax eigenvalue of the normalized graph $p$-Laplacian.

Besides, for any eigenvector $\vec x$ w.r.t. $\lambda_i(\Delta_p)$ whose multiplicity is $r$, the number  of connected components of the support  of $\vec x$ is smaller than or equal to $\min\{ i+r-1, n-i+r\}$.
\end{pro}

\vspace{0.16cm}

\textbf{Connections with a method by Huang }

The following   eigenvalue estimate shows a nonlinear generalization of the first ingredient of Huang's proof for  the Sensitivity Conjecture \cite{Huang19}, and it  can be applied to  adjacency tensors on hypergraphs (the method proposed by Huang  works for the case of matrices, and it  is based on the Cauchy interlacing lemma, but as far as we know, there is no interlacing lemma for tensors).

\begin{theorem}[Theorems \ref{thm:FG-p-homo-signed} and \ref{thm:generalized-inertia-Huang}]\label{thm:fg-Huang}
Given $f,g:\power(V)^k\to[0,+\infty)$,   let $S(f)=\{F:  |F(\vec x)|\le f^M_\triangle(|\vec x|) ,\forall \vec x\in\R^n\}$, where $|\vec x|:=(|x_1|,\cdots,|x_n|)$ for $\vec x=(x_1,\cdots,x_n)$. Then, for any $m=1,\cdots,n$,
$$\min\limits_{U\subset V,\#U=m}\max\limits_{\;\text{ chain }A_1,\cdots,A_k\text{ in } U}\frac{f(A_1,\cdots,A_k)}{g(A_1,\cdots,A_k)}\ge \sup\limits_{F\in S(f)}\max\{\lambda_m(F),-\lambda_m'(F)\},$$
where $\lambda_m(F)$ (resp. $\lambda_m'(F)$) indicates  the $m$-th min-max (resp. max-min)  eigenvalue of the function pair $(F,g^M_\triangle(|\cdot|))$. 
\end{theorem}

The above result is stronger than the classical inertia bounds for independence numbers, and it also provides a nonlinear generalization of Huang's method.  We refer to  Sections \ref{sec:largest-eigen}
 and \ref{sec:signed-graph} for details. 
 
\vspace{0.16cm}

\textbf{Applications to $p$-Laplacians on hypergraphs}

Motivated by the total variation on hypergraphs, and its regularization  functionals \cite{TVhyper-13}, we provide a general {\sl Lov\'asz $p$-Laplacian} eigenvalue problem, and we apply it to  chemical hypergraphs (see Section \ref{sec:p-Lap}).  There is a direct way to define another $p$-Laplace operator induced by the incidence matrix of a chemical  hypergraph \cite{JMZ}, which is called the {\sl incidence  $p$-Laplacian} on  hypergraphs. 

The spectral theory for function pairs developed in Section \ref{sec:spectrum} can be applied to both the Lov\'asz $p$-Laplacian and the incidence  $p$-Laplacian, by which we have established Cheeger inequalities, inertia bounds and nodal domain properties for Lov\'asz $p$-Laplacian in Section \ref{sec:p-Lap}, and a  spectral duality theorem for incidence  $p$-Laplacian in Section \ref{sec:dual}. 

\vspace{0.16cm}

\textbf{Applications to Cheeger inequalities on simplicial complexes}

By constructing the associated signed graph for a simplicial complex, we establish $k$-way Cheeger inequalities involving the  eigenvalues of the $d$-th Hodge up-Laplacian in Section \ref{sec:simplicial-complex}.  Formally, these Cheeger inequalities  can be written as
$$C_{k,d} h_k(S_d)^2\le d+2-\lambda_{n+1-k}(\Delta^{up}_d)\le 2h_k(S_d)$$
where $h_k(S_d)$ is the so-called $k$-way Cheeger constant for $d$-simplices of a complex,  $\Delta^{up}_d$ is the normalized up Laplacian on $d$-simplices, 
 and the constant $C_{k,d}$ only depends on $k\ge 1$ and $d\ge 0$.

These Cheeger bounds for the  spectral gaps  reveal that the multiplicity of the possible eigenvalue $d+2$ equals the  number of balanced components of the associated signed graph.  We also  introduce  $p$-Laplacians on simplicial complexes, and based on the spectral theory for  function pairs developed in Sections \ref{sec:dual} and \ref{sec:structure-eigenspace}, we prove that the multiplicity of the possible eigenvalue $(d+2)^{p-1}$  equals the  number of balanced components of the associated signed graph if $p>1$, and for $p=1$,  the multiplicity of the eigenvalue 1 for the up  $1$-Laplacian is bounded by  some combinatorial quantities  involving the  balanced cliques of the associated signed graph. 
We then suggest a Cheeger constant $h(S_d)$  defined as the smallest nontrivial eigenvalue of the  1-Laplacian on the $d$-faces of  a simplicial complex, which  is positive if and only if the $d$-th reduced homology  vanishes. If the simplicial complex is  combinatorially  equivalent  to  a  uniform  triangulation  of  a $(d+1)$-dimensional,  orientable, compact, closed Riemannian manifold, we prove the Cheeger inequality $$ \frac{h^2(S_d)}{C}\le \lambda(\Delta_{d}^{up}) \le Ch(S_d), $$
in which $\lambda(\Delta_{d}^{up})$ is the smallest nontrivial eigenvalue of the $d$-th up Laplacian, and $C>1$ is a uniform constant. 
Such a Cheeger constant also closely relates to Gromov's filling profile \cite{Gromov}.  This  result may open up a new perspective on  the
 long-standing open problem regarding  Cheeger-type  inequalities on simplicial complexes \cite{DK12,GS15,GW16,PRT15,SKM14} (this open problem is a  discrete version of  Cheeger-Yau's open problem about building Cheeger-type  inequalities on differential $k$-forms \cite{Cheeger,Yau}).

\section{Spectral theory for homogeneous function pairs}
\label{sec:spectrum}

Spectral analysis has been widely used  in recent decades in numerous fields like digital image analysis, signal processing, machine learning and spectral clustering. In the linear setting, the well-known discrete Laplacian  attracts much attention \cite{TGHS20}.  Also, in smooth but nonlinear settings, there exists research on the  $p$-Laplacian eigenvalue problem and its generalized version \cite{HeinBuhler2009,LM18}. 

 For more  
 applications, some researchers turn to the 
 non-smooth setting where
  variational methods in nonlinear analysis have been proved to be very powerful.  For example, the study of  the  $1$-Laplacian eigenvalue problem $\vec 0\in \Delta_1 \vec x-\lambda  \Sgn(\vec x)$ and its signless analogue  is of great help to find better Cheeger cuts and dual Cheeger solutions \cite{Chang16,HeinBuhler2010,CSZ16,CSZ17}. In image science, many works \cite{Guy12,Guy14,Guy15,Guy16} 
 focus on the 
 eigenvalue problem in the form of $\lambda \vec u\in \nabla J(\vec u)$ where $J(\cdot)$ is convex and (absolutely) one-homogeneous, which can also be formulated as $\vec 0\in \nabla J(\vec u)-\lambda\nabla \|\vec u\|_2^2$, where $\nabla$ represents the {\sl Clarke derivative} operator.

All the above eigenvalue problems can be unified into the spectral theory for function pairs:
\begin{defn}[eigenpair]\label{def:eigenpair}
 Given a pair $(F,G)$ of two locally Lipschitz functions $F$ and $G$, we call $(\lambda,\vec x)\in \mathbb{R}\times \R^n$ an eigenpair of $(F,G)$ if \begin{equation}\label{eq:FG-eigenpair-o}\nabla F(\vec x)\cap \lambda \nabla G(\vec x)\neq\varnothing,\end{equation}
where $\vec x$ is called an eigenvector and $\lambda$ is the corresponding eigenvalue.  Using the notation of Minkowski summation, the eigenvalue problem \eqref{eq:FG-eigenpair-o} for $(F,G)$ can be written as
\begin{equation}\label{eq:FG-eigenpair}
   \vec 0\in \nabla F(\vec x)- \lambda \nabla G(\vec x). 
\end{equation}
\end{defn}

Moreover, it can be used in the variational analysis of functions on a convex
body. Given $p\ge 1$ and an $n$-dimensional convex body $P\subset
\mathbb{R}^n$ with the origin  in its interior, it is easy to show that there exists a unique $p$-homogeneous function $G:\mathbb{R}^n\to [0,\infty)$ such that $P=\{\vec x\in\mathbb{R}^n: G(\vec x)\le 1\}$ with its boundary $\partial P=G^{-1}(1)$. One way to study the variational properties of a given function $F$ on $\partial P$ is to analyse the function pair $(F,G)$ via the eigenvalue problem $\vec 0\in \nabla F(\vec x)-\lambda\nabla G(\vec x)$. In  many reasonable and valuable cases, the unit spheres of polyhedral Banach spaces (such as the polyhedrons determined by $\|\vec x\|_1=1$ or $\|\vec x \|_\infty=1$ in $\mathbb{R}^n$) attracted much attention \cite{Fonf81,Fonf90,DFH98}. 
For example, the case of $G(\vec x):=\|\vec x\|_{1,d}$ has been investigated in \cite{Chang16} and turns out to be effective in the study of Cheeger cuts and dual Cheeger problems, where $\|\cdot\|_{1,d}$ is a weighted one-norm on $\mathbb{R}^n$ (see \cite{CSZ15}).


Since the general eigenvalue problem  \eqref{eq:FG-eigenpair} is representative and useful and in view of the  lack of a  general study, in this section, we consider the spectral theory for a pair $(F,G)$ of Lipschitz functions $F$ and $G$,  
 which will be applied in the extension theory in Section  \ref{sec:extension}. 

\textbf{Unless otherwise stated, the functions $F,G:\R^n\to\R$ appearing in this section are at least} {\sl\textbf{locally Lipschitz}}. 

\begin{defn}[critical pair]
For a locally Lipschitz function $\frac FG:\R^n\to \R\cup\{\pm\infty\}$, we call $(\lambda,\vec x)\in \mathbb{R}\times \R^n$ a critical pair of $F/G$ if $$\vec 0\in \nabla \frac{F(\vec x)}{G(\vec x)} ,\;\text{ and }\lambda=\frac{F(\vec x)}{G(\vec x)},$$
where $\vec x$ is said to be a critical point and $\lambda$ is the corresponding critical value.
\end{defn}

It is known  that   $\{\text{critical points of }  F/G\}\subset \{\text{eigenvectors of }(F,G)\}$. 

\begin{defn}Given $p\in\mathbb{R}$, 
a function $F:\R^n\to \mathbb{R}$ is said to be {\sl $p$-homogeneous} if and only if
$$F(t\vec x)=t^pF(\vec x),\;\forall \vec x\in \R^n,\,\forall t>0.$$ 
\end{defn}

Let $A\subset \R^n\setminus\{ 0\}$ be a compact symmetric set, i.e., $-A=A$.  The Krasnoselskii $\mathbb{Z}_2$ {\sl genus} of $A$, denoted by $\gen(A)$, is defined to be
\begin{equation*}
\gen(A) =
\begin{cases}
\min\limits\{k\in\mathbb{Z}^+: \exists\; \text{odd continuous}\; h: A\to \mathbb{S}^{k-1}\}, & \text{if}\; A\ne\varnothing,\\
0, & \text{if}\; A=\varnothing.
\end{cases}
\end{equation*}
Let $\Gamma_k=\{ A\subset \R^n\setminus\{ 0\}: A\text{ is compact and symmetric with } \gen(A)\ge k\}$.

\begin{pro}\label{cor:LScritical}
Let $(F,G)$ be a function pair such that $F/G$ is even,  zero-homogenous and  locally Lipschitz continuous on $\R^n\setminus\{0\}$. Then, for any $k=1,2,\cdots$,
$$\lambda_k := \inf_{A\in\Gamma_k}\sup\limits_{\vec x\in A} \frac{F(\vec x)}{G(\vec x)}$$
is an eigenvalue of $(F,G)$. These eigenvalues satisfy $\lambda_1\le \lambda_2\le \cdots$, and if $\lambda=\lambda_{k+1}=\cdots=\lambda_{k+l}$ for $0\le k<k+l\le \dim X$, then $\gen(\{\text{eigenvectors w.r.t. }\lambda\})\ge l$. Similar properties hold for $\lambda_k ':= \sup\limits_{A\in\Gamma_k}\inf\limits_{\vec x\in A} \frac{F(\vec x)}{G(\vec x)}$.
\end{pro}

\begin{defn}[min-max critical pair]\label{def:minmaxpair}
Under the conditions in Proposition  \ref{cor:LScritical},
$(\lambda_k,\vec x)\in \mathbb{R}\times \R^n$ is called a {\sl min-max critical pair} if 
$\vec x$ is a critical point with the additional condition that $\vec x\in S$ for some $S\in \Gamma_k$ with $$ \sup\limits_{\vec y\in S} \frac{F(\vec y)}{G(\vec y)}=\frac{F(\vec x)}{G(\vec x)}=\lambda_k.$$
A {\sl max-min critical pair} $(\lambda_k',\vec x)$ is defined in a similar way.
\end{defn}

Some basic and important facts are:
\begin{itemize}
    \item  Critical pairs of $F/G$ are eigenpairs of $(F,G)$. 
    
Remark:   The  eigenvectors of $(F,G)$ may not be the critical points of
$F/G$, because for a critical pair, we look at $\nabla$ of the quotient,
  whereas for an eigenpair, we require a relation between  the gradients of
  the two functions involved.  
In fact, if $F$ and $G$ are smooth, then the eigenvalue problem $  \nabla F(\vec x)=\lambda \nabla G(\vec x)$ closely relates to the local bifurcation for the system of equations induced by $\nabla F$ and $\nabla G$. For example, 
the eigenvalue problem of the pair $(F,G)$ with $F(x):=\sin x$ and $G(x):=x$ is  $\cos x= \lambda$; while the nonzero critical points of $F/G$  are determined by the equation $\cos x=\sin x/x$.  

For homogeneous $F$ and $G$,    there is a  counterexample involving the 1-Laplacian   (see details in \cite{CSZ15}).

    \item  If  $F/G$ is even, then the  min-max critical pairs of $F/G$ are critical pairs of $F/G$.

   \item Assume that $F$ is $p$-homogeneous, and  $G$ is $q$-homogeneous. If $G(\vec x)\ne 0$, and  $(\lambda,\vec x)$ is an eigenpair of $(F,G)$, then  $F(\vec x)/G(\vec x)=\frac{q}{p}\lambda$.
    
    Proof: Since $\vec 0\in \nabla F(\vec x)-\lambda \nabla G(\vec x)$, there
    exists $\vec u\in \nabla G(\vec x)$ such that $\lambda \vec u\in \nabla
    F(\vec x)$. Hence, by the  Euler identity for homogeneous Lipschitz  functions, we have
$pF(\vec x)=\langle \lambda \vec u,\vec x\rangle$ and $\langle \vec u,\vec x\rangle=qG(\vec x)\ne0$. Then, there is $$\frac{F(\vec x)}{G(\vec x)}=\frac{q\langle\lambda \vec u,\vec x\rangle}{p\langle \vec u,\vec x\rangle}=\frac{q\lambda}{p}.$$
\item If $(\lambda,\vec x)$ is an eigenpair of $(F,G)$, $G(\vec y)\ne 0$, $\nabla F(\vec x)\subset \nabla F(\vec y)$ and $\nabla G(\vec x)\subset \nabla G(\vec y)$, then $(\lambda,\vec y)$ is an  eigenpair.

\begin{proof}  Since $\nabla F(\vec x)\subset \nabla F(\vec y)$ and $\nabla G(\vec x)\subset \nabla G(\vec y)$, we deduce that $\vec 0\in \nabla F(\vec x)-\lambda \nabla G(\vec x)\subset \nabla F(\vec y)-\lambda \nabla G(\vec y)$ by the properties of Minkowski summation. Consequently, $(\lambda,\vec y)$ is an eigenpair of $(F,G)$. 
   \end{proof}
 \item  For smooth $p$-homogeneous functions $F$ and $G$,  $\{\text{critical pairs of }F/G\}=\{\text{eigenpairs of }(F,G)\}$. 
    
\end{itemize}


\textbf{From now on, we further assume that $F$ and $G$ are even and  $p$-homogeneous}. In this setting, we have
\begin{equation}\label{eq:three-class-eigenpair}
\{\text{min-max critical pairs of }  F/G\}\subset \{\text{critical pairs of }  F/G\}\subset \{\text{eigenpairs of }(F,G)\}.
\end{equation}

\begin{remark}
Let $\widetilde{F/G}:\mathbb{RP}^{n-1}\to \R$ be defined by $\widetilde{F/G}([ x])=F(t x)/G(t x)$ which is independent of $t\ne 0$. Then, the critical values of $F/G$ on $\R^n\setminus\{0\}$ reduce to the critical values of  $\widetilde{F/G}$ 
on $\mathbb{RP}^{n-1}$.
\end{remark}

\begin{defn}[multiplicity]
Denote by $K_\lambda$ the set of critical points of $\frac FG$ w.r.t. the critical  value $\lambda$, $S_\lambda$ the  set of eigenvectors w.r.t. the eigenvalue $\lambda$ of $(F,G)$,  and $\{\frac FG=\lambda\}$ the level set of $\frac FG$ at the level $\lambda$.  Clearly, $K_\lambda\subset S_\lambda\subset \{\frac FG=\lambda\}$, $\forall \lambda\in\R$, and these three kinds of sets are all centrally symmetric. We use  $\gen(S_\lambda)$ (resp. $\gen(K_\lambda)$) to denote the {\sl multiplicity} of the eigenvalue (resp. critical value)  $\lambda$.

\end{defn}

\begin{pro}\label{pro:odd-homeomorphism}
For an odd smooth homeomorphism $\varphi:\R^n\to\R^n$, 
$\lambda$ is an eigenvalue of $(F\circ\varphi,G\circ\varphi)$ if and only if it is an eigenvalue of $(F,G)$, and the multiplicities of $\lambda$ for $(F\circ\varphi,G\circ\varphi)$ and $(F,G)$ coincide. 
\end{pro}

\begin{proof}
Let $(\lambda,\varphi(\vec x))$ be an eigenpair of $(F,G)$, i.e., $ \vec 0\in \nabla F(\varphi)- \lambda \nabla G(\varphi)$. Then 
$$ \vec 0\in 
J_x(\varphi)(\nabla F(\varphi)- \lambda \nabla G(\varphi))=\nabla (F\circ \varphi)(\vec x)- \lambda \nabla (G\circ \varphi)(\vec x)$$
where $J_x(\varphi)$ is the Jacobi matrix of $\varphi$ at $\vec x$. Hence, $(\lambda,\vec x)$ is an eigenpair of  $(F\circ\varphi,G\circ\varphi)$. 
Therefore, it can be verified that $S_\lambda$ is the eigenspace w.r.t. $\lambda$ of $(F,G)$ if and only if $\varphi^{-1}(S_\lambda)$  is the eigenspace w.r.t. $\lambda$ of  $(F\circ\varphi,G\circ\varphi)$. Since $\varphi$ is homeomorphism and odd (i.e., $\varphi(-\vec x)=-\varphi(\vec x)$, $\forall\vec x\in\R^n$), we have $\gen(\varphi^{-1}(S_\lambda))=\gen(S_\lambda)$. The proof is completed.
\end{proof}

Proposition \ref{pro:odd-homeomorphism} and the above statements 
could be widely applied to 
the analysis 
of homogeneous functions 
including some useful special cases,  such as Lemma 2.1 in \cite{CSZ15}, Lemma 1 in \cite{CSZ16} and Lemma 6.3 in \cite{LM18}.

One reason for us to work on a pair of $p$-homogeneous functions is the  discrete-continuous equivalence of optimization and min-max relation:  
\begin{lemma}\label{lemma:0-homo-optimal}
Let $H:X\to \R$ be a zero-homogeneous continuous function,  where  $X\subset
\R^n\setminus\{\vec0\}$ is a cone \footnote{The cone $X$ doesn't need to be
  convex, but it should satisfy the condition for a  cone, i.e., $\vec x\in X\Rightarrow t\vec x\in X$, $\forall t>0$. }.  If we further assume that $X$ is  topologically  regular, i.e., $X\subset \overline{\mathrm{int}(X)}$, where $\overline{\mathrm{int}(X)}$ is the closure of the interior of $X$, 
then
$$\inf\limits_{\vec x\in X} H(\vec x)=\inf\limits_{\vec x\in X\cap \mathbb{Z}^n} H(\vec x)\;\;\;\text{ and }\;\;\; \sup\limits_{\vec x\in X} H(\vec x)=\sup\limits_{\vec x\in X\cap \mathbb{Z}^n} H(\vec x)$$
and moreover, 
$$\inf\limits_{ A\subset X,\mathrm{cat}(A)\ge k}\sup\limits_{\vec x\in A}H(\vec x)=\inf\limits_{A\subset X,\mathrm{cat}(A)\ge k}\sup\limits_{\vec x\in \mathrm{cone}(A)\cap \mathbb{Z}^n}H(\vec x)$$
is the $k$-th min-max critical value of $H$, 
where $\mathrm{cat}(A)$ is the Lusternik–Schnirelmann category of $A$, and $\mathrm{cone}(A):=\{t\vec x:t>0,\vec x\in A\}$ is the cone hull of $A$. 
\end{lemma}
\begin{proof}
Since $X$ is topologically  regular, we have $X\subset \overline{(X\cap \mathbb{Q}^n)}$. Then, by the continuity of $H$, we have
$$\inf\limits_{\vec x\in X} H(\vec x)=\inf\limits_{\vec x\in X\cap \mathbb{Q}^n} H(\vec x)\;\;\;\text{ and }\;\;\; \sup\limits_{\vec x\in X} H(\vec x)=\sup\limits_{\vec x\in X\cap \mathbb{Q}^n} H(\vec x).$$
Note that for any $\vec x\in X\cap \mathbb{Q}^n$, there exists a positive integer $k$ such that $k\vec x\in X\cap \mathbb{Z}^n$, and by the zero-homogeneity of $H$, we have  $H(k\vec x)=H(\vec x)$. Hence, we have
$$\inf\limits_{x\in X\cap \mathbb{Q}^n} H(\vec x)=\inf\limits_{x\in X\cap \mathbb{Z}^n} H(\vec x)\;\;\;\text{ and }\;\;\; \sup\limits_{x\in X\cap \mathbb{Q}^n} H(\vec x)=\sup\limits_{x\in X\cap \mathbb{Z}^n} H(\vec x).$$

Denote by $c_k=\inf\limits_{A\in \mathrm{Cat}_k(X)}\sup\limits_{x\in A}H(\vec x)$ the $k$-th min-max critical value of $H$, where $\mathrm{Cat}_k(X)$ collects all subsets in $X$ with the Lusternik–Schnirelmann category at least $k$. For any $\epsilon>0$, there exists  $A\in \mathrm{Cat}_k(X)$ such that $\mathrm{cat}(A)\ge k$ and $\sup\limits_{x\in  A}H(\vec x)<c_k+\epsilon$, and 
there exists a neighborhood of $A$, denoted by $U_A$,  such that $\mathrm{cat}(U_A)\ge k$ and $\sup\limits_{x\in U_A}H(\vec x)<\sup\limits_{x\in  A}H(\vec x)+\epsilon<c_k+2\epsilon$. By the zero-homogeneity of $H$, we can replace $U_A$ by its cone hull $\mathrm{cone}(U_A)$, i.e.,  
$\sup\limits_{x\in \mathrm{cone}(U_A)}H(\vec x)=\sup\limits_{x\in U_A}H(\vec x)$. Thus, by the arbitrariness of $\epsilon>0$,  $c_k=\inf\limits_{\text{open cone }A\in \mathrm{Cat}_k(X)}\sup\limits_{x\in A} H(\vec x)$. The proof is completed. 
\end{proof}

\begin{remark}
We can always replace $\mathbb{Z}^n$ by any lattice $\{\sum_{i=1}^m n_i\vec v_i:n_i\in\mathbb{Z}\}$ with $\mathrm{span}(\vec v_1,\cdots,\vec v_m)=\R^n$. 
Furthermore, if both the cone $X$ and the zero-homogeneous function $H$ in Lemma \ref{lemma:0-homo-optimal} are  centrally symmetric (i.e., even), then we can replace $\mathrm{cat}(A)$ and  $\mathrm{Cat}_k(X)$ by $\gen(A)$ and  $\Gamma_k(X)$, respectively. 
\end{remark}

\begin{lemma}\label{lemma:0-homo-minmax}
Let $H:\R^n\setminus\{\vec0\}\times \R^m\setminus\{\vec0\}\to \R$ be a  continuous function which is zero-homogeneous on both components, and let    $X\subset \R^n\setminus\{\vec0\}$ and  $Y\subset \R^m\setminus\{\vec0\}$ be topologically  regular cones. 
Then 
$$\inf\limits_{ x\in X}\sup\limits_{y\in Y}H(\vec x,\vec y)=\sup\limits_{y\in Y}\inf\limits_{ x\in X}H(\vec x,\vec y)$$
if and only if
$$ \inf\limits_{ x\in X\cap \mathbb{Z}^n}\sup\limits_{y\in Y\cap \mathbb{Z}^m}H(\vec x,\vec y)=\sup\limits_{y\in Y\cap \mathbb{Z}^m}\inf\limits_{ x\in X\cap \mathbb{Z}^n}H(\vec x,\vec y).$$
\end{lemma}

\begin{proof}
We shall prove that
$$\inf\limits_{ x\in X}\sup\limits_{y\in Y}H(\vec x,\vec y)=\inf\limits_{ x\in X\cap \mathbb{Z}^n}\sup\limits_{y\in Y\cap \mathbb{Z}^m}H(\vec x,\vec y)$$
and
$$\sup\limits_{y\in Y}\inf\limits_{ x\in X}H(\vec x,\vec y)=\sup\limits_{y\in Y\cap \mathbb{Z}^m}\inf\limits_{ x\in X\cap \mathbb{Z}^n}H(\vec x,\vec y).$$
Indeed, by Lemma \ref{lemma:0-homo-optimal}, $\sup\limits_{y\in Y\cap \mathbb{Z}^n}H(\vec x,\vec y)=\sup\limits_{y\in Y}H(\vec x,\vec y)$ for any $\vec x$. Since $H(\vec x,\vec y)$ is zero-homogeneous and continuous of $\vec y$, $$\sup\limits_{y\in Y}H(\vec x,\vec y)=\sup\limits_{y\in Y\cap \mathrm{S}^{m-1}}H(\vec x,\vec y)=\max\limits_{y\in \overline{Y}\cap \mathrm{S}^{m-1}}H(\vec x,\vec y).$$ It follows from the continuity of $H$ and the compactness of $\overline{Y}\cap \mathrm{S}^{m-1}$ that $\vec x\mapsto \max\limits_{y\in \overline{Y}\cap \mathrm{S}^{m-1}}H(\vec x,\vec y)$ is continuous. We are able to apply Lemma \ref{lemma:0-homo-optimal} again to derive that $\inf\limits_{ x\in X} \max\limits_{y\in \overline{Y}\cap \mathrm{S}^{m-1}}H(\vec x,\vec y)=\inf\limits_{ x\in X\cap \mathbb{Z}^n}\max\limits_{y\in \overline{Y}\cap \mathrm{S}^{m-1}}H(\vec x,\vec y) $. Thus, 
$$\inf\limits_{ x\in X}\sup\limits_{y\in Y}H(\vec x,\vec y)=\inf\limits_{ x\in X\cap \mathbb{Z}^n}\sup\limits_{y\in Y\cap \mathbb{Z}^m}H(\vec x,\vec y)$$ is proved. The proof of   $\sup\limits_{y\in Y}\inf\limits_{ x\in X}H(\vec x,\vec y)=\sup\limits_{y\in Y\cap \mathbb{Z}^m}\inf\limits_{ x\in X\cap \mathbb{Z}^n}H(\vec x,\vec y)$ is similar. 
\end{proof}

For $p$-homogeneous functions $F$ and $G$, taking $H=F/G$ in Lemma \ref{lemma:0-homo-optimal}, we find that the continuous optimization can be transformed into a discrete optimization   restricted on  $\mathbb{Z}^n$.  If we want to replace `$\sup$' by `$\max$', some necessary  conditions should be added, and at this time, $\mathbb{Z}^n$ can be changed to a certain finite feasible set like  $\{-N,\cdots,0,1,\cdots,N\}^n$.  Based on the piecewise linear extension in  \cite{JostZhang-PL} and  the general extension theory developed in Section \ref{sec:extension}, we can further replace $\mathbb{Z}^n$  by the simplest feasible set $\{-1,0,1\}^n$ or $\{0,1\}^n$. From this viewpoint, our extension theory makes progress on the converse of Lemma \ref{lemma:0-homo-optimal}.

\subsection{Characterization of the second eigenvalue}

We show the following characterization of the second (i.e., the first non-trivial)  eigenvalue of the function pair $(F,G)$, where we don't count the  multiplicity of the first eigenvalue. 
\begin{theorem}\label{thm:smallest-nonzero} Let $F$ and $G$ be   even and $p$-homogeneous nonnegative functions on $\R^n$. Suppose that $G$ is positive and convex, and $\Pi:=\{\text{zeros of }F/G\}\cup\{\vec0\}$ is a linear subspace, as well as $F(\vec x+\vec y)=F(\vec x)$, $\forall \vec y\in\Pi$, $\forall \vec x\in\R^n$.  Then 
\begin{equation}\label{eq:lambda-min}
    \min\limits_{x\in \Pi^\bot}\max\limits_{y\in\Pi}\frac{F(\vec x+\vec y)}{G(\vec x+\vec y)}=  \min\limits_{x\in \Pi^\bot}\frac{F(\vec x)}{\min\limits_{y\in \Pi}G(\vec x+\vec y)}=\min\limits_{x:\nabla G(\vec x)\cap\Pi^\bot\ne\varnothing}\frac{F(\vec x)}{G(\vec x)} =\lambda_{\dim\Pi+1}
\end{equation}
is the second smallest 
eigenvalue of $(F,G)$. 
\end{theorem}


\begin{example}\label{exam:character-2nd-p-Lap}
For a weighted  graph $(V,W)$ with $V=\{1,\cdots,n\}$ and $W=(w_{ij})_{i,j\in V}$, let $F(\vec x)=\sum_{i,j\in V}w_{ij}|x_i-x_j|^p$ and $G(\vec x):=\|\vec x\|_p^p$. Suppose that the graph has $k$ connected components $U_1,\cdots,U_k\subset V$, and let $\Pi=\mathrm{span}\{\vec1_{U_i}:i=1,\cdots,k\}$. Then \eqref{eq:lambda-min} in Theorem \ref{thm:smallest-nonzero} reduces to
\begin{align*}
&\min\limits_{\vec x\not\in\mathrm{span}(\vec1_{U_1},\cdots,\vec1_{U_k})}\frac{\sum_{i,j}w_{ij}|x_i-x_j|^p}{\min\limits_{t_1,\cdots,t_k\in\R}\|\vec x- t_1\vec1_{U_1}-\cdots-t_k\vec1_{U_k}\|_p^p}
\\=~&\min\limits_{\vec x:\langle \nabla \|\vec x\|^p_p,\vec 1_{U_i} \rangle\ni 0,\forall i}\frac{\sum_{i,j}w_{ij}|x_i-x_j|^p}{\|\vec x\|_p^p} =\lambda_{k+1}    
\end{align*}
which is a generalization of the characterization for  the second eigenvalue
of the graph $p$-Laplacian (see Chung \cite{Chung},  Hein et al \cite{HeinBuhler2010} and  Chang \cite{Chang16}). 
\end{example}

The proof of Theorem \ref{thm:smallest-nonzero} is based on the following auxiliary proposition:
 
\begin{pro}\label{pro:GPi-property}
Given a convex function $G:\R^n\to \R$ and a linear subspace $\Pi$ of $\R^n$,  the convex function $G_\Pi$ defined by 
\begin{equation}\label{eq:def:G_Pi}
G_\Pi(\vec x):=\inf\limits_{\vec z\in\Pi}G(\vec x+\vec z)
\end{equation}
is  translation invariant along  $\Pi$, i.e., $G_\Pi(\vec x+\vec z)=G_\Pi(\vec x)$, $\forall \vec x\in \R^n$, $\forall\vec z\in\Pi$. And for any $\vec x$, 
\begin{equation}\label{eq:gradient-convex-G}
    \nabla G(\vec x)\cap \Pi^{\bot}\subset\nabla G_{\Pi}(\vec x)\ne\varnothing\;\text{ and }\;\nabla G_{\Pi}(\vec x)=\nabla G(\vec x_{\Pi})\cap \Pi^{\bot}\neq\varnothing,
\end{equation}
where $x_\Pi$ is a minimizer (if exists) of $G$ restricted on the affine plane $\vec x+\Pi$.  
Moreover, we have
\begin{equation}\label{eq:minimizer-gradient-G}
    \{\vec x\in\R^n:\nabla G(\vec x)\cap \Pi^\bot\ne\varnothing\}=\bigcup\limits_{\vec x\in\R^n}\{\text{minimizers  of }G|_{\Pi+\vec x}\}=\{\vec x\in\R^n:G(\vec x)=G_\Pi(\vec x)\}
\end{equation} 
and it is closed (but might be empty\footnote{ 
 For example, taking $G(\vec x)=e^{x_1}+e^{x_2}$, $\forall \vec x=(x_1,x_2)\in\R^2$, and $\Pi=\{(0,x_2):x_2\in\R\}$, one has  $G_\Pi(\vec x)=e^{x_1}$ and $\{\vec x:\nabla G(\vec x)\cap \Pi^\bot\ne\varnothing\}=\varnothing$. So, the set $\{\vec x:\nabla G(\vec x)\cap \Pi^\bot\ne\varnothing\}$ might be empty, but in any cases it is closed. }).
    \end{pro}
    
    \begin{example}
Let $G(\cdot):=\|\cdot\|_p^p$ and $\Pi=\mathrm{span}\{\vec y\}$ for some $\vec y\in\R^n\setminus\{\vec0\}$.  Then we get a new interpretation for the $p$-median (or $p$-mean).  Indeed, the condition $\nabla G(\vec x)\cap \Pi^\bot\ne\varnothing$ is equivalent to $\langle \nabla \|\vec x\|^p_p,\vec y \rangle\ni 0$; while the term
$\mathop{\mathrm{argmin}}\limits_{t\in\R}\|\vec x-t\vec y\|_p^p$ indicates the $p$-median along the direction $\vec y$.

In many practical situations, we set $\vec y=\vec 1$. For example,  $\mathop{\mathrm{argmin}}\limits_{t\in\R}\|\vec x-t\vec 1\|_2^2$ is the average  of $\vec x$; while $\mathop{\mathrm{argmin}}\limits_{t\in\R}\|\vec x-t\vec 1\|_1$ is the median of $\vec x$, w.r.t. a prescribed  weight  \cite{HeinBuhler2010}. 
    \end{example}

\begin{proof}
By the definition \eqref{eq:def:G_Pi}, $G_\Pi(\vec x+\vec z)=\inf\limits_{\vec z'\in\Pi}G(\vec x+\vec z+\vec z')=\inf\limits_{\vec z''\in\Pi}G(\vec x+\vec z'')=G_\Pi(\vec x)$, where $\vec z'':=\vec z+\vec z'$. This confirms  the translating invariant property. 

\vspace{0.09cm}

Convexity of $G_\Pi$: for any $\vec x,\vec y\in \R^n$, $\forall t\in[0,1]$,
\begin{align*}
tG_\Pi(\vec x)+(1-t)G_\Pi(\vec y)
&=t\inf\limits_{\vec z_1\in\Pi}G(\vec x+\vec z_1)+(1-t)\inf\limits_{\vec z_2\in\Pi}G(\vec y+\vec z_2)
\\ &\ge\inf\limits_{\vec z_1,\vec z_2\in\Pi}G(t(\vec x+\vec z_1)+(1-t)(\vec y+\vec z_2))
\\&=\inf\limits_{\vec z_1,\vec z_2\in\Pi}G(t\vec x+(1-t)\vec y+t\vec z_1+(1-t)\vec z_2))
\\ &=\inf\limits_{\vec z\in\Pi}G(t\vec x+(1-t)\vec y+\vec z) =  G_\Pi(t\vec x+(1-t)\vec y).
 \end{align*}

Closedness of $\{\vec x\in\R^n:\nabla G(\vec x)\cap \Pi^\bot\ne\varnothing\}$: Suppose $\vec x_n\to \vec x$ with $\nabla G(\vec x_n)\cap \Pi^\bot\ne\varnothing$. By the u.s.c. of $\nabla G(\cdot)$, there exist a subsequence $\{n_m\}$  and $\vec y_{n_m}\in \nabla G(\vec x_{n_m})\cap \Pi^\bot$ such that $\vec y_{n_m}\to \vec y\in \nabla G(\vec x)$. Since $\Pi^\bot$ is closed, we have $\vec y\in \Pi^\bot$. This means  $\vec y\in \nabla G(\vec x)\cap \Pi^\bot\ne\varnothing$. 

The  relation \eqref{eq:gradient-convex-G}  is a combination of the following claims:
\begin{enumerate}[{Claim} 1.]
\item $\nabla G(\vec x_\Pi)\cap \Pi^\bot\ne\varnothing$:

It is deduced by \eqref{eq:minimizer-gradient-G}, which is proved in the next part. 

\item 
$\nabla G_\Pi(\vec x)\subset \Pi^\bot$:

Note that for any $\vec  y\in \nabla G_\Pi(\vec x)$, $0=G_\Pi(\vec x+\vec z)-G_\Pi(\vec x)\ge\langle \vec y,\vec z\rangle$, $\forall \vec z\in\Pi$. This implies that $0=\langle \vec y,\vec z\rangle$, $\forall \vec z\in\Pi$, i.e., $\vec y\bot \Pi$. Hence, $\nabla G_\Pi(\vec x)\subset \Pi^\bot$.

\item $\nabla G_\Pi(\vec x)\subset \nabla G(\vec x_\Pi)$:

For $\vec y\in \nabla G_\Pi(\vec x)$, for any $\vec x'\in \R^n$, $G(\vec x')-G(\vec x_\Pi)\ge G_\Pi(\vec x')-G(\vec x_\Pi)=G_\Pi(\vec x')-G_\Pi(\vec x)\ge\langle\vec  y,\vec x'-\vec x\rangle$, which derives  $\vec y\in \nabla G(\vec x_\Pi)$. Thus, $\nabla G_\Pi(\vec x)\subset \nabla G(\vec x_\Pi)$.





\item $\nabla G(\vec x)\cap \Pi^\bot\subset \nabla G_\Pi(\vec x)$:

For any $\vec y\in \nabla G(\vec x)\cap \Pi^\bot$, $G(\vec x')-G(\vec x)\ge \langle y,\vec x'-\vec x\rangle$.  Thus, for any $\vec z,\vec z'\in \Pi$ satisfying $G(\vec x+\vec z)\le G(\vec x)$,  $G(\vec x'+\vec z')-G(\vec x+\vec z)\ge G(\vec x'+\vec z')-G(\vec x)\ge \langle \vec y,\vec x'+\vec z'-\vec x\rangle=\langle \vec y,\vec x'-\vec x\rangle$. Letting $\vec z'$ and  $\vec z$ be such that $G(\vec x+\vec z)\to G_\Pi(\vec x)$ and $G(\vec x'+\vec z')\to G_\Pi(\vec x')$, we immediately get $G_\Pi(\vec x')-G_\Pi(\vec x)\ge \langle y,\vec x'-\vec x\rangle$. Therefore, $\vec y\in \nabla G_\Pi(\vec x)$. 
\end{enumerate}


We are ready to prove \eqref{eq:minimizer-gradient-G}, that is, $$ \{\vec x:\nabla G(\vec x)\cap \Pi^\bot\ne\varnothing\}=\{\vec y:G(\vec y)=G_\Pi(\vec x)\text{ with }\vec y-\vec x\in\Pi\text{ for some }\vec x\}.$$
 
Note that $\nabla G(\vec x)\cap \Pi^\bot\ne\varnothing$ 
$\Longleftrightarrow$ $\exists \vec v\in \nabla G(\vec x)$ with $\vec v\bot\Pi$ $\Longleftrightarrow$ $\exists\vec v\in\nabla G(\vec x)$ s.t. $G(\vec x+\vec z)-G(\vec x)\ge \langle \vec v,\vec z\rangle=0$, $\forall \vec z\in \Pi$  $\Longrightarrow$ $\vec x$ is a minimizer of $G$ on $\vec x+\Pi$. 

Conversely, suppose $\vec x$ is a minimizer of $G$ restricted  on $\vec x+\Pi$. Note that
\begin{align*}
    \{\vec y\in\R^n:\langle \vec v,\vec y-\vec x\rangle=0\text{ for some }\vec v\in \nabla G(\vec x)\}&=\vec x+\bigcup\limits_{\vec v\in \nabla G(\vec x)} \vec v^\bot
    \\&=\R^n\setminus (\vec x+N_{\vec x}^{+}\cup N_{\vec x}^{-})
\end{align*} 
where $N_{\vec x}^{\pm}=\{\vec w\in\R^n:\pm\langle\vec v,\vec w\rangle>0,\,\forall \vec v\in\nabla G(\vec x)\}$ 
and $\vec v^\bot=\{\vec x\in \R^n:\langle\vec x,\vec v\rangle=0\}$ is the orthogonal complement of $\vec v$.

\textbf{Statement}. $\vec x$ is the minimizer of $G$ restricted on the closed
cone $\R^n\setminus (\vec x+ N_{\vec x}^{-})$, and  $\vec x$ is also  the local  maximizer of $G$ restricted on the  cone  $\vec x+N_{\vec x}^{-}$.

\textbf{Proof}. For any $\vec y\in \R^n\setminus (\vec x+ N_{\vec x}^{-})$, there exists $\vec v\in\nabla G(\vec x)$ such that $\langle\vec v,\vec y-\vec x\rangle\ge 0$. Hence,  $G(\vec y)-G(\vec x)\ge \langle\vec v,\vec y-\vec x\rangle\ge 0$. 

For any $\vec y\in \vec x+ N_{\vec x}^{-}$, $\langle\vec v,\vec y-\vec x\rangle<0$,  $\forall \vec v\in \nabla G(\vec x)$. By the compactness of $\nabla G(\vec x)$, there exists $\delta>0$ such that  $\langle\vec v,\vec y-\vec x\rangle<-\delta$, $\forall \vec v\in \nabla G(\vec x)$. Thus, the directional derivative along the direction $\vec y-\vec x$ at $\vec x$ is
$$
\limsup_{t\to0^+,\vec x'\to \vec x}\frac{G(\vec x'+t(\vec y-\vec x))-G(\vec x')}{t}=\max_{\vec v\in \nabla G(\vec x)}\langle\vec v, \vec y-\vec x\rangle<-\frac{\delta}{2}.$$
Hence, there exists a neighborhood $U_{\vec x}$ of $\vec x$ such that $G(\vec x'+t(\vec y-\vec x))-G(\vec x')<-\frac\delta2t$ for sufficiently small $t>0$ and $\vec x'\in U_{\vec x}$. Particularly, $G(\vec x+t(\vec y-\vec x))<G(\vec x)$ for sufficiently small $t>0$.  Thus, we complete the proof.

\vspace{0.2cm}

 By the above statement, if $\vec x$ is a minimizer of $G$ restricted  on $\vec x+\Pi$, then $\Pi\cap N_{\vec x}^{-}=\varnothing$, i.e.,  $\Pi\subset \R^n\setminus (N_{\vec x}^{+}\cup N_{\vec x}^{-})=\bigcup\limits_{\vec v\in \nabla G(\vec x)} \vec v^\bot$, which implies $\Pi\subset \vec v^\bot$ for some $\vec v\in \nabla G(\vec x)$, that is, $\nabla G(\vec x)\cap \Pi^\bot\ne\varnothing$.  
\end{proof}

\begin{proof}[Proof of Theorem \ref{thm:smallest-nonzero}]
Let
\begin{equation}
\label{eq:linking-d+1}
   \tilde{\lambda}:=\inf\limits_{x\not\in \Pi}\frac{F_{\Pi}(\vec x)}{\min\limits_{y\in \Pi}G(\vec x-\vec y)}\;\;\;\text{ and }\;\;\;  \hat{\lambda}:=\min\limits_{x:\nabla G(\vec x)\cap\Pi^\bot\ne\varnothing}\frac{F_\Pi(\vec x)}{G(\vec x)} .
\end{equation}
We shall prove that both $ \tilde{\lambda}$ and $\hat{\lambda}$  coincide with  $\lambda_{d+1}$.  Denote by $d=\dim\Pi$. 
Proposition \ref{pro:GPi-property} derives  
$\{\vec x:\nabla G(\vec x)\cap \Pi^\bot\ne\varnothing\}=\{\vec x:G(\vec x)=G_\Pi(\vec x)\}$, and thus  
$$\hat{\lambda}=\min\limits_{\vec x:G(\vec x)=G_\Pi(\vec x)}\frac{F_\Pi(\vec x)}{G(\vec x)}=\min\limits_{\vec x:G(\vec x)=G_\Pi(\vec x)}\frac{F_\Pi(\vec x)}{G_\Pi(\vec x)}\ge \inf\limits_{\vec x\not\in\Pi}\frac{F_\Pi (\vec x)}{G_\Pi(\vec x)}=\inf\limits_{\vec x\in\Pi^\bot}\frac{F_\Pi (\vec x)}{G_\Pi(\vec x)}=\tilde{\lambda}.$$ 

Since $\Pi\in \Gamma_d$ and $\frac {F(\vec x)}{G(\vec x)}=0$, $\forall\vec x\in \Pi$, we have $\lambda_1=\ldots=\lambda_d=0$. According to the local compactness of $\Pi^\bot$, the zero-homogeneity of $\frac FG$ and the fact that $F(\vec x)>0$ whenever $\vec x\in  \Pi^\bot\setminus \Pi$,  we obtain  $\tilde{\lambda}>0$. The remaining part of the proof is divided into the following steps: 
\begin{enumerate}[(I)]
    \item $\lambda_{d+1}\ge \tilde{\lambda}$:\newline
    
    It is clear that  $\dim \Pi^\bot=n-d$. 
        We first suppose that $G$ is strictly convex and $C^1$-smooth.
    Then, 
    for each $\vec x$ there is a unique  $\vec y_x\in \Pi$ such that   $G(\vec x-\vec y_x)=\min_{y\in \Pi}G(\vec x-\vec y)$
    and the map $\varphi:\vec x\mapsto \vec x-\vec y_x$ is $C^1$-smooth.  Moreover, $\varphi|_{\Pi^\bot}:\Pi^\bot\to \varphi(\Pi^\bot)$ is  bicontinuous (i.e., homeomorphism). Clearly, $\varphi$ satisfies  $-\vec x\mapsto -\vec x-\vec y_{-x}=-\vec x+\vec y_x$, which implies that $\varphi$ is odd. Hence, if we let $\vec x'$ be the projection of $\vec x$ to $\Pi^\bot$, we get an odd homeomorphism $\psi:\R^n\to \R^n$, $\vec x\mapsto \vec x-\vec y_{x'}$ which is a natural extension of $\varphi|_{\Pi^\bot}$.\newline
    Thus, by the homotopy   property of the $\mathbb{Z}_2$-genus, for any $S\in \Gamma_{d+1}$,  $\psi^{-1}(S)\in \Gamma_{d+1}$. Moreover, by the intersection property of the $\mathbb{Z}_2$-genus, $\psi^{-1}(S)\cap \Pi^\bot\ne\varnothing$, which implies $S\cap \psi(\Pi^\bot)=\psi(\psi^{-1}(S)\cap \Pi^\bot)\ne\varnothing$. Also note that $\psi(\Pi^\bot)=\varphi(\Pi^\bot)$.   Hence for any $S\in \Gamma_{d+1}$, $$\sup\limits_{x\in S}\frac{F(\vec x)}{G(\vec x)}\ge \inf\limits_{x\in \varphi(\Pi^\bot)}\frac{F(\vec x)}{G(\vec x)}=\tilde{\lambda}.$$ This proves that $\lambda_{d+1}\ge \tilde{\lambda}$. 
    
Now for general $G$ that is $p$-homogeneous and convex, take a sequence $\{G_n\}_{n\ge 1}$ of strictly convex and $C^1$-smooth $p$-homogeneous  functions that   converges to $G$.  Then by the theory of Gamma-convergence \cite{Braides02,DM14},  $\lambda_{d+1}(F,G_n)\to \lambda_{d+1}$  and $\tilde{\lambda}(F,G_n)\to \tilde{\lambda}$,  where the constants  $\lambda_{d+1}(F,G_n)$ and $\tilde{\lambda}(F,G_n)$ are the  corresponding quantities for the function pair $(F,G_n)$.

     \item $\lambda_{d+1}\le \tilde{\lambda}$:


     For any $\vec x\in \Pi^\bot\setminus \Pi$, let $\Pi':=\mathrm{span}(\Pi\cup\{\vec x\})$. Then, $\Pi'\in \Gamma_{d+1}$ and
     $$\lambda_{d+1}\le \sup\limits_{x'\in \Pi'}\frac{F(\vec x')}{G(\vec x')}=\sup\limits_{y\in \Pi}\frac{F(\vec x)}{G(\vec x+\vec y)}= \frac{F(\vec x)}{\min\limits_{y\in \Pi}G(\vec x+\vec y)}.$$ Since this holds for all $\vec x\in \Pi^\bot$, we derive that $\lambda_{d+1}\le \tilde{\lambda}$.
     \item There is no positive eigenvalue between $\lambda_1=0$ and $\lambda_{d+1}>0$:\newline

 Suppose the contrary and let $\hat{\vec x}$ be an eigenfunction corresponding
 to an eigenvalue $\lambda\in(0,\tilde{\lambda})$. Then, $\hat{\vec
   x}\not\in\Pi$, and  $0\in\nabla F(\hat{\vec x})-\lambda \nabla G(\hat{\vec
   x})$. If $G(\hat{\vec x})=G_\Pi(\hat{\vec x})$,
 then $$\lambda=\frac{F(\hat{\vec x})}{G(\hat{\vec x})}=\frac{F(\hat{\vec
     x})}{G_{\Pi}(\hat{\vec x})}\ge \inf\limits_{x\not\in\Pi}\frac{F_\Pi(\vec
   x)}{G_\Pi(\vec x)}=\tilde{\lambda},$$ which contradicts  the assumption
 that  $\lambda<\tilde{\lambda}$. So, $G(\hat{\vec x})>G_\Pi(\hat{\vec
   x})$. And thus there exists a nonzero $\vec y_{\hat{x}}\in\Pi$ satisfying $G(\hat{\vec x}-\vec y_{\hat{x}})=G_\Pi(\hat{\vec x})$.  Now, consider a flow near  $\hat{\vec x}$ defined by $\eta(\vec x,t):=\vec x-t\vec y_x$, where $t\ge 0$ and $\vec x\in  \mathbb{B}_\delta(\hat{\vec x})$ for sufficiently small $\delta>0$. Note that $$ F(\vec x-t\vec y_x)-\lambda  G(\vec x-t\vec y_x)= F(\vec x)-\lambda  G(\vec x-t\vec y_x)$$ is an increasing function of $t\in [0,1]$, since $G(\vec x-\vec y_x)<G(\vec x)$ and $G(\cdot)$ is convex.  Consequently, with the help of the theory of  weak slope \cite{Liusternik}, it is  easy to verify that   $\vec0\not\in \nabla (F(\hat{\vec x})-\lambda G(\hat{\vec x}))$, 
 which is a contradiction. This completes the proof. 
\end{enumerate}
\end{proof}

Similarly, the second eigenvalue (counting multiplicity) of $(F,G)$ has a mountain pass characterization: 
\begin{pro}\label{pro:mountain-pass}
 Let $F$ and $G$ be   even and $p$-homogeneous  functions on $\R^n$. Given the first eigenpair  $(\lambda_1,\vec x)$ of the function pair $(F,G)$,  
we have
\begin{equation}
\lambda_2=\inf\limits_{\text{curve }\gamma:[-1,1]\to \mathbb{R}^n\setminus\{0\},\gamma(\pm 1)=\pm \vec x}\sup\limits_{\vec y\in \gamma([-1,1])}\frac{F(\vec y)}{G(\vec y)}. \label{eq:mountain-pass}
\end{equation}
If $G$ is further assumed to be positive and convex, and $F$ is further assumed to be nonnegative and $F(\vec x+\vec y)=F(\vec y)$, $\forall\vec y\in\R^n$, then
$$\lambda_2=\min\limits_{y\bot x}\max\limits_{t\in\R}\frac{F(\vec y-t\vec x)}{G(\vec y-t\vec x)} 
=\min\limits_{y\bot x}\frac{F(\vec y)}{\min\limits_{t\in\R}G(\vec y-t\vec x)}.
$$
\end{pro}

 The  RatioDCA method   
introduced in \cite{HeinBuhler2010,HS11} (see also Section 3.3 in \cite{JostZhang-PL}) can be applied directly to  calculate 
the second smallest eigenvalue appearing in Theorem \ref{thm:smallest-nonzero} and Proposition \ref{pro:mountain-pass}. 
In detail,  these schemes can be rewritten in the following way:

Suppose that $F$  and $G$ satisfy the conditions in  Theorem \ref{thm:smallest-nonzero}, and we additionally assume that $F=F_1-F_2$ with $F_1$ and $F_2$ being convex and $p$-homogeneous. Then, applying the Dinkelbach-type scheme  to $\min\frac{F(\vec x)}{G_\Pi(\vec x)}$, and  by Proposition \ref{pro:GPi-property}, we have
\begin{subequations}
\label{iter1}
\begin{numcases}{}
 \tilde{\vec x}^{k+1}\in \argmin\limits_{\vec x\in \mathbb{B}} \{F_1(\vec x) -(\langle \vec u^k,\vec x\rangle+r^k \langle \vec v^k,\vec x\rangle) + H_{\vec x^k}(\vec x)\}, \label{eq:twostep_x2}
\\
\vec x^{k+1}= \tilde{\vec x}^{k+1}+\vec y^{k+1},\; \vec y^{k+1}\in \argmin\limits_{\vec y\in \Pi}G(\vec y+\tilde{\vec x}^{k+1})
\\
r^{k+1}=F( \vec x^{k+1})/G( \vec x^{k+1}),
\label{eq:twostep_r2}
\\
 \vec u^{k+1}\in\nabla F_2( \vec x^{k+1}),\;
 \vec v^{k+1}\in\nabla G( \vec x^{k+1})\cap\Pi^\bot,
\label{eq:twostep_s2}
\end{numcases}
\end{subequations}
and its modified version
\begin{subequations}
\label{iter2}
\begin{numcases}{}
\tilde{\vec x}^{k+1}\in \argmin\limits_{\vec x\in \R^n} \{F_1(\vec x) -(\langle \vec u^k,\vec x\rangle+r^k \langle \vec v^k,\vec x\rangle) + H_{\vec x^k}(\vec x)\}, \label{eq:2twostep_x2}
\\
\label{eq:2twostep_r2}
\hat{\vec x}^{k+1}= \tilde{\vec x}^{k+1}+\vec y^{k+1},\; \vec y^{k+1}\in \argmin\limits_{\vec y\in \Pi}G(\vec y+\tilde{\vec x}^{k+1})
\\
\vec x^{k+1}=\partial \mathbb{B}\cap\{t\hat{\vec x}^{k+1}:t\ge 0\} ,~~ r^{k+1}=F( \vec x^{k+1})/G( \vec x^{k+1})
\\
 \vec u^{k+1}\in\nabla F_2( \vec x^{k+1}),\;
 \vec v^{k+1}\in\nabla G( \vec x^{k+1})\cap\Pi^\bot,
\label{eq:2twostep_s2}
\end{numcases}
\end{subequations}
where $\mathbb{B}$ is the unit ball w.r.t.  a given norm. This generalizes the
inverse power method for the graph 1-Laplacian (see Algorithm 3 in \cite{HeinBuhler2010})

\subsection{Inertia bounds and a nodal domain inequality}
\label{sec:inertia-nodal}

We first provide the following technical lemma regarding  the distribution of min-max eigenvalues. Given $\lambda\in\R$, we use $\#\{\lambda_i=\lambda\}$ to denote the number of min-max eigenvalues that equals $\lambda$, i.e., $\#\{i\in\{1,\cdots,n\}:\lambda_i=\lambda\}$. Other notions such as  $\#\{\lambda_i\le\lambda\}$ and $\#\{\lambda_i\ge\lambda\}$ are defined similarly. For a centrally symmetric set $A$, 
denote by $\dim_{in} A:=\max\{\dim X:\text{linear subspace }X\subset A\cup\{\vec 0\}\}$.

\begin{lemma}\label{lemma:key-inertia-nodal} 
 For any $ \lambda\in\R$,  $\max\{\#\{\lambda_i=\lambda\}, \#\{\lambda_i'=\lambda\}\}\le \gen \{\vec x:F(\vec x)/G(\vec x) =\lambda \}$, and 
\begin{equation}\label{eq:lower-level-estimate}
\min\{\#\{\lambda_i\le \lambda\}, \#\{\lambda_i'\le\lambda\}\}\ge \dim_{in} \{\vec x:F(\vec x)/G(\vec x) \le \lambda \}    
\end{equation}  
and \eqref{eq:lower-level-estimate} still holds when we replace all `$\le \lambda$' by `$\ge \lambda$'.  In consequence, we have
\begin{equation}\label{eq:level-estimate}
\min\{\#\{\lambda_i\le \lambda\}, \#\{\lambda_i'\le\lambda\},\#\{\lambda_i\ge \lambda\}, \#\{\lambda_i'\ge\lambda\}\}\ge \dim_{in} \{\vec x:F(\vec x)/G(\vec x) = \lambda \}   . 
\end{equation}
\end{lemma}

\begin{proof}We divide the proof into several claims:
\begin{enumerate}[{Claim} 1.]
    \item $   \#\{\lambda_i=\lambda\}\le \gen \{\vec x:F(\vec x)/G(\vec x) =\lambda \}$.
    
 Proof:   It follows from the relation \eqref{eq:three-class-eigenpair} and Proposition \ref{cor:LScritical}   that for any $\lambda\in\R$,
\begin{equation}\label{eq:genus=}
    \#\{\lambda_i=\lambda\}\le \gen (K_\lambda) \le \gen (S_\lambda) \le \gen \{\vec x:F(\vec x)=\lambda G(\vec x)\}.
\end{equation}

\item $ \#\{\lambda_i\le \lambda\}= \gen\{\vec x:F(\vec x)/G(\vec x)\le \lambda \}$. 

Proof: Let $k=\gen\{\vec x:F(\vec x)/G(\vec x)\le \lambda \}$. Then taking $A_0=\{\vec x:F(\vec x)/G(\vec x)\le \lambda \}$, we have $$\lambda_k = \inf_{A\in\Gamma_k}\sup\limits_{\vec x\in A} \frac{F(\vec x)}{G(\vec x)}\le \sup\limits_{\vec x\in A_0} \frac{F(\vec x)}{G(\vec x)}\le\lambda,$$which implies $ \#\{\lambda_i\le \lambda\}\ge k=\gen\{\vec x:F(\vec x)/G(\vec x)\le \lambda \}$. 
The inverse inequality  is also true (see \cite{PereraAgarwalO'Regan}). 

\item  
$ \#\{\lambda_i\ge \lambda\}\ge \gen'\{\vec x:F(\vec x)/G(\vec x)\ge \lambda \}$ where
$$\gen'(A)=\max\{k:A'\subset A,\mathrm{cone}(A')\cap \mathbb{S}^{n-1}\mathop{\cong}\limits^{odd}\mathbb{S}^{k-1} \}$$
where $\mathrm{cone}(A')\cap \mathbb{S}^{n-1}\mathop{\cong}\limits^{odd}\mathbb{S}^{k-1}$ means that there is an odd homeomorphism between $\mathrm{cone}(A')\cap \mathbb{S}^{n-1}$ and $\mathbb{S}^{k-1}$.

Proof: Suppose  $\gen'\{\vec x:F(\vec x)/G(\vec x)\ge \lambda \}= n-k+1$ for some $k\in\{1,\cdots,n\}$. Then there exist $A'\subset \{\vec x:F(\vec x)/G(\vec x)\ge \lambda \}$ and an odd homeomorphism $\psi:A'\to \mathbb{S}^{n-k}$. The  intersection property of $\mathbb{Z}_2$-genus implies that $A\cap \mathrm{cone}(A')\ne  \varnothing$ for any $A\in \Gamma_k$. Therefore, 
$$\lambda_k = \inf_{A\in\Gamma_k}\sup\limits_{\vec x\in A} \frac{F(\vec x)}{G(\vec x)}\ge \inf\limits_{\vec x\in \mathrm{cone}(A')} \frac{F(\vec x)}{G(\vec x)}\ge\lambda,$$
and this yields $\#\{\lambda_i\ge \lambda\}\ge n-k+1=\gen'\{\vec x:F(\vec x)/G(\vec x)\ge \lambda \}$.
\item $ \#\{\lambda_i'=  \lambda\}\le \gen\{\vec x:F(\vec x)/G(\vec x)= \lambda \}$, $ \#\{\lambda_i'\le \lambda\}\ge \gen'\{\vec x:F(\vec x)/G(\vec x)\le \lambda \}$ and $ \#\{\lambda_i'\ge \lambda\}= \gen\{\vec x:F(\vec x)/G(\vec x)\ge \lambda \}$.

We omit the proof because  it is similar to   Claims 1, 2 and 3.
\end{enumerate}

Note that $\gen(A)\ge \gen'(A)\ge \dim_{in}(A)$. In consequence, \eqref{eq:lower-level-estimate} holds, and thus  \eqref{eq:level-estimate} can be verified directly.

\end{proof}


Based on Lemma \ref{lemma:key-inertia-nodal}, we can get the inertia bound of the independence number, and the nodal domain estimate of an eigenvector.

\begin{defn}[nodal domains] \label{def:nice-nodal-domain}
Given $(F,G)$ and $\vec x$, a family of {\sl up nodal domains} of $\vec x$  w.r.t. $(F,G)$ consists of $k$ pairwise  disjoint nonempty subsets $U_1,\cdots,U_k$
 of the support  $\supp(\vec x)$   satisfying 
 \begin{equation}\label{eq:def-nodal}
 \frac{F(\sum_{i=1}^kt_i\vec x|_{U_i})}{G(\sum_{i=1}^kt_i\vec x|_{U_i})}\ge \frac{F(\vec x)}{G(\vec x)},\,\forall t_1,\cdots,t_k\in\R.   
 \end{equation}
 Similarly, we can define the {\sl down nodal domain} by instead    `$\ge$' in \eqref{eq:def-nodal} of `$\le$'. We call $\{U_i\}_{i=1}^k$ the family of {\sl  nodal domains} if   `$\ge$' in \eqref{eq:def-nodal} is replaced by `$=$`.

Denote by $N^+(\vec x)$ (resp. $N^-(\vec x)$) the largest possible $k$ such that there exists a family of $k$ up (resp. down) nodal domains of $\vec x$ w.r.t. $(F,G)$. And let $N(\vec x)$ be the number of nodal domains of $\vec x$.
\end{defn}

\begin{defn}[independence number]
The $c$-level  independence number of $(F,G)$ is  $\alpha_c:=\max\{k:\exists \text{pairwise disjoint }U_1,\cdots,U_k\subset\{1,\cdots,n\}\text{ s.t. }F(\vec x)/G(\vec x)=c,\forall \vec x\in \mathrm{span}(\vec1_{U_1},\cdots,\vec1_{U_k})\}$.  
\end{defn}

\begin{example}
For a simple graph determined by its adjacency matrix $A$, let $F(\vec x)=\vec x^\top A\vec x$ and $G(\vec x)=\vec x^\top\vec x$. Then one can check that  $\alpha_0$ is the usual independence number of the graph. 
\end{example}

\begin{example}For a graph $(V,E)$, taking  $f(A)=|\partial A|$ and $g(A)=\vol(A)$, considering the function pair $(f^L,g^L)$, then it is interesting that   $\alpha_0$ indicates the number of  connected components; while  $\alpha_1$ is the standard independence number.  \end{example}

\begin{theorem}\label{thm:nodal-inertia-bound} Let $F$ and $G$ be  even $p$-homogeneous Lipschitz functions on $\R^n$. Then we have the inertia bound  $$\alpha_c\le \min\{\#\{\lambda_i\le c\},\#\{\lambda_i\ge c\}\}.$$ 
For any eigenvector $\vec x$ w.r.t. the eigenvalue $\lambda_k$ whose multiplicity is $r$, we have 
$$N^-(\vec x)\le k+r-1 \text{ and } N^+(\vec x)\le n-k+r.$$ 
\end{theorem}

\begin{proof}
By \eqref{eq:level-estimate} in Lemma \ref{lemma:key-inertia-nodal}, $\dim_{in} \{\vec x:F(\vec x)/G(\vec x) = c \}\le \min\{\#\{\lambda_i\le c\}, \#\{\lambda_i\ge c\}\}$. And by the definition of independence number, $\alpha_c\le \dim_{in} \{\vec x:F(\vec x)/G(\vec x)=c\}$. Thus, the inertia bound is proved.  Since  $(\lambda_k,\vec x)$ is an  eigenpair of $(F,G)$,  it follows from the definition of nodal domain that
$$ N^+(\vec x)\le \dim_{in} \{\vec x:F(\vec x)/G(\vec x) \ge\lambda_k \}\le \#\{\lambda_i\ge \lambda_k\}\le n-k+r, $$
$$ N^-(\vec x)\le \dim_{in} \{\vec x:F(\vec x)/G(\vec x) \le\lambda_k \}\le \#\{\lambda_i\le \lambda_k\}\le k+r-1.$$

Hence, the nodal domain inequality is proved. 
\end{proof}

\subsection{
Collatz-Wielandt formula for the largest eigenvalue}\label{sec:largest-eigen}

As a generalization of Collatz-Wielandt formula, we give a min-max characterization for the maximal eigenvalue of $(F,G)$:  
 \begin{lemma}\label{lemma:Collatz-Wielandt}
Let $F$ and $G$ be $p$-homogeneous Lipschitz functions  such that  $\max\limits_{x\in\R^n}\frac{F(\vec x)}{G(\vec x)}$  achieves its maximum at some $\vec x\in\R^n_{\ge0}\setminus\{\vec0\}$.  Then
$$\sup\limits_{x\in\R^n_+}\inf\limits_{y\in\R^n_+}\sup\frac{\langle\vec y,\nabla F(\vec x)\rangle}{\langle\vec y,\nabla G(\vec x)\rangle}:=\sup\limits_{x\in\R^n_+}\inf\limits_{y\in\R^n_+}\sup\left\{\frac{\langle\vec y,\vec u\rangle}{\langle\vec y,\vec v\rangle}:\vec u\in\nabla F(\vec x),\vec v\in\nabla G(\vec x)\right\}$$
is the maximum of $F/G$, and it is also the maximal eigenvalue of $(F,G)$.
\end{lemma}
 
 \begin{proof}
 Since $F/G$ achieves its maximum at some $\vec x$,  we have $\nabla \frac{F(\vec x)}{G(\vec x)}\ni 0$ and thus $\nabla F(\vec x)-\lambda \nabla G(\vec x)\ni \vec 0$ with $\lambda= \frac{F(\vec x)}{G(\vec x)}$ being the maximum of $F/G$ on $\R^n$. 
 This implies that there exist $\vec u\in\nabla F(\vec x),\vec v\in\nabla G(\vec x)$ such that $\vec u=\lambda\vec v$, and thus 
$$\sup\limits_{z\in\R^n_+}\inf\limits_{y\in\R^n_+}\sup\frac{\langle\vec y,\nabla F(\vec z)\rangle}{\langle\vec y,\nabla G(\vec z)\rangle}  \ge  \inf\limits_{y\in\R^n_+}\sup\frac{\langle\vec y,\nabla F(\vec x)\rangle}{\langle\vec y,\nabla G(\vec x)\rangle}\ge \inf\limits_{y\in\R^n_+}\frac{\langle\vec y,\vec u\rangle}{\langle\vec y,\vec v\rangle}=\lambda. $$

On the other hand, since $F$ and $G$ are $p$-homogeneous, by the Euler identity $\langle\vec x,\nabla F(\vec x)\rangle=pF(\vec x)$, we have 
$$\sup\limits_{x\in\R^n_+}\inf\limits_{y\in\R^n_+}\sup\frac{\langle\vec y,\nabla F(\vec x)\rangle}{\langle\vec y,\nabla G(\vec x)\rangle}\le \sup\limits_{x\in\R^n_+}\sup\frac{\langle\vec x,\nabla F(\vec x)\rangle}{\langle\vec x,\nabla G(\vec x)\rangle}=\sup\limits_{x\in\R^n_+}\frac{pF(\vec x)}{pG(\vec x)}=\lambda.$$
The proof is completed.
 \end{proof}
 
  The condition of Lemma \ref{lemma:Collatz-Wielandt} is  satisfied in most of the interesting cases. For example, if $F(|\vec x|)\ge F(\vec x)$ and $G(|\vec x|)= G(\vec x)>0$ for any $\vec x=(x_1,\cdots,x_n)\in\R^n\setminus\{\vec0\}$, where 
  $|\vec x|:=(|x_1|,\cdots,|x_n|)$, then  $\max\limits_{x\in\R^n}\frac{F(\vec x)}{G(\vec x)}$ can achieve its maximum at some $\vec x\in\R^n_{\ge 0}\setminus\{\vec0\}$.  

\begin{example}
  Let $k$ be a positive even number, and let 
  $$F(\vec x)=\sum\limits_{i_1,\cdots,i_k=1}^nc_{i_1,\cdots,i_k}x_{i_1}\cdots x_{i_k}$$ and $G(\vec x)=\sum\limits_{i_1,\cdots,i_k=1}^nd_{i_1,\cdots,i_k}x_{i_1}\cdots x_{i_k}$ such that every monomial term of the polynomial $G(\vec x)$ is the square of some monomial, and $G(\vec x)>0$ whenever $\vec x\ne\vec 0$, where $c_{i_1,\cdots,i_k}\ge0$ and  $d_{i_1,\cdots,i_k}\ge0$. Then 
  $$\sup\limits_{x\in\R^n_+}\inf\limits_{y\in\R^n_+}\frac{\langle\vec y,C\vec x^{k-1}\rangle}{\langle\vec y,D\vec x^{k-1}\rangle}=\sup\limits_{x\in\R^n_+}\min\limits_i\frac{(C\vec x^{k-1})_i}{(D\vec x^{k-1})_i}$$ is the maximal H-eigenvalue of the tensor pair $(C,D)$, where $C=(c_{i_1,\cdots,i_k})$ and $D=(d_{i_1,\cdots,i_k})$ (see Section \ref{sec:tensor} for the definitions).  This gives a Collatz-Wielandt  formula for positive tensors.
\end{example}

As far as we know, all known 
generalizations of the Collatz-Wielandt  formula are about  
homogeneous
single valued maps \cite{GTH19}.  Lemma \ref{lemma:Collatz-Wielandt} might be the first version for  $(k-1)$-homogeneous set-valued maps $\nabla F$ and $\nabla G$.

Next we show a spectral lower bound for the largest eigenvalue of $(F,G)$  restricted on a subspace. 

\begin{theorem}\label{thm:FG-p-homo-signed}
Given $p$-homogeneous functions  $F,G:\R^n\to[0,+\infty)$, and a linear subspace $X\subset \R^n$,  we introduce the  set of pairs of functions 
$$S(F,G;X)=\{(F',G'):|F'(\vec x)|\le F(|\vec x|) \text{ and }G'(\vec x)=G(|\vec x|) ,\forall \vec x\in X\},$$ where\footnote{By $|\cdot|:\R\to\R$, we mean the absolute value. } $|\cdot|:\R^n\to\R^n$ is a map such that $\dim|X|=\dim X=m$, and $|X|:=\mathrm{span}\{|\vec x|:\vec x\in X\}$.   Then
$$\lambda_{\max}(F,G)|_{|X|}\ge \sup\limits_{(F',G')\in S(F,G;X)}\max\{\lambda_m(F',G'),-\lambda_m'(F',G')\}.$$
\end{theorem}

\begin{proof}
For any $(F',G')\in S(F,G;X)$, $\frac{F(|\vec x|)}{G(|\vec x|)}\ge\frac{F'(\vec x)}{G'(\vec x)} $ and thus 
$$\lambda_{\max}(F,G)|_{|X|}:=\max\limits_{y\in |X|}\frac{F(\vec y)}{G(\vec y)}=\max\limits_{ y\in \mathrm{span}\{|\vec x|:\vec x\in X\}}\frac{F(\vec y)}{G(\vec y)}\ge\max\limits_{ x\in X} \frac{F(|\vec x|)}{G(|\vec x|)}\ge \max\limits_{ x\in X}\frac{F'(\vec x)}{G'(\vec x)} .$$
Since $\dim X=\dim |X|=m$, 
$$\lambda_{\max}(F,G)|_{|X|}\ge\max\limits_{ x\in X}\frac{F'(\vec x)}{G'(\vec x)}\ge \inf\limits_{\gen(A)\ge m}\max\limits_{ x\in A}\frac{F'(\vec x)}{G'(\vec x)}= \lambda_m(F',G').$$ 
Changing $F'$ to $-F'$, we also have \begin{align*}
\lambda_{\max}(F,G)|_{|X|}&\ge\max\limits_{ x\in X}\frac{-F'(\vec x)}{G'(\vec x)}\ge 
\inf\limits_{\gen(A)\ge m}\sup\limits_{ x\in A}\frac{-F'(\vec x)}{G'(\vec x)} \\&=-\sup\limits_{\gen(A)\ge m}\inf\limits_{ x\in A}\frac{F'(\vec x)}{G'(\vec x)} =-\lambda_m'(F',G').
\end{align*} 
The proof is completed. 
\end{proof}



\subsection{Structure of eigenspaces}\label{sec:structure-eigenspace}

The eigenspace of an eigenvalue $\lambda$ is the collection of all eigenvectors w.r.t. $\lambda$. We list below some useful  observations:
\begin{itemize}
\item If both $F$ and $G$ are even, then each eigenspace is centrally symmetric w.r.t. the center $\vec 0$.
\item If both $F$ and $G$ are $p$-homogeneous,  then each eigenspace is a cone. 
\item If both $F$ and $G$ are piecewise linear,  then each eigenspace is piecewise linear.
\item However, there exist convex functions $F$ and $G$ such that not every eigenspace is convex.
\end{itemize}

Below, we show further results on piecewise linear function pairs and convex function pairs, which will be used in the analysis of several typical applications in Section \ref{sec:application}.


\vspace{0.16cm}

Let $F$ and $G$ be continuous and piecewise linear functions on $\R^n$. We always assume that there are only finite pieces, i.e., there are convex polyhedral domains $\Omega_1,\cdots,\Omega_k$ with $\cup_{i=1}^k\Omega_i=\R^n$ and  $\Omega_i^o\cap \Omega_j^o=\varnothing$ $\forall i\ne j$, such that both $F$ and $G$ are linear restricted on each $\Omega_i$, $\forall i$. The extreme set of $(F,G)$ is defined to be union of all extreme points of $\Omega_i$ for $i=1,\cdots,k$. 

\begin{theorem}\label{piecewise-linear-vertex}
There are finitely many eigenvalues of $(F,G)$, and every eigenvalue has an eigenvector in the extreme set.
\end{theorem}
\begin{proof}
We may assume that $F|_{\Omega_i}(\vec x)=\langle \vec a_i,\vec x\rangle+ c_i$ and $G|_{\Omega_i}(\vec x)=\langle \vec b_i,\vec x\rangle+ c_i'$, $i=1,\cdots,k$. If $\vec x$ is a relative interior point of $\cap_{i\in I}\overline{\Omega}_i$ for some index set $I\subset\{1,\cdots,k\}$, then by the properties of subderivative, we have  $\nabla F(\vec x)=\mathrm{conv}\{\vec a_i:i\in I\}$ and $\nabla G(\vec x)=\mathrm{conv}\{\vec b_i:i\in I\}$.

For any eigenpair $(\lambda,\vec x)$, we suppose $\vec x\in\Omega_i$. Then, $\nabla F(\vec x)\subset\nabla F(\vec v)$ and $\nabla G(\vec x)\subset\nabla G(\vec v)$, where   $\vec v$ is a vertex of $\Omega_i$. Therefore, 
$$\vec0\in\nabla F(\vec x)-\lambda\nabla G(\vec x)\subset \nabla F(\vec v)-\lambda \nabla G(\vec v)$$
implying that $(\lambda,\vec v)$ is also an eigenpair. The proof is completed.
\end{proof}

Note that $\{\Omega_i\}_{i=1}^k$ gives  $\R^n$ the structure of a complex. And if we regard $\R^n$ as the polyhedral  complex $\cup_{i=1}^k\Omega_i$, every eigenspace should be a  subcomplex.

\vspace{0.16cm}


Next, we show a result involving the subgradient  of a convex function on an inner product space $X$, which is a useful 
tool on convexity:

\begin{pro}\label{pro:convex-property}
Let $F:X\to\R$ be a convex function. Then, given $\vec x,\vec y\in X$, the following statements are equivalent:
\begin{enumerate}
\item[(1)] $\nabla F(\vec x)\cap \nabla F(\vec y)=\nabla F(t\vec x+(1-t)\vec y)$  whenever $0< t<1$;
\item[(1')] $\nabla F(\vec x)\cap \nabla F(\vec y)=\nabla F(t\vec x+(1-t)\vec y)$  for some $0<t<1$;
\item[(2)] $\nabla F(\vec x)\cap \nabla F(\vec y)\ne \varnothing$;
\item[(3)] $tF(\vec x)+(1-t)F(\vec y)=F(t\vec x+(1-t)\vec y)$ whenever $0< t<1$;
\item[(3')] $tF(\vec x)+(1-t)F(\vec y)=F(t\vec x+(1-t)\vec y)$ for some $0< t<1$.
\end{enumerate}

   \end{pro}

\begin{proof}
The proof is organized in the following steps: 
\begin{enumerate}
    \item[Claim 1.] Denote by  $N_p:=\{\vec v:\langle \vec v,\vec p'-\vec p\rangle\le 0,\forall \vec p'\in \Omega\}$ the normal cone of a convex set $\Omega$ at $\vec p$. Then, for any $ \vec v\in N_p\cap N_q$, there hold $\vec v\bot (\vec p-\vec q)$, and $N_{tp+(1-t)q}=  N_p\cap N_q$ whenever $0<t<1$.
    
\begin{proof}
For any $ \vec v\in N_p\cap N_q$, $\langle \vec v,\vec r-\vec p \rangle\le 0$  and  $\langle \vec v,\vec r-\vec q \rangle\le 0$, $\forall \vec r\in \Omega$. Thus, $\langle \vec v,\vec r-t\vec p-(1-t)\vec q \rangle=t\langle \vec v,\vec r-\vec p \rangle+(1-t)\langle \vec v,\vec r-\vec q \rangle\le 0$ for any $\vec r\in \Omega$,  which means $\vec v\in N_{tp+(1-t)q}$. So, $N_p\cap N_q\subset N_{tp+(1-t)q}$.

Conversely, for every $\vec v\in N_{tp+(1-t)q}$, $\langle \vec v,t\vec r+(1-t)\vec r'-t\vec p-(1-t)\vec q \rangle\le 0$, $\forall \vec r,\vec r'\in \Omega$. Since $0<t<1$,  taking $\vec r'=\vec q$, we have $\langle \vec v,\vec r-\vec p\rangle=\frac1t\langle \vec v,t\vec r-t\vec p\rangle\le 0$, $\forall \vec r\in \Omega$. Hence, $\vec v\in N_p$. Similarly, $\vec v\in N_q$. Thus, $N_{tp+(1-t)q}\subset   N_p\cap N_q$.

Taking $\vec r=\vec p$, we have $\langle \vec v,\vec p-\vec q \rangle\le 0$. Similarly, $\langle \vec v,\vec q-\vec p \rangle\le 0$. Hence, $\langle \vec v,\vec q-\vec p \rangle= 0$.
\end{proof}

 \item[Claim 2.] 
$\nabla F(\vec x)=\mathrm{Proj}_{X}(N_p(\mathbf{epi}(F))\cap (X\times\{-1\})$, where $\mathbf{epi}(F):=\{(\vec x,c)\in X\times\R:\vec x\in X,c\ge F(\vec x)\}$,  $p:=(\vec x,F(\vec x))$, and $\mathrm{Proj}_{X}:X\times\R\to X$ is the  projection onto $X$. 

\begin{proof}
Let $\vec p=(\vec x,F(\vec x))$ and $\Omega=\mathbf{epi}(F)$.  Since  $\nabla F(\vec x)=\{\vec u:\langle\vec u,\vec x'-\vec x\rangle\le F(\vec x')-F(\vec x)\}$, we have\footnote{This fact is known to the experts \cite{Clarke}. }
\begin{align*}
N_p&=\{\vec v=(\vec v',\xi)\in X\times\R:\langle \vec v,\vec p'-\vec p\rangle\le 0,\forall \vec p'=(\vec x',c)\in\mathbf{epi}(F) \}
\\&= \{(\vec v',\xi)\in X\times(-\infty,0):\langle \vec v',\vec x'-\vec x\rangle+\xi(c-F(\vec x))\le 0,\forall \vec x'\in X,\forall c\ge F(\vec x')\}\cup\{\vec0\}
\\&=\{(\vec v',\xi)\in X\times(-\infty,0):\langle  \frac{\vec v'}{|\xi|},\vec x'-\vec x\rangle\le F(\vec x')-F(\vec x),\forall \vec x'\in X\}\cup\{\vec0\}
\\&=\{t(\vec u,-1):\vec u\in \nabla F(\vec x),t\ge 0\}.
\end{align*}
The  equality $N_p(\mathbf{epi}(F))=\{t(\vec u,-1):\vec u\in \nabla F(\vec x),t\ge 0\}$ then implies Claim 2.
\end{proof}    
 

\item[Claim 3.]  $\nabla F(\vec x)\cap \nabla F(\vec y)=\begin{cases}
\varnothing, &\text{ if }tF(\vec x)+(1-t)F(\vec y)>F(t\vec x+(1-t)\vec y),
\\ \nabla F(t\vec x+(1-t)\vec y)&\text{ if }tF(\vec x)+(1-t)F(\vec y)=F(t\vec x+(1-t)\vec y).
\end{cases}
$ 

\begin{proof}
According to Claims 1 and 2, we have
\begin{align*}
&\nabla F(\vec x)\cap \nabla F(\vec y)\\=~&\mathrm{Proj}_{X}(N_p(\mathbf{epi}(F))\cap N_q(\mathbf{epi}(F))\cap (X\times\{-1\}) \\=~& \mathrm{Proj}_{X}(N_{tp+(1-t)q}(\mathbf{epi}(F))\cap (X\times\{-1\}) 
\\=~&\begin{cases}
\varnothing, &\text{ if }tp+(1-t)p\ne(t\vec x+(1-t)\vec y,F(t\vec x+(1-t)\vec y)),
\\\nabla F(t\vec x+(1-t)\vec y)&\text{ if }tp+(1-t)p=(t\vec x+(1-t)\vec y,F(t\vec x+(1-t)\vec y)).
\end{cases}
\end{align*}
\end{proof}
\end{enumerate}
The proof is completed. 
\end{proof}


We are also  interested in  the converse of Proposition \ref{pro:convex-property}. 
 
\begin{Conj}
A Lipschitzian function $F:X\to\R$ is convex if and only if $\nabla F(\vec x)\cap \nabla F(\vec y)=\nabla F(t\vec x+(1-t)\vec y)$  whenever $0\le t\le1$ and $\nabla F(\vec x)\cap \nabla F(\vec y)\ne \varnothing$. 
\end{Conj}

In the proof of Proposition \ref{pro:convex-property}, we use the section  $N_p(\mathbf{epi}(F))\cap (X\times\{-1\})$, while we note that in  \cite{MartinezPintea20}, the authors investigate the spherical  section $N_p(\mathbf{epi}(F))\cap \mathbb{S}^{\dim X}$, i.e., the Gauss map of $\mathbf{graph}(F)$ at $\vec p$. 
By Proposition \ref{pro:convex-property} and the results in  \cite{MartinezPintea20}, we have
\begin{pro}
If $\dim X<\infty$,  then the range of the Gauss map of the graph of  a convex function $F:X\to\R$ is open if and only if $\{\vec y\in X:\nabla F(\vec y)=\nabla F(\vec x)\}$ is bounded for any $\vec x\in X$,   if and only if every convex subset $\Omega$ with  $F|_\Omega$ being  linear is bounded. 
\end{pro}

\begin{pro}
For a convex function $F:X\to\R$ with $\dim X<\infty$, if $F$ is one-homogeneous, then the range of the Gauss map of $\mathbf{graph}(F)$ is closed; while, if $F$ is $p$-homogeneous with $p>1$, and $F(\vec x)>0$ whenever $\vec x\ne\vec 0$, then  the range of the Gauss map of $\mathbf{graph}(F)$ is open. 
\end{pro}

Based on Proposition \ref{pro:convex-property}, we obtain the following results on eigenpairs, which are very similar to Theorem \ref{piecewise-linear-vertex}.

\begin{cor}
For two convex functions $F$ and $G$, if 
$F(t\vec x+(1-t)\vec y)=tF(\vec x)+(1-t)F(\vec y)$ and $G(t\vec x+(1-t)\vec y)=tG(\vec x)+(1-t)G(\vec y)$  and 
$(\lambda,t\vec x+(1-t)\vec y)$ is an eigenpair of $(F,G)$ for some $0<t<1$, then both $(\lambda,\vec x)$ and $(\lambda,\vec y)$ are eigenpairs. \end{cor}   

\begin{cor}\label{cor:vertex-critical}
For two convex functions $F$ and $G$, if 
$F$ and $G$ are linear on a  convex polyhedron  $\triangle$,  and 
$(\lambda,\vec x)$ is an eigenpair of $(F,G)$ for some relative interior point $\vec x$ in $ \triangle$, then  for any vertex $\vec v$ of $\triangle$,  $(\lambda,\vec v)$ is also an eigenpair. \end{cor}  

\subsection{Duality and convex conjugate}
\label{sec:dual}
\begin{defn}[conjugate]
The convex conjugate of a convex function $F:\R^n\to\R$ is defined as $F^\star(\vec x)=\sup\limits_{y\in\R^n}\{\langle \vec x,\vec y\rangle-F(\vec y)\}$, $\forall \vec x\in\R^n$. 
\end{defn}
The convex conjugate is also known as Legendre  transformation or Fenchel dual
(see \cite{Zeidler}). If we restrict ourselves to a  convex $p$-homogeneous
function $F:\R^n\to\R$ with the additional positive-definiteness condition
that $F(\vec x)>0$ whenever $\vec x\ne\vec0$, then  $F^\star$ is  convex,
$p^*$-homogeneous and positive-definite, where $p,p^*>1$ satisfy
$\frac1p+\frac{1}{p^*}=1$. It should be noted that the convex conjugate is
useless  for  the one-homogeneous  case. For this case, we introduce the concept of convex duality as follows.
\begin{defn}[duality]
For a convex one-homogeneous function $G:\R^n\to[0,+\infty)$ with  $G(\vec x)>0$ whenever $\vec x\ne\vec0$, we define its dual function $G^*(\vec x)=\sup\limits_{y\ne 0}\frac{\langle \vec x,\vec y\rangle}{G(y)}$. 
\end{defn}
 It is clear that $G^*$ is also a convex one-homogeneous function with the positive-definiteness  property that  $G^*(\vec x)>0$ whenever $\vec x\ne\vec0$. 
 
\begin{defn}[projection]
For a function $G:\R^n\to\R$, and a linear map $T:\R^n\to\R^m$ with $\mathrm{Range}(T)\ne\vec0$, define the function  $G_{\inf}:\mathrm{Range}(T)\to\R$  by $G_{\inf}(\vec y):=\inf\limits_{x\in T^{-1}(y)}G(\vec x)$, and define $G_{\mathrm{Ker}(T)}:\R^n\to\R$ by  $G_{\mathrm{Ker}(T)}(\vec x):=\inf\limits_{\vec z\in \mathrm{Ker}(T)}G(\vec x+\vec z)$. We call $G_{\inf}$ the projection of $G$ to $\mathrm{Range}(T)$, and $G_{\mathrm{Ker}(T)}$ the projection of $G$ to $\mathrm{Ker}(T)^\bot$. 
\end{defn}

\begin{remark}
When we think of $G$ as a norm on $\R^n$, then  $G_{\inf}$ is a norm on $\mathrm{Range}(T)$, and $T$ maps the unit ball in $\R^n$ with the norm  $G$ to the unit ball in $\mathrm{Range}(T)$ equipped with the norm $G_{\inf}$. We note that $G_{\inf}$ is called the filling norm in \cite{Gromov} when $G$ is a norm.  Likewise, $G_{\mathrm{Ker}(T)}$ induces a norm on $\mathrm{Ker}(T)^\bot$, and the unit ball in $\mathrm{Ker}(T)^\bot$ that has  the norm $G_{\mathrm{Ker}(T)}$ is the projection of the  $G$-norm unit ball 
to $\mathrm{Ker}(T)^\bot$, in which the  $G$-norm unit ball means the unit ball in $\R^n$ under the  norm $G$.  This is the reason why we call $G_{\inf}$ and $G_{\mathrm{Ker}(T)}$ the projections of $G$. 
\end{remark}

We have the following useful  lemma which reveals the  connections among the non-vanishing eigenvalues of 
function pairs 
involving duality and  projection.  

 \begin{lemma}\label{lem:dual}
Let $F:\R^m\to[0,+\infty)$ and $G:\R^n\to[0,+\infty)$ be positive-definite  one-homogeneous convex functions. 
 Let $T:\R^n\to\R^m$ be a linear map  (regarding as a matrix $T\in \R^{m\times n}$). Then, the   eigenvalue problems of the  function pairs $(F\circ T,G)$,  $(G^*\circ T^\top,F^*)$,  $(F\circ T,G_{\mathrm{Ker}(T)})$,  $(G^*_{\inf},F^*)$, and  $(F,G_{\inf})$,   are equivalent in the sense that their nonzero eigenvalues are the same. 
 \end{lemma}
 
This lemma has many interesting applications. For example, taking $F(\vec
y)=\|\vec y\|_p$ and $G(\vec x)=\|\vec x\|_q$, the positive eigenvalues of
$(\|T\cdot\|_p,\|\cdot\|_q)$ and $(\|T^\top\cdot\|_{q^*},\|\cdot\|_{p^*})$
coincide. In particular, we have
\begin{example}
For a real matrix $T$ of the order $m\times n$, for $p,q\in[1,\infty]$, we have
$$\max\limits_{x\in\R^n\setminus\{ 0\}}\frac{\|T\vec x\|_p}{\|\vec x\|_q}=\max\limits_{y\in\R^m\setminus\{ 0\}}\frac{\|T^\top\vec x\|_{q^*}}{\|\vec x\|_{p^*}},$$
where $p^*,q^*$ are the H\"older conjugates of $p,q$. Taking $p=2$ and $q=\infty$, we immediately obtain the equality on the  $l^1$-polarization  constant (Proposition 3  in \cite{AmbrusNietert}).
\end{example}
 Since a hypergraph is uniquely determined  by its incidence matrix, we can directly define the eigenvalues of $p$-Laplacians  on vertices (resp. hyperedges) as the spectrums of $(\|T \cdot\|_p^p,\|\cdot\|_p^p)$ (resp.  $(\|T^\top \cdot\|_p^p,\|\cdot\|_p^p)$)   by means of the incidence matrix $T$. 
More interestingly, we can prove that there is a simple one-to-one
correspondence between the nonzero  eigenvalues of the vertex  $p$-Laplacian  
and the 
 edge 
$p^*$-Laplacian on a (hyper-)graph. 
This is quite important because it offers us two alternative ways to  estimate the nonvanishing eigenvalues, either through the $p$-Laplacian on vertices or through  the  
$p^*$-Laplacian on edges.
\begin{pro}\label{pro:p-Lap-dual}
The nonzero eigenvalues of the vertex 1-Laplacian and the edge $\infty$-Laplacian coincide.  
For $p>1$, denote by $\Delta_p^V$ and $\Delta_{p^*}^E$ the vertex  $p$-Laplacian and edge  $p^*$-Laplacian, respectively. Then
\begin{equation}\label{eq:1/p-eigen-p-Lap}
\{\lambda^{\frac 1p}:\lambda\text{ is a positive eigenvalue of }\Delta_p^V\}=\{\lambda^{\frac {1}{p^*}}:\lambda\text{ is a positive eigenvalue of }\Delta_{p^*}^E\}.    
\end{equation}
\end{pro}
\begin{proof}
Let $T$ be the vertex-edge incidence  matrix of the graph.  
According to Lemma \ref{lem:dual}, the positive eigenvalues  of
$(\|T\cdot\|_1,\|\cdot\|_1)$ and
$(\|T^\top\cdot\|_{\infty},\|\cdot\|_{\infty})$ coincide. It is known that the
unnormalized  eigenvalue problem of the vertex 1-Laplacian (resp. the  edge $\infty$-Laplacian) agrees with the eigenvalue problem of $(\|T\cdot\|_1,\|\cdot\|_1)$ (resp. $(\|T^\top\cdot\|_{\infty},\|\cdot\|_{\infty})$). 

Next, we consider the case of  $p\in(1,+\infty)$. By Lemma \ref{lem:dual}, the positive eigenvalues of $(\|T\cdot\|_p,\|\cdot\|_p)$ and $(\|T^\top\cdot\|_{p^*},\|\cdot\|_{p^*})$ coincide. 
And it can be checked that the unnormalized eigenvalue problem of the vertex $p$-Laplacian $\Delta_p^V$ is nothing but the eigenvalue problem of the $p$-homogeneous  function pair $(\|T\cdot\|_p^p,\|\cdot\|_p^p)$. Also,  
$$\{\lambda^{\frac 1p}:\lambda\text{ is an eigenvalue of } (\|T\cdot\|_p^p,\|\cdot\|_p^p)\}=\{\text{eigenvalues  of } (\|T\cdot\|_p,\|\cdot\|_p)\}.$$
Similar statements hold for $(\|T^\top\cdot\|_{p^*}^{p^*},\|\cdot\|_{p^*}^{p^*})$. These facts deduce  the desired relation  \eqref{eq:1/p-eigen-p-Lap}. 

For the normalized version, we need to consider
$(\|T\cdot\|_p,\|\cdot\|_{\deg,p})$ and
$(\|T^\top\cdot\|_{\deg,p,*},\|\cdot\|_{p^*})$, where $\|\vec
x\|_{\deg,p}=(\sum_{i\in V}\deg_i|x_i|^p)^{\frac1p}$, $p\ge1$. The previous
discussion  still works. The proof is then  completed. As a supplement,  the case of $p>1$ can also be proved via Lemma \ref{lem:conjugate}.
\end{proof}

In the setting of convex conjugates, we have the following analog of Lemma \ref{lem:dual}.
 \begin{lemma}\label{lem:conjugate}
Let $F:\R^m\to[0,+\infty)$ and $G:\R^n\to[0,+\infty)$ be positive-definite
$p$-homogeneous convex functions with $p>1$. Let  $T:\R^n\to\R^m$ be a linear
map  (i.e., a matrix $T\in \R^{m\times n}$). Then $\lambda$ is a  nonzero
eigenvalue of $(F\circ T,G)$ if and only if $\lambda^{p^*-1}$ is a  nonzero
eigenvalue of $(G^\star\circ T^\top,F^\star)$. Also, the nontrivial spectra of $(F\circ T,G)$,  $(F\circ T,G_{\mathrm{Ker}(T)})$ and   $(F,G_{\inf})$ coincide.  
 \end{lemma}

Below, we  present the proofs of Lemmas \ref{lem:dual} and \ref{lem:conjugate}.

\begin{proof}[Proof of Lemma \ref{lem:dual}]We need two claims for positive-definite  one-homogeneous convex functions: 
\begin{enumerate}[{Claim} 1.]
    \item If $\vec x\ne\vec0$ and $\vec y\in \nabla G(\vec x)$,  $G^*(\vec y)=1$. Similarly, for any  $\vec y\ne\vec0$  and $\vec x\in \nabla G^*(\vec y)$, $G(\vec x)=1$.

Proof: For any $ \vec y\in \nabla G(\vec x)$, the Euler identity for
one-homogeneous functions gives $\langle \vec y,\vec x\rangle=G(\vec x)$ and
then the definition of the subgradient implies $\langle \vec y,\vec x'\rangle\le G(\vec x')$, $\forall \vec x'\in\R^n$. Thus,  $G^*(\vec y)=\sup\limits_{x'\ne 0}\frac{\langle \vec y,\vec x'\rangle}{G(\vec x')}=1$. 
The other identity $G(\nabla G^*(\vec x))=1$ is similar.

\item If $G^*(\vec y)=G(\vec x)=1$, then $\vec y\in\nabla G(\vec x)$ if and only if $\vec x\in\nabla G^*(\vec y)$.

Proof: Suppose that $\vec y\in\nabla G(\vec x)$ and $G^*(\vec y)=G(\vec x)=1$. Then, for any
$\vec y'\in \R^n$, 
$$\langle \vec y'-\vec y,\vec x\rangle=\langle \vec y',\vec x\rangle-G(\vec x)\le G^*(\vec y')G(\vec x)-G(\vec x)= G^*(\vec y')-G^*(\vec y),$$ which means $\vec x\in \nabla G^*(\vec y)$. The other direction is similar.
\end{enumerate}

Let $(\lambda,\vec x)\in\R_+\times (\R^n\setminus\{\vec0\})$ be an eigenpair of $(F\circ T,G)$, i.e., $\vec0\in\nabla_x F(T\vec x)-\lambda \nabla_x G(\vec x)$. Thus, there exists $\vec u\in \nabla G(\vec x)$ such that $\lambda \vec u\in \nabla_x F(T\vec x)= T^\top\nabla F(T\vec x)$. Hence, there is a  $\vec v\in \nabla F(T\vec x)$ satisfying $\lambda \vec u=T^\top\vec v$. 
Without loss of generality, we suppose $G(\vec x)=1$, and then $F(T\vec x)=\lambda$. By Claim 1, we have $G^*(\vec u)=1$ and $F^*(\vec v)=1$. Note that  $F(T\vec x/\lambda)=1$ and $\vec v\in \nabla F(T\vec x)=\nabla F(T\vec x/\lambda)$.  Then, we could apply Claim 2 to derive that $T\vec x/\lambda\in \nabla F^*(\vec v)$ and $\vec x\in\nabla G^*(\vec u)=\nabla G^*(\lambda\vec u)=\nabla G^*(T^\top\vec v)$. Therefore, $T\vec x\in T\nabla G^*(T^\top\vec v)=\nabla_v G^*(T^\top\vec v)$. In consequence, we have
$$\vec 0=T\vec x-\lambda\cdot T\vec x/\lambda\in \nabla_v G^*(T^\top\vec v)-\lambda\nabla_v F^*(\vec v).$$ Consequently, $(\lambda,\vec v)\in\R_+\times \R^n$ is an eigenpair of $(G^*\circ T^\top,F^*)$. The other direction is similar. In summary, we have proved that the nonzero  eigenvalues of $(F\circ T,G)$ coincide with the nonzero eigenvalues of $(G^*\circ T^\top,F^*)$.

Next we replace $G$ by its projections, $G_{\inf}$ and $G_{\mathrm{Ker}(T)}$, respectively.  By 
Proposition \ref{pro:GPi-property},   $G_{\mathrm{Ker}(T)}(\vec x)$ is a convex function of $\vec x$, and it  satisfies  $G_{\mathrm{Ker}(T)}(\vec x+\vec z)=G_{\mathrm{Ker}(T)}(\vec x)$ for any $\vec z\in \mathrm{Ker}(T)$. It is easy to check that  $G_{\mathrm{Ker}(T)}$ is  one-homogeneous and positive-definite on $ \mathrm{Ker}(T)^\bot$. 

For any $ \vec y\in\mathrm{Range}(T)$, there exists a unique $ \vec x\in\mathrm{Ker}(T)^\bot$ such that $\vec y=T\vec x$. Thus, $$G_{\inf}(\vec y)=G_{\inf}(T\vec x)=\inf\limits_{x'\in T^{-1}(T\vec x)}G(\vec x')=\inf\limits_{z\in \mathrm{Ker}(T)}G(\vec x+\vec z)=G_{\mathrm{Ker}(T)}(\vec x).$$ 
Since $T|_{\mathrm{Ker}(T)^\bot}:\mathrm{Ker}(T)^\bot\to\mathrm{Range}(T)
$ is a linear isomorphism, $G_{\inf}$ is convex,  one-homogeneous and positive-definite on $\mathrm{Range}(T)$. And it is clear that  $\{\vec y\in \mathrm{Range}(T):G_{\inf}(\vec y)\le1\}=T\{\vec x\in\R^n:G(\vec x)\le 1\}$. Moreover, we have
\begin{align*}
G^*(T^\top\vec x)&=\sup\limits_{G(y)\le 1}\langle T^\top\vec x,\vec y\rangle=\sup\limits_{G(y)\le 1}\langle \vec x,T\vec y\rangle\\&=\sup\limits_{y:G_{\inf}(T y)\le 1}\langle \vec x,T\vec y\rangle=\sup\limits_{z\in\mathrm{Range}(T):G_{\inf}(z)\le 1}\langle \vec x,\vec z\rangle=G_{\inf}^*(\vec x).
\end{align*}
For any $\vec y\in \R^m$, 
\begin{align*}
G^*_{\mathrm{Ker}(T)}(T^\top\vec y)&=\sup\limits_{G_{\mathrm{Ker}(T)}(z)\le 1}\langle  T^\top \vec y,\vec z\rangle=  \sup\limits_{G_{\inf}(Tz)\le 1}\langle  \vec y,T\vec z\rangle\\&=  \sup\limits_{G_{\inf}(x)\le 1}\langle  \vec y,\vec x\rangle=\sup\limits_{G(y)\le 1}\langle \vec y,\vec x\rangle=G^*_{\inf}(\vec y).
\end{align*}
Thus, $G^*_{\mathrm{Ker}(T)}(T^\top\vec y)=G^*(T^\top\vec y)$. In consequence, the nonzero eigenvalues of $(G^*\circ T^\top,F^*)$, $(G^*_{\inf},F^*)$ and $(G^*_{\mathrm{Ker}(T)}\circ T^\top,F^*)$   are the same. By the previous results, the nonzero eigenvalues of $(F\circ T,G_{\mathrm{Ker}(T)})$ and $(G^*_{\mathrm{Ker}(T)}\circ T^\top,F^*)$ coincide; while the nonzero eigenvalues of $(G^*_{\inf},F^*)$ and $(F,G_{\inf})$ are the same.  We then complete the proof by putting these statements  together.
\end{proof}

\begin{remark}
The equality $G^*(T^\top\vec x)=G_{\inf}^*(\vec x)$ in the above proof is useful and interesting. It implies that, roughly speaking,  the section of the dual equals the dual of the projection, from  which one can easily prove   that every convex  polytope is a section of a regular simplex, and every centrally symmetric convex  polytope is a section of a crosspolytope ($l^1$-ball).  
\end{remark}

\begin{proof}[Proof of Lemma \ref{lem:conjugate}]
Let $(\lambda,\vec x)\in\R_+\times (\R^n\setminus\{\vec0\})$ be an eigenpair
of $(F\circ T,G)$. Then,  there exists $\vec u\in \nabla G(\vec x)$ such that
$\lambda \vec u\in \nabla_x F(T\vec x)= T^\top\nabla F(T\vec x)$. Hence, there
is a  $\vec v\in \nabla F(T\vec x)$ satisfying $\lambda \vec u=T^\top\vec
v$. By the properties of the Fenchel conjugate,  $\vec x\in \nabla G^\star(\vec u)$ and $T\vec x\in \nabla F^\star(\vec v)$. Since $G^\star$ is $p^*$-homogeneous, $\nabla G^\star$ is $(p^*-1)$-homogeneous.  Accordingly, $$T\vec x\in T\nabla G^\star(\frac1\lambda T^\top\vec v)=(\frac1\lambda)^{p^*-1}T\nabla G^\star(T^\top\vec v)=\lambda^{1-p^*}\nabla_v G^\star(T^\top\vec v)$$
and hence, $\vec0\in \nabla_v G^\star(T^\top\vec v)-\lambda^{p^*-1}\nabla_v F^\star(\vec v)$, meaning that  $\lambda^{p^*-1}$ is a  nonzero eigenvalue of $(G^\star\circ T^\top,F^\star)$. The converse is similar. 

Next, we focus on the function  pair $(F\circ T,G_{\mathrm{Ker}(T)})$. 
By the fact that $\nabla_x F(T\vec x)=T^\top \nabla F(T\vec x)\subset
\mathrm{Range}(T^\top) = \mathrm{Ker}(T)^\bot$ and $\lambda\ne 0$, in combination with \eqref{eq:gradient-convex-G} in Proposition \ref{pro:GPi-property}, we have
$$
\vec0\in 
\nabla_x F(T\vec x)\cap \mathrm{Ker}(T)^\bot-\lambda \nabla_x G(\vec x) \cap\mathrm{Ker}(T)^\bot
\subset\nabla_x F(T\vec x)-\lambda \nabla_x G_{\mathrm{Ker}(T)}(\vec x) 
$$
implying that $(\lambda,\vec x)$ is  an eigenpair of $(F\circ T,G_{\mathrm{Ker}(T)})$. 
The converse needs the following statement.

Argument: If $G:\R^n\to[0,+\infty)$  is continuous,   positive-definite and $p$-homogeneous with $p\ge1$, then for any $\vec x$, $\inf_{\vec z\in \mathrm{Ker}(T)}G(\vec x+\vec z)$ can reach its minimum.  

Proof: 
Suppose on the contrary that there exists  $\vec x$ such that $\inf_{\vec z\in \mathrm{Ker}(T)}G(\vec x+\vec z)$ cannot reach its infimum. Then $\vec x\ne\vec 0$ and there exist $\vec x^n$ with $\vec x^n-\vec x\in \mathrm{Ker}(T)$, such that 
$$\lim\limits_{n\to+\infty}G(\vec x^n)= \inf\limits_{\vec z\in \mathrm{Ker}(T)}G(\vec x+\vec z)\;\text{ and }\;\lim\limits_{n\to+\infty}\|\vec x^n\|_2=+\infty.$$
 Then $\vec x^n/\|\vec x^n\|_2$ has a limit point $\vec x^0$. Clearly, $\|\vec x^0\|_2=1$. By the continuity of $G$, $G(\vec x^0)=\lim\limits_{n\to+\infty}G(\frac{\vec x^n}{\|\vec x^n\|_2})=\lim\limits_{n\to+\infty}\frac{G(\vec x^n)}{\|\vec x^n\|_2^p}=0$, which contradicts  the condition that $G$ is positive-definite.  

\vspace{0.16cm}

Now, let $(\lambda,\vec x)$ be an eigenpair of  $(F\circ
T,G_{\mathrm{Ker}(T)})$. The above argument yields that there exists $\vec x'$
such that $\vec x'-\vec x\in \mathrm{Ker}(T)$ and $G(\vec
x')=G_{\mathrm{Ker}(T)}(\vec x)$. Then, in combination  with \eqref{eq:gradient-convex-G} in Proposition \ref{pro:GPi-property}, we  derive  
$$\vec0\in \nabla_x F(T\vec x)-\lambda \nabla_x G_{\mathrm{Ker}(T)}(\vec x) =\nabla_x F(T\vec x')-\lambda \nabla_x G(\vec x') \cap\mathrm{Ker}(T)^\bot.$$
This implies that $(\lambda,\vec x')$ is an eigenvalue of $(F\circ T,G)$. 
Therefore, the nonzero eigenvalues of $(F\circ T,G)$ and $(F\circ T,G_{\mathrm{Ker}(T)})$ are the same.

Since $G_{\mathrm{Ker}(T)}=G_{\inf}\circ T$ and $T|_{\mathrm{Ker}(T)^\bot}:\mathrm{Ker}(T)^\bot\to\mathrm{Range}(T)
$ is a homeomorphism, we can write $(F\circ T,G_{\mathrm{Ker}(T)})=(F\circ T,G_{\inf}\circ T)$. Then, we can apply Proposition \ref{pro:odd-homeomorphism} to derive that the nonzero eigenvalues of
$(F\circ T,G_{\mathrm{Ker}(T)})$ and   $(F,G_{\inf})$ coincide. 
\end{proof}

\begin{remark}
The variational   characterization of   the second eigenvalue of graph $p$-Laplacian (see Example \ref{exam:character-2nd-p-Lap}) is also a direct consequence of  Lemmas \ref{lem:dual} and \ref{lem:conjugate}.  
\end{remark}

\section{Homogeneous and piecewise multilinear  extensions}
\label{sec:extension}

First, we recall the  definition of the \emph{original Lov\'asz extension}.
\begin{defn}
Given a function $f:\power(V)\to \R$, its \textbf{original Lov\'asz extension} is the function $f^L:\R^V\to \R$ defined as
\begin{equation}\label{eq:Lovasz-PL}
f^L(\vec x):=\sum_{i=1}^{n-1} (x_{(i+1)}-x_{(i)})f(V_{i}(\vec x))+x_{(1)}f(V),
\end{equation}
where $V_{i}(\vec x):=\{j\in V: x_j> x_{(i)}\}$ and $x_{(1)}\le x_{(2)}\le\ldots\le x_{(n)}$ is a rearrangement of $\vec x:=(x_1,\ldots,x_n)$ in non-deceasing order.
\end{defn}

The disjoint-pair Lov\'asz extension is defined in a similar manner (see \cite{JostZhang-PL} for details), 
and we still use $f^L$ to indicate  the disjoint-pair Lov\'asz extension  of $f$.



\begin{defn}[piecewise multilinear extension]\label{defn:piece-multilinear}
Given $V_i=\{1,\cdots,n_i\}$ and the power set $\mathcal{P}(V_i)$, $i=1,\cdots,k$,   for a discrete function  $f:\mathcal{P}(V_1)\times \cdots\times \mathcal{P}(V_k)\to \R$, we define the piecewise multilinear function on $\R^{n_1}\times\cdots\times\R^{n_k}$ 
by
\begin{equation*}
f^M(\vec x^1,\cdots,\vec x^k)=\sum_{i_1\in V_1,\cdots,i_k\in V_k}\prod_{l=1}^k(x_{(i_l)}^l-x_{(i_l-1)}^l)f(V^{(i_1)}(\vec x^1),\cdots,V^{(i_k)}(\vec x^k)),
\end{equation*}
 where $V^{(i)}(\vec x^l):=\{j\in V_l: x_j^l> x_{(i-1)}^l\}$ for $i\ge 2$, $V^{(1)}(\vec x^l)=V_l$, $x_{(0)}^l:=0$,  $x_{(1)}^l\le x_{(2)}^l\le\ldots\le x_{(n_l)}^l$ is a rearrangement of $\vec x^l:=(x_1^l,\ldots,x_{n_l}^l)$ in non-deceasing order,   for any 
 $\vec x^1\in\R^{n_1},\cdots,\vec x^k\in\R^{n_k}$. 
\end{defn}

Since the definition of $f^M$ doesn't involve the data on $(A_1,\cdots,A_k)$ if $A_i=\varnothing$ for some $i$, we   can set $f(A_1,\cdots,A_k)=0$ whenever $A_i=\varnothing$  for some $i=1,\cdots,k$. 

\begin{pro}\label{pro:L-P}
Under the notions in Definition \ref{defn:piece-multilinear}, for fixed $\vec x^2\in\R^{n_2},\cdots,\vec x^k\in\R^{n_k}$, let $\tilde{f}:\power(V_1)\to\R$ be defined as $\tilde{f}(A)=f^M(\vec1_{A},\vec x^2,\cdots,\vec x^k)$. Then $\tilde{f}^L(\vec x)=f^M(\vec x,\vec x^2,\cdots,\vec x^k)$ for any $\vec x\in\R^{n_1}$.
\end{pro}
Proposition \ref{pro:L-P} shows that the piecewise multilinear  extension induces the Lov\'asz extension by  restricting $f^M$ to each component $\vec x^l\in\R^{n_l}$, $l=1,\cdots,k$; while if we restrict the piecewise multilinear  extension of a function  $f:\power(V)^k\to\R$ to the diagonal $\underbrace {(\vec x,\vec x,\cdots,\vec x)}_{k\text{ times}}\in (\R^n)^k$, we  obtain the following 

\begin{defn}
\label{defn:piece-polynomial}
Given $V=\{1,\cdots,n\}$ and its power set $\mathcal{P}(V)$, for a  function $f:\mathcal{P}(V)^k\to \R$, we define the piecewise polynomial extension $f^M_\triangle$ on $\R^n$ by
$$f^M_\triangle(\vec x):=f^M(\vec x,\cdots,\vec x),\;\;\;\forall \vec x\in\R^n. $$
\end{defn}


Some special examples on graphs are presented in Table \ref{tab:double-Lov}.



It is also useful to provide  the multiple integral representation of the piecewise multilinear extension in Definition \ref{defn:piece-multilinear}. For example, given a function  $f:\power(V)^{k}\to \R$ with the assumption that $f(A_1,\cdots,A_k)=0$ whenever $A_i\in\{V,\varnothing\}$ for some $i$, we have
\begin{equation}\label{eq:multiple-integral}
f^M(\vec x^1,\cdots,\vec x^k)=\int_{\min\vec x^k}^{\max\vec x^k}\cdots\int_{\min\vec x^1}^{\max\vec x^1}f(V^{t_1}(\vec x^1),\cdots,V^{t_k}(\vec x^k))dt_1\cdots dt_k,
\end{equation}
where  $V^{t_l}(\vec x^l)=\{j\in V: x^l_j>t_l\}$,  $l=1,\cdots,k$. 
For a general  $f:\power(V)^{k}\to \R$ without any additional assumptions, the definition \eqref{eq:multiple-integral} should be modified by  adding some standard  remainder terms to guarantee the  condition $f^M(\vec 1_{A_1},\cdots,\vec1_{A_k})=f(A_1,\cdots,A_k)$. Since these remainder terms are routine, we don't write down them explicitly  for simplicity. Next, we show a simple formula for $f^M$ when $f$ is modular on each component.

\begin{defn}
Given a function $f:\mathcal{P}(V_1)\times \cdots\times \mathcal{P}(V_k)\to \R$, 
let $f_{A_1,\cdots,\widehat{A_i},\cdots,A_k}:\power(V_i)\to\R$ be defined as  $f_{A_1,\cdots,\widehat{A_i},\cdots,A_k}(A_i)=f(A_1,\cdots,A_k)$. 
We say that $f$ is modular on each component if 
$f_{A_1,\cdots,\widehat{A_i},\cdots,A_k}$ is modular 
for any $i$,  $A_1,\cdots,A_k$.
\end{defn}

\begin{pro}\label{pro:modular-f}
A function $f:\mathcal{P}(V_1)\times \cdots\times \mathcal{P}(V_k)\to \R$ is modular on each component if and only if  $f^M$ is multilinear. And at this time, $f^M$ is determined by 
\begin{equation}\label{eq:multiple-integral2}
f^M(\vec x^1,\cdots,\vec x^k)=\int_{0}^{\max\vec x^k}\cdots\int_{0}^{\max\vec x^1}f(V^{t_1}(\vec x^1),\cdots,V^{t_k}(\vec x^k))dt_1\cdots dt_k
\end{equation}
where  $V^{t_l}(\vec x^l)=\{j\in V_l: x^l_j>t_l\}$,  $l=1,\cdots,k$. 
\end{pro}
\begin{proof}
Suppose that $f$ is modular on each component. By Definition \ref{defn:piece-multilinear},  $f^M$  must be linear on each component. Thus, $f^M$ is a $k$-homogeneous polynomial and it is linear on each variable  $x_i^l$. 
Therefore, the explicit expression is uniquely determined by the data on the subset $\{\vec x^1\in\R^{n_1}:\min\vec x^1=0\}\times\cdots\times\{\vec x^k\in\R^{n_k}:\min\vec x^k=0\}$. Note that on such a subset, the formula  \eqref{eq:multiple-integral2} can be derived directly from Definition \ref{defn:piece-multilinear}. 

For the converse, suppose $f^M$ is multilinear and $f$ is not modular on its first component. Then, as shown in Proposition \ref{pro:L-P}, the restriction of $f^M$ to its first component $f^M(\vec x,\vec 1_{A_2},\cdots,\vec 1_{A_k})=\tilde{f}^L(\vec x)$ is the Lov\'asz extension of a non-modular function, which implies that $f^M$ is not linear on its first component, a contradiction.
\end{proof}



\begin{table}
\centering
\caption{\small Piecewise bilinear extension of some objective functions on $\power(V)\times \power(V)$.}
\begin{tabular}{|l|l|}
               \hline
               Objective function $f(A,B)$ & 
          Piecewise bilinear      extension $f^Q(\vec x,\vec y)$   \\
               \hline
                 $\#E(A,B)$ & $\sum_{i\sim j}(x_iy_j+x_jy_i)$\\
                 \hline
               constant $c$&$c\max_i x_i \max_i y_i$\\
               \hline
               $\# A\cdot\# B$ & $(\sum_i x_i)(\sum_i y_i)$\\
               \hline
                       $\# (A\cap B)$ & $\sum_i x_iy_i$\\
               \hline
             \end{tabular}
             \label{tab:double-Lov}
\end{table}


 \begin{example}\label{example:k-uniform-hypergraph}
 Let $G=(V,E)$ be a $k$-uniform hypergraph, i.e., every edge has cardinality $k$. 
Let $f:\mathcal{P}(V)^k\to\R$ be defined as  $f(A_1,\cdots,A_k)=\#E(A_1,\cdots,A_k):=\#\{(i_1,\cdots,i_k):i_1\in A_1,\cdots,i_k\in A_k,\{i_1,\cdots,i_k\}\in E\}$. 
Then we have $f^M(\vec x^1,\cdots,\vec x^k)=\sum\limits_{i_1,\cdots,i_k\in V,\{i_1,\cdots,i_k\}\in E}x^1_{i_1}\cdots x^k_{i_k}$.

For $g(A_1,\cdots,A_k)=\prod_{j=1}^k\#A_j$, one has $g^M(\vec x^1,\cdots,\vec x^k)=\prod_{j=1}^k\sum_{i\in V} x^j_i$, where $\vec x^j=(x^j_1,\cdots,x^j_n)$, $j=1,\cdots,k$.

We can use the formula  \eqref{eq:multiple-integral2} to get the closed form of $f^M$. Note that $\#E(A_1,\cdots,A_k)$ is modular on each $A_i$. Thus, for $\vec x^1,\cdots,\vec x^k$ with $\min\vec x^1=\cdots=\min\vec x^k=0$, 
\begin{align*}
f^M(\vec x^1,\cdots,\vec x^k)
&=\int_{0}^{\max\vec x^k}\cdots\int_{0}^{\max\vec x^1}\#E(V^{t_1}(\vec x^1),\cdots,V^{t_k}(\vec x^k))dt_1\cdots dt_k
\\&=\int_{0}^{\max\vec x^k}\cdots\int_{0}^{\max\vec x^1}\sum\limits_{\{i_1,\cdots,i_k\}\in E} 1_{x^1_{i_1}>t_1}\cdots1_{x^k_{i_k}>t_k} dt_1\cdots dt_k
\\&=\sum\limits_{\{i_1,\cdots,i_k\}\in E}\int_{0}^{\max\vec x^k}\cdots\int_{0}^{\max\vec x^1} 1_{x^1_{i_1}>t_1}\cdots1_{x^k_{i_k}>t_k} dt_1\cdots dt_k
\\&=\sum\limits_{\{i_1,\cdots,i_k\}\in E}x^1_{i_1}\cdots x^k_{i_k}   . 
\end{align*}
According to Proposition \ref{pro:modular-f}, for any $\vec x^1,\cdots,\vec x^k$, 
$$f^M(\vec x^1,\cdots,\vec x^k)=\sum\limits_{i_1,\cdots,i_k\in V,\{i_1,\cdots,i_k\}\in E}x^1_{i_1}\cdots x^k_{i_k}.$$

One can do a similar calculation for $\#A_1\cdots\#A_k$ by employing Proposition \ref{pro:modular-f}, but it is more convenient to use Proposition \ref{pro:separable-product} below.
\end{example}

We will see in Section  \ref{sec:Turan} that  Table \ref{tab:double-Lov} and
Example \ref{example:k-uniform-hypergraph} are closely related to the
Motzkin-Straus theorem and the Lagrangian density 
 of hypergraphs in the study of the Tur\'an problems.

According to Theorem 3.10 in \cite{BH20-Lorentzian}, $\B\subset \power(V)$ is the set of bases of a  matroid on $V$ if and only if the polynomial  $\sum_{B\in \B}\prod_{i\in B} x_i$ is Lorentzian (i.e., strong log-concave). Combining this argument with  Proposition \ref{pro:modular-f}, we immediately obtain
\begin{pro}
Suppose that $f:\mathcal{P}(V)^k\to \{0,1\}$ satisfy $f(\{i_1\},\cdots,\{i_k\})=0$  if $i_1,\cdots,i_k$ are not pairwise distinct. Then, $f^M_\triangle$ is a Lorentzian polynomial if and only if $f$ is modular and $\{\{i_1,\cdots,i_k\}\subset V:f(\{i_1\},\cdots,\{i_k\})=1\}$ is the set of bases of a matroid on $V$. 
\end{pro}

The following is a generalization of the disjoint-pair Lov\'asz extension.

\begin{defn}\label{def:multiple-integral}
For a function $f:\power_2(V_1)\times\cdots\times\power_2(V_k)\to \R$,  the multiple integral  extension on $\R^{n_1}\times\cdots\times\R^{n_k}$ is defined as
\begin{align*}
&f^M(\vec x^1,\cdots,\vec x^k)\\=~&\int_0^{\|\vec x^k\|_\infty}\cdots\int_0^{\|\vec x^1\|_\infty}f(V^{t_1}_+(\vec x^1),V^{t_1}_-(\vec x^1),\cdots,V^{t_k}_+(\vec x^k),V^{t_k}_-(\vec x^k))dt_1\cdots dt_k,
\end{align*}
where $\power_2(V_l)=\{(A_+,A_-):A_+,A_-\subset V_l,A_+\cap A_-=\varnothing\}$,  and  $V^{t_l}_\pm(\vec x^l)=\{j\in V_l:\pm x^l_j>t_l\}$,  $l=1,\cdots,k$.
\end{defn}

The property of the multiple integral extension in Definition \ref{def:multiple-integral} is very similar to the piecewise multilinear extension introduced in Definition \ref{defn:piece-multilinear}, but its integral formulation is more concise and it  is convenient for computation. As an analog to Table \ref{tab:double-Lov}, we refer to Table \ref{tab:L-multi-integral} for some examples of Definition \ref{def:multiple-integral}. 

\begin{table}
\centering
\caption{\small Multiple integral extension of typical  objective functions.}
\begin{tabular}{|l|l|}
               \hline
 Objective function $f(A_+^1,A_-^1,\cdots,A_+^k,A_-^k)$  & Multiple 
 extension 
               $f^M(\vec x^1,\cdots,\vec x^k)$  \\
               \hline
               $\Pi_{i=1}^k\#(A_+^i\cup A_-^i)$ & $\Pi_{i=1}^k\|\vec x^i\|_1$\\
               \hline
               $1$ & $\Pi_{i=1}^k\|\vec x^i\|_\infty$\\
               \hline
            \end{tabular}
             \label{tab:L-multi-integral}
\end{table}

\begin{pro}\label{pro:separable-product}
For  $f:\mathcal{P}(V_1)\times \cdots\times \mathcal{P}(V_k)\to \R$ in the form of multiplication  $f(A_1,\cdots,A_k):=\prod_{i=1}^kf_i(A_i)$, $\forall (A_1,\cdots,A_k)\in \mathcal{P}(V_1)\times \cdots\times \mathcal{P}(V_k)$, we have $f^M(\vec x^1,\cdots,\vec x^k)=\prod_{i=1}^kf_i^L(\vec x^i)$, $\forall (\vec x^1,\cdots,\vec x^k)$.

For $f:\power_2(V_1)\times\cdots\times\power_2(V_k)\to \R$ with the form  $f(A_1,B_1\cdots,A_k,B_k):=\prod_{i=1}^kf_i(A_i,B_i)$, $\forall (A_1,B_1,\cdots,A_k,B_k)\in \power_2(V_1)\times\cdots\times\power_2(V_k)$, there similarly holds $f^M(\vec x^1,\cdots,\vec x^k)=\prod_{i=1}^kf_i^L(\vec x^i)$.
\end{pro}

\begin{remark}
The original Lov\'asz extension  identifies $A\in \power(V){\setminus\{\varnothing\}}$ with $\vec 1_A\in \R^V$, and then extends 
$f$ from the set $\{\vec 1_A:A\in\mathcal{P}(V){\setminus\{\varnothing\}}\}$ to $\R^V$ in a piecewise linear way.

The disjoint-pair Lov\'asz extension  identifies $(A,B)\in \power_2(V){\setminus\{(\varnothing,\varnothing)\}}$ with $\vec 1_A-\vec 1_B\in \R^V$, and then extends $f:\power_2(V)\to\R$ piecewise-linearly.


The piecewise bilinear extension   identifies $(A,B)\in (\power(V){\setminus\{\varnothing\}})^2$ with $(\vec 1_A,\vec 1_B)\in  \R^{n}\times \R^n$ where $n=\#V$, and then extends $f$ to a piecewise bilinear function. 

In general settings, the piecewise multilinear extension 
identifies $(A_1,\cdots,A_k)\in (\power(V){\setminus\{\varnothing\}})^k$  with $(\vec 1_{A_1},\cdots,\vec 1_{A_k})\in  (\R^{n})^k$, and then extends $f$ to  a $k$-homogeneous piecewise  multilinear  function.  Moreover, the  multiple integral extension identifies 
$(A_+^1,A_-^1,\cdots,A_+^k,A_-^k)\in (\power_2(V){\setminus\{(\varnothing,\varnothing)\}})^k$  with $(\vec 1_{A_+^1}-\vec 1_{A_-^1},\cdots,\vec 1_{A_+^k}-\vec 1_{A_-^k})\in  (\R^{n})^k$, and then extends $f$ to  a piecewise $k$-homogeneous polynomial.
\end{remark}

\begin{defn}[rank of a function] 
Let $X_1,\cdots,X_k$  be  nonempty sets. A function $F:X_1\times \cdots\times X_k\to\R$ is a {\sl basic function} if $F(\vec x^1,\cdots,\vec x^k)=\prod_{i=1}^kF_i(\vec x^i)$ for some function $F_i:X_i\to\R$.  The  rank of a function
$F:X_1\times \cdots\times X_k\to\R$, denoted by $\mathrm{rank}(F)$,  is
the minimum number of basic functions needed to sum to $F$. If there is no such a representation, we set $\mathrm{rank}(F)=\infty$.
\end{defn}

\begin{defn}[slice rank of a function] 
Let $X_1,\cdots,X_k$  be  nonempty sets. A function $F:X_1\times \cdots\times X_k\to\R$ is a {\sl slice} if it can be written as the  
product of  a function on $X_i$ and a function on $\prod_{j\in\{1,\cdots,k\}\setminus\{i\}}X_j$,  for some $i\in \{1,\cdots,k\}$.  The slice rank of a function
$F:X_1\times \cdots\times X_k\to\R$, denoted by $\mathrm{slice}$-$\mathrm{rank}(F)$,  is
the minimum number of slices needed to sum to $F$. 
\end{defn}

\begin{pro}\label{pro:slice-rank-invariant} 
The slice rank  of $f:\mathcal{P}(V_1)\times \cdots\times \mathcal{P}(V_k)\to \R$ equals  the slice rank  of  $f^M:\R^{n_1}\times\cdots\times\R^{n_k}\to \R$, namely, $\mathrm{slice}$-$\mathrm{rank}(f^M)=\mathrm{slice}$-$\mathrm{rank}(f)$. Moreover, $\mathrm{rank}(f^M)=\mathrm{rank}(f)$ and $\mathrm{slice}$-$\mathrm{rank}(f)\le \mathrm{rank}(f)\le \#\mathrm{support}(f)$. 
\end{pro}
\begin{proof}
If $f$ is a slice, then by the definition of piecewise multilinear extension, $f^M$ must be also a slice, If $f^M(\vec x^1,\cdots,\vec x^k)=\hat{F}(\vec x^1)\tilde{F}(\vec x^2,\cdots,\vec x^k)$ is a slice, then 
taking $\hat{f}(A_1)=\hat{F}(\vec1_{A_1})$ and $\tilde{f}(A_2,\cdots,A_k)=\tilde{F}(\vec 1_{A_2},\cdots, \vec 1_{A_k})$, we have  $f(A_1,\cdots,A_k)=f^M(\vec 1_{A_1},\cdots, \vec 1_{A_k})=\hat{f}(A_1)\tilde{f}(A_2,\cdots,A_k)$, meaning that $f$ is a slice. Then the equality  $\mathrm{slice}$-$\mathrm{rank}(f^M)=\mathrm{slice}$-$\mathrm{rank}(f)$ is proved by the  equivalence of slices.  The proof of $\mathrm{rank}(f^M)=\mathrm{rank}(f)$ is similar.
\end{proof}

This implies that rank and slice rank are invariant under the piecewise multilinear extension.

\begin{pro}
Suppose that  $f:\power(V)^k\to\R$ satisfies  $f(A_1,\cdots,A_k)\ne 0$ if and only if $A_1=\cdots=A_k\ne\varnothing$.  Then the slice rank  of $f^M$ is $2^{\#V}-1$.  Also, if $f:\power(V)^k\to\R$ is modular on  each  component, and $f(\{i_1\},\cdots,\{i_k\})\ne 0$ if and only if $i_1=\cdots=i_k$, then $\mathrm{slice}$-$\mathrm{rank}(f^M)=\#V$. 
\end{pro}
\begin{proof}
Regarding $f$ as a $\underbrace {2^{\#V}\times\cdots\times 2^{\#V}}_{k\text{ times}}$  tensor, the condition means that $f$ is a diagonal tensor with only one zero diagonal  element. Then, by Tao's lemma on diagonal tensors, $f$ has the slice rank  $2^{\#V}-1$. By Proposition \ref{pro:slice-rank-invariant},  $f^M$ also has the slice rank $2^{\#V}-1$. 

For the modular case,  Proposition \ref{pro:modular-f} implies that $f^M$ is multilinear and thus we can regard $f^M$ (or $f$) as a tensor. Then the result is equivalent to   Tao's lemma on diagonal tensors.
\end{proof}

Both the piecewise multilinear extensions (Definition
\ref{defn:piece-multilinear}) and the multiple integral extension (Definition
\ref{def:multiple-integral})  are also  called {\sl  homogeneous extensions},
and it should be noted that \textbf{we use the same notion $f^M$ to express
  these extensions of a function $f$}. Next, we introduce the novel concept of
perfect domain pairs for studying incomplete data with the tools of extension methods.  

For constraint  sets $\A\subset (\power(V)\setminus \{\varnothing\})^k$ (or $\A\subset(\power_2(V)\setminus \{(\varnothing,\varnothing)\})^k$) and $\D\subset (\R^n)^k$,  
we can define their dual feasible sets $\A(\D)$ and $\D(\A)$ as follows:
\begin{itemize}
    \item $\D(\A)\subset (\R^n)^k$ is the maximal domain such that $f^M$ is well-defined (resp. positive/nonnegative) on $\D(\A)$ whenever $f$ is well-defined (resp. positive/nonnegative) on $\A$;
    \item $\A(\D)\subset \power(V)^k$ (or $\A(\D)\subset\power_2(V)^k$)  is the minimal domain of discrete functions   for defining   their extensions    on $\D$. 
\end{itemize}
 In concrete cases, it is defined by the following way:

For the piecewise multilinear extension introduced in Definition \ref{defn:piece-multilinear}, $\D(\A)=\{(\vec x^1, \cdots, \vec x^k)\in (\R_+^{n})^k:(V^{t_1}(\vec x^1),\cdots,V^{t_k}(\vec x^k))\in \A,\forall t_i<\max\vec x^i,i=1,\cdots,k\}$. 
Conversely, given a subset $\D\subset (\R^n)^k$, we have  $\A(\D)=\{(V^{t_1}(\vec x^1),\cdots,V^{t_k}(\vec x^k))\in(\power(V)\setminus \{\varnothing\})^k:(\vec x^1, \cdots, \vec x^k)\in \D,\,t_1,\cdots,t_k\in\R\}$. We call  $(V^{t_1}(\vec x^1),\cdots,V^{t_k}(\vec x^k))$ the {\sl multiple upper level set} of $\vec x:=(\vec x^1, \cdots, \vec x^k)$ at the multiple level $(t_1,\cdots,t_k)\in\R^k$. 

For the multiple integral extension introduced in Definition \ref{def:multiple-integral}, we similarly have $\D(\A)=\{(\vec x^1, \cdots, \vec x^k)\in (\R^{n})^k:(V^{t_1}_+(\vec x^1),V^{t_1}_-(\vec x^1),\cdots,V^{t_k}_+(\vec x^k),V^{t_k}_-(\vec x^k))\in \A,\forall t_i<\|\vec x^i\|_\infty\}$, and $\A(\D)=\{(V^{t_1}_+(\vec x^1),V^{t_1}_-(\vec x^1),\cdots,V^{t_k}_+(\vec x^k),V^{t_k}_-(\vec x^k))\in(\power_2(V)\setminus \{(\varnothing,\varnothing)\})^k:(\vec x^1, \cdots, \vec x^k)\in \D,\,t_1,\cdots,t_k\ge0\}$.   

\begin{defn}[perfect domain pair]
Given an extension way like Definition \ref{defn:piece-multilinear} or  \ref{def:multiple-integral}, a pair $(\A,\D)$ is a {\sl perfect domain pair} if $\A=\A(\D)$ and $\D=\D(\A)$. 
\end{defn}

It can be verified that both $\D\circ \A$ and $\A\circ \D$ are idempotent, i.e.,  $\D(\A(\D(\A)))=\D(\A)$ and $\A(\D(\A(\D)))=\A(\D)$ for any $\A$ and $\D$. Thus, for  $\A\ne\varnothing$, 
$(\A(\D(\A)),\D(\A))$ must be a perfect domain pair. Conversely, for  $\D\ne\varnothing$,  $(\A(\D),\D(\A(\D)))$ is a perfect domain pair.   For example, taking $$\mathcal{I}_k=\{(A_1,\cdots,A_k)\in\power(V)^k:\{A_i\}_{i=1}^k\text{ forms an inclusion chain}\}$$ and 
$$\mathcal{C}_k=\{(\vec x^1,\cdots,\vec x^k):\vec x^i\in\R^n_{\ge0}, \vec x^i\text{ and } \vec x^j \text{ are  comonotonic}, \forall i,j \}$$ then  $(\mathcal{I}_k,\mathcal{C}_k)$    is a perfect domain pair. This fact is shown in Proposition \ref{pro:perfect-domain-pair} and the proof of Theorem \ref{thm:Turan-general}.

We provide the following fundamental  theorem. 
\begin{theorem}\label{thm:optimal-identity-fg}
Given $f:\A\to\R$ and $g:\A\to [0,+\infty)$, we have
\begin{equation}\label{eq:AD-pair-inequality}
 \sup\limits_{A\in\A\cap \mathrm{supp}(g)}\frac{f(A)}{g(A)}\le \sup\limits_{\vec x\in\D\cap \mathrm{supp}(g^M)}\frac{f^M(\vec x)}{g^M(\vec x)}\le \sup\limits_{A\in\widetilde{\A}}\frac{f(A)}{g(A)}
\end{equation}
whenever $\{\vec1_A:A\in\A\}\subset \D$ and $\A(\D)\subset \widetilde{\A}$. The above inequality still holds when we replace all `$\sup$' and `$\le$' by `$\inf$' and `$\ge$', respectively. If we further assume that $(\A,\D)$ is a perfect domain pair, and $\mathrm{supp}(f)\subset \mathrm{supp}(g)$, then 
there hold the identities 
\begin{equation}\label{eq:equality-supp-fg}
\max\limits_{A\in\A\cap \mathrm{supp}(g)}\frac{f(A)}{g(A)}=\max\limits_{\vec x\in\D\cap \mathrm{supp}(g^M)}\frac{f^M(\vec x)}{g^M(\vec x)}\;\;\text{ and }\;\; \min\limits_{A\in\A\cap \mathrm{supp}(g)}\frac{f(A)}{g(A)}=\min\limits_{\vec x\in\D\cap \mathrm{supp}(g^M)}\frac{f^M(\vec x)}{g^M(\vec x)}.    
\end{equation}
\end{theorem}

\begin{proof}
Since $g^M(\vec 1_A)=g(A)$, we have $\vec 1_A\in \D\cap \mathrm{supp}(g^M)$ whenever $A\in \A\cap \mathrm{supp}(g)$. Thus, the first inequality in \eqref{eq:AD-pair-inequality} is proved. Note that for any $\vec x\in \D\cap \mathrm{supp}(g^M)$, $g^M(\vec x)>0$, 
and every multiple upper level set $(V^{t_1}(\vec x^1),\cdots,V^{t_k}(\vec x^k))$ belongs to $\A(\D)\subset \widetilde{\A}$. Hence,  an  approach  similar to the proof of  Theorem A in \cite{JostZhang-PL} can  derive the second  inequality in \eqref{eq:AD-pair-inequality}. 
In fact, we also have  
$$\sup\limits_{A\in\A}\frac{f(A)}{g(A)}\le \sup\limits_{\vec x\in\D}\frac{f^M(\vec x)}{g^M(\vec x)}\le \sup\limits_{A\in\widetilde{\A}}\frac{f(A)}{g(A)}.$$

For a perfect domain pair $(\A,\D)$, taking $\widetilde{\A}=\A(\D)=\A$, we immediately get $$\sup\limits_{A\in\A}\frac{f(A)}{g(A)}=\sup\limits_{\vec x\in\D}\frac{f^M(\vec x)}{g^M(\vec x)}\;\;\text{ and similarly}\;\; \inf\limits_{A\in\A}\frac{f(A)}{g(A)}=\inf\limits_{\vec x\in\D}\frac{f^M(\vec x)}{g^M(\vec x)}.$$
The additional condition $\mathrm{supp}(f)\subset \mathrm{supp}(g)$ implies that $f(A)=0$ whenever $g(A)=0$. Thus, by the definition of piecewise multilinear extension, for any $\vec x\in\D\cap \mathrm{supp}(g^M)$,  $$\frac{f^M(\vec x)}{g^M(\vec x)}\in\mathrm{conv}\left\{\frac{f(A)}{g(A)}:A\text{ is a multiple upper level set of }\vec x,\text{ and }g(A)>0\right\}.$$
The proof of \eqref{eq:equality-supp-fg} is completed.
\end{proof}

\begin{remark}
If we take the $k$-way Lov\'asz extension introduced in \cite{JostZhang-PL},
such as the (disjoint-pair) Lov\'asz extension, then $(\A,\D(\A))$ is always a
perfect domain pair, for any given $\A$. This is the reason why we don't use
the Terminology `perfect domain pair' in \cite{JostZhang-PL}. However, for
$k$-homogeneous extensions with $k\ge2$, such as the piecewise multilinear
extension, $(\A,\D(\A))$ does not necessarily have  to be a perfect domain pair, which leads to  a subtle difference. 
\end{remark}

\begin{pro}\label{pro:perfect-domain-pair} We have the following properties and examples on perfect domain pairs: 
\begin{itemize}
\item $(\A,\D_\A)$ and $(\B,\D_\B)$ are perfect domain pairs if and only if $(\A\times \B,\D_\A\times\D_\B)$ is a perfect domain pair;
\item The sets  $\{(A_1,\cdots,A_k)\in\power(V)^k:A_{\sigma(1)}\subset\cdots\subset A_{\sigma(k)}\text{ for some permutation }\sigma\in S_k\}$ and $\{(\vec x^1,\cdots,\vec x^k)\in (\R^n_+)^k:\text{pairwise comonotonic }\vec x^1,\cdots,\vec x^k\}$  form a perfect domain pair w.r.t. the  piecewise multilinear extension.
\item   $\{(A_1,\cdots,A_k)\in\power_2(V)^k:A_{\sigma(1)}\subset\cdots\subset A_{\sigma(k)}\text{ for some permutation }\sigma\in S_k\}$ and $\{(\vec x^1,\cdots,\vec x^k)\in (\R^n)^k:\text{pairwise absolutely  comonotonic }\vec x^1,\cdots,\vec x^k\}$  form a perfect domain pair w.r.t. the  multiple integral  extension, where $A\subset B$ for set-pairs $A=(A_+,A_-)$ and $B=(B_+,B_-)$ in $\power_2(V)$ means $A_+\subset B_+$ and $A_-\subset B_-$. Here the concept of  absolute  comonotonicity is introduced in Definition 2.4 in \cite{JostZhang-PL}.
\end{itemize}
\end{pro}

\begin{proof}[Proof of Theorem \ref{thm:Turan-general}]
Let $$\A=\{(A_1,\cdots,A_k)\in\power(V)^k:A_1,\cdots,A_k\text{ form an inclusion chain}\}.$$ Then 
\begin{align*}
  \D(\A)&=\{(\vec x^1,\cdots,\vec x^k)\in (\R^n_+)^k:V^{t_1}(\vec x^1),\cdots,V^{t_k}(\vec x^k)\text{ form an inclusion chain},\forall t_1,\cdots,t_k\}
  \\&=  \{(\vec x^1,\cdots,\vec x^k)\in (\R^n_+)^k:\vec x^1,\cdots,\vec x^k\in \overline{X_\sigma}\text{ for some }\sigma\in S_n\}
  \\&= \bigcup_{\sigma\in S_n}(\overline{X_\sigma})^k=\{(\vec x^1,\cdots,\vec x^k)\in (\R^n_+)^k:\vec x^1,\cdots,\vec x^k\text{ are pairwise comonotonic}\}
\end{align*}
where $X_{\sigma}=\{\vec x\in \R^n_+:x_{\sigma(1)}<\cdots<x_{\sigma(n)}\}$,  $\sigma$ is a permutation and $S_n$ is the finite symmetric group over $\{1,\cdots,n\}$. It is clear that $\A(\D(\A))=\A$, and thus $(\A,\D(\A))$ is a perfect pair. Thus, by Theorem \ref{thm:optimal-identity-fg}, we have
\begin{align*}
\max\limits_{\text{chain }\{A_1, A_2,\cdots, A_k\}}\frac{f(A_1,\cdots,A_k)}{g(A_1,\cdots,A_k)}&=\max\limits_{\sigma\in S_n}
 \max\limits_{\vec x^1,\cdots,\vec x^k\in \overline{X_\sigma}}
 \frac{f^M(\vec x^1,\cdots,\vec x^k)}{g^M(\vec x^1,\cdots,\vec x^k)}
 \\&=
 \max\limits_{\text{comonotonic }\vec x^1,\cdots,\vec x^k}
 \frac{f^M(\vec x^1,\cdots,\vec x^k)}{g^M(\vec x^1,\cdots,\vec x^k)}.
\end{align*}

Let $\A'=\{(A,\cdots,A):A\subset V\}$ and $\D'=\{(\vec x,\cdots,\vec x):\vec x\in \R^n_{\ge 0}\}$. Then $\A(\D')=\A\supset \A'$ and $\D(\A(\D'))=\D(\A)\supset \D'$. Applying Theorem \ref{thm:optimal-identity-fg} to the pairs $(\A,\D(\A))$, $(\A,\D')$ and $(\A',\D')$, respectively,  we obtain the desired result. 
\end{proof}

\begin{example}\label{ex:spectral-radius}
Given a simple graph $(V,E)$, 
let $f(A,B)=\#E(A,B)$ and $g(A,B)=\#(A\cap B)$ for $A,B\subset V$. Then  $f^Q(\vec x,\vec y)=\sum_{\{i,j\}\in E}(x_iy_j+x_jy_i)$ and $g^Q(\vec x,\vec y)=\sum_{i\in V} x_iy_i$. By Theorem \ref{thm:Turan-general} (or Proposition \ref{pro:Q-inequality}), we obtain
\begin{equation}\label{eq:inequality-adjacent-matrix}
\max\limits_{A}\frac{\# E(A,A)}{\#A}\le \max\limits_{x\ne 0}\frac{2\sum\limits_{\{i,j\}\in E} x_ix_j}{\|\vec x\|^2_2}\le \max\limits_{A\subset B}\frac{\# E(A,B)}{\#A}.
\end{equation}
Note that $\max\limits_{A\subset B}\frac{\# E(A,B)}{\#A}=\max\limits_{i\in V}
\mathrm{deg}(i)$, and $\max\frac{2\sum\limits_{\{i,j\}\in E} x_ix_j}{\|\vec
  x\|^2_2}=\lambda_{\max}$ is the largest eigenvalue of the  adjacency matrix
of the graph $(V,E)$, and  $\frac{\# E(A,A)}{\#A}$ is the average degree of the  subgraph  induced on $A$. Therefore,  \eqref{eq:inequality-adjacent-matrix} can be reformulated as $$\max\limits_{S\subset V}(\text{average degree of the induced subgraph on }S)\le\lambda_{\max}\le \max_{i\in V} \mathrm{deg}(i),$$
which leads to the standard upper bound and an interesting lower bound\footnote{We don't know whether the lower bound for the spectral radius is new. } for the classical graph spectral radius. Besides, by Theorem  \ref{thm:optimal-identity-fg}, we further have 
\begin{align*}
\max\limits_{A\subset B}\frac{| E(A,B)|}{|A|}&=\max\limits_{\substack{x,y\in\R^n_{\ge0}\\\text{comonotonic}}}\frac{\sum\limits_{\{i,j\}\in E} (x_iy_j+x_jy_i)}{\vec x^\top\vec y}
\\&\le\max\limits_{A\cap B\ne\varnothing}\frac{|E(A,B)|}{|A\cap B|}=\max\limits_{\substack{x,y\in\R^n_{\ge0}\\ \text{comaximal}}}\frac{\sum\limits_{\{i,j\}\in E} (x_iy_j+x_jy_i)}{\vec x^\top\vec y}
\end{align*}
where two vectors $\vec x$ and $\vec y$ are {\sl comaximal} if there exists an  index $i$ such that $x_i=\max_{j\in V} x_j$ and $y_i=\max_{j\in V} y_j$.
\end{example}

\begin{proof}[Proof of Theorem \ref{thm:fg-Huang}]
For $U\in\power(V)$ with $\#U=m$, taking $X=\{\vec x\in\R^n:\mathrm{supp}(\vec x)\subset U\}$,  $F=f^M_\Delta$ and $G=g^M_\Delta$ in Theorem \ref{thm:FG-p-homo-signed}, we have $\dim X=m$, and by  Theorem \ref{thm:Turan-general} as well as the proof of Theorem \ref{thm:FG-p-homo-signed},
\begin{align*}
\sup\limits_{(F',G')\in S(F,G;X)}\max\{\lambda_m(F',G'),-\lambda_m'(F',G')\}&\le \max\limits_{\vec x\in X}\frac {f^M_\Delta(\vec x)}{g^M_\Delta(\vec x)}
\\&\le \max\limits_{\text{chain }A_1,\cdots,A_k\text{ in }U} \frac{f(A_1,\cdots,A_k)}{g(A_1,\cdots,A_k)}.
\end{align*}
Finally, let $S(f,g)=\{(F',G'): f^M_\triangle(|\vec x|)\ge  |F'(\vec x)|\text{ and }g^M_\triangle(|\vec x|)= G'(\vec x),\forall \vec x\in\R^n\}=\{(F',g^M_\triangle(|\cdot|)):F'\in S(f)\}$, and note that $S(f,g)\subset \bigcap\limits_{U\subset V,\#U=m}S(F,G;\R^U)$. This completes the proof. 
\end{proof}


{
To state our general results, we shall recall the definition of log-concave polynomials and quasi-convex functions.

A $d$-homogeneous polynomial $P$ in $n$ real variables is  {\sl log-concave} if $P$ is positive on $\R_+^n$ and $\log P$ is concave on $\R_+^n$.

Given a convex set $\Omega$, a function  $F:\Omega\to\R$ is {\sl quasi-convex} (resp., {\sl quasi-concave})  if $F(t\vec x+(1-t)\vec y)\le\max\{F(\vec x),F(\vec y)\}$ (resp., $F(t\vec x+(1-t)\vec y)\ge\min\{F(\vec x),F(\vec y)\}$) for any $\vec x,\vec y\in\Omega$ and $t\in[0,1]$.
 }

\begin{theorem}\label{thm:tilde-H-f-M}
Let $H:\R^n_{\ge0}\setminus\{\vec 0\}\to\R\cup\{+\infty\}$ be a  zero-homogeneous and \textbf{quasi-concave}  function. For any functions $f_1,\cdots,f_n:\A\to \R_{\ge0}$, we have
 \begin{equation}\label{eq:H-minimum}
 \min\limits_{A\in \A}H(f_1(A),\cdots,f_n(A))=\inf\limits_{\vec x\in \D} H(f^M_1(\vec x),\cdots,f^M_n(\vec x))\end{equation}
 where $(\A,\D)$ forms a perfect domain pair w.r.t. the piecewise  multilinear extension. In addition, if $H:\R^n_{\ge0}\setminus\{\vec 0\}\to\R\cup\{-\infty\}$ is zero-homogeneous and \textbf{quasi-convex}, then  \begin{equation}\label{eq:H-max}
 \max\limits_{A\in \A}H(f_1(A),\cdots,f_n(A))=\sup\limits_{\vec x\in \D} H(f^M_1(\vec x),\cdots,f^M_n(\vec x))\end{equation}
\end{theorem}

We omit the proof because it is a slight modification of Theorem 3.1 in \cite{JostZhang-PL}.  

According to the proof of Theorem 2.30 in \cite{BH20-Lorentzian}, if a $d$-homogeneous polynomial $P$ in $n$ variables is log-concave, then  $P^{\frac1d}$ is concave on $\R_+^n$. Thus, there is no difficulty to check that  
$H(f_1,\cdots,f_n):=\frac{P(f_1,\cdots,f_n)}{(f_1+\cdots+f_n)^d}$ is zero-homogeneous and quasi-concave on $\R_+^n$. We then derive Proposition \ref{pro:log-concave-optimal} by employing Theorem  \ref{thm:tilde-H-f-M}. 

 {
In summary, if we want to solve a  combinatorial optimization problem, we can consider the piecewise  multilinear  extension, and solve the corresponding  continuous optimization problem  first. Then,  we can go back to the original  combinatorial optimization problem, as in the following cases: 
 \begin{itemize}
\item[Case 1.] 
If we are working on a perfect domain pair, then the solution of the corresponding equivalent continuous reformulation provides a solution of the original  combinatorial  problem exactly as we expect.
        \item[Case 2.] 
    If the domain pair is not perfect,  then the solution of  the corresponding  continuous optimization problem may not be a solution of the original problem, but  it provides a continuous relaxation and  a reasonable estimate for the  original  combinatorial  problem.
 \end{itemize}
}

In Sections  \ref{sec:LS} and \ref{sec:saddle},  we devote ourselves to min-max statements in the  context of Lusternik-Schnirelmann theory, saddle point problem and von Neumann's min-max theorem. 

\subsection{Min-max relation on  Lusternik-Schnirelmann theory}\label{sec:LS}

We will set up some new min-max arguments to construct the extension theory
related to  the critical point theory of  Lusternik-Schnirelmann type. 

Considering a tuple of finite sets $V:=(V_1,\cdots,V_k)$, we write $A\subset V$ if $A=(A_1,\cdots,A_k)$ with $A_i=(A_{i+},A_{i-})\in\power_2(V_i)$,  $i=1,\cdots,k$. For $A,B\subset V$, we have  the union $A\vee B=(A_{1+}\cup B_{1+},A_{1-}\cup B_{1-},\cdots,A_{k+}\cup B_{k+},A_{k-}\cup B_{k-})$, and the exchange    $A'=(A_1',\cdots,A_k')$ with $A_i'=(A_{i-},A_{i+})$. We say that $A$ and $B$ are disjoint if $(A_{i+}\cup A_{i-})\cap (B_{i+}\cup B_{i-})=\varnothing$,  $\forall i$.  Let $$\tilde{P}_{m}(V)=\{\{A^j\}_{j=1}^m\subset \power_2(V_1)\times\cdots\times \power_2(V_k):   A^1,\cdots,A^m\text{ are pairwise disjoint}\},$$
where  $A^j=(A^j_1,\cdots,A^j_k)\in \power_2(V_1)\times\cdots\times \power_2(V_k)$.  Clearly,  $\tilde{P}_{1}(V)=\{A\subset V\}:=  \power_2(V_1)\times\cdots\times \power_2(V_k)$.     For given $\{A^j\}_{j=1}^m\in \tilde{P}_{m}(V)$, denote by $\Sigma\{A^j\}$ the smallest family  containing $\{A^j\}$  which is closed under the union and the exchange operators.
\begin{defn}[subadditivity]
 We say $f:\tilde{P}_{1}(V)\to \R$ is {\sl weakly sub-additive} (resp., weakly super-additive) if $f(A)+f(B)- f(A\vee B)\ge0$ (resp., $\le 0$), for any disjoint subsets $A$ and $B$ in $V$.
\end{defn}

\begin{defn}
Given two functions 
$F,G:\R^n\to\R$,  a {\sl nodal domain decomposition}  of an eigenvector $\vec x\in\R^n$ w.r.t. an eigenvalue $\lambda$ of the function pair $(F,G)$   is a family of pairwise disjoint sets  $A^1,\cdots,A^m\subset V$ such that   $A^i\subset\mathrm{supp}(\vec x)$ and every   $(\lambda,\vec x|_{A^i})$ is an eigenpair of $(F,G)$, $\forall i$.  
\end{defn}

\begin{proof}[Proof of Theorem \ref{thm:Cheeger-type}]
Fix $A^1,\cdots,A^m$, and consider the linear subspace $X$ spanned by the corresponding  characteristic functions $\vec 1_{A^1}$, $\cdots$, $\vec 1_{A^m}$, where $\vec 1_{A^j}:=(\vec 1_{A^j_{1+}}-\vec 1_{A^j_{1-}},\cdots,\vec 1_{A^j_{k+}}-\vec 1_{A^j_{k-}})$. Given  $\vec x\in X$, there exist $t_1,\cdots,t_m\in\R$ such that $\vec x=\sum_{j=1}^m t_j\vec 1_{A^j} $. Then it can be verified that 
$(V^{t_1}_+(\vec x^1),V^{t_1}_-(\vec x^1),\cdots,V^{t_k}_+(\vec x^k),V^{t_k}_-(\vec x^k))\in \Sigma\{A^j\}$, $\forall (t_1,\cdots,t_k)\in \R^k_{\ge 0}$. Thus, taking $\A=\{A^1,\cdots,A^m\}$, $\D=X$, and  $\tilde{\A}=\Sigma\{A^j\}$ in Theorem \ref{thm:optimal-identity-fg}, 
we have
$$\frac{f^M(\vec x)}{g^M(\vec x)}\le \max_{A\in \Sigma\{A^j\}} \frac{f(A)}{g(A)}.$$
 Consequently, $\sup\limits_{\vec x \in X} \frac{f^M(\vec x)}{g^M(\vec x)}\leq \max\limits_{A\in \Sigma\{A^j\}} \frac{f(A)}{g(A)}$. 
 Together with the fact that $\gen(X)=\dim X\ge m$, 
 we have $$ \inf_{\gen(X)\ge m}\sup\limits_{\vec x\in X} \frac{f^M(\vec x)}{g^M(\vec x)}\le \min_{\{A^j\}\in \tilde{P}_{m}(V)}\max_{A\in \Sigma\{A^j\}} \frac{f(A)}{g(A)}.$$
 Thus, the first inequality  in \eqref{eq:minmax-union} is derived.

Similarly, fix $A^1,\cdots,A^{n+1-m}$,  and consider the linear subspace $X'$ spanned by the characteristic functions $\vec1_{A^1}$, $\cdots$, $\vec1_{A^{n+1-m}}$, where $n=\#V_1+\cdots+\#V_k$. According to the intersection theorem and the fact that $\dim X'=n+1-m$,  we have $X\cap X'\ne\varnothing$, for any $X\in \Gamma_m$.  Therefore, 
\begin{equation}\label{eq:PsiPsi'}
\sup_{x\in X} \frac{f^M(\vec x)}{g^M(\vec x)}
\ge\inf_{\vec x\in X'} \frac{f^M(\vec x)}{g^M(\vec x)}.
\end{equation}
For any $\vec x\in X'$, there exist $t_1,\cdots,t_{n+1-m}\in\R$ such that $\vec x=\sum_{j=1}^{n+1-m} t_j\vec1_{A^j} $. 
Similarly, by Theorem  \ref{thm:optimal-identity-fg},
we get
\begin{equation}\label{eq:fLgLge}
\frac{f^M(\vec x)}{g^M(\vec x)}\ge \min_{A\in \Sigma\{A^j\}} \frac{f(A)}{g(A)}.
\end{equation}
Together with \eqref{eq:PsiPsi'} and \eqref{eq:fLgLge}, $\sup\limits_{\vec x \in X} \frac{f^M(\vec x)}{g^M(\vec x)}\ge\inf\limits_{\vec x\in X'} \frac{f^M(\vec x)}{g^M(\vec x)}\ge \min\limits_{A\in \Sigma\{A^j\}} \frac{f(A)}{g(A)}$. By the arbitrariness of $X$ and $X'$, we have  the second inequality  in \eqref{eq:minmax-union}:
$$\inf_{\gen(X)\ge m}\sup\limits_{\vec x\in X} \frac{f^M(\vec x)}{g^M(\vec x)} \ge \max_{\{A^j\}\in \tilde{P}_{n+1-m}(V)}\min_{A\in \Sigma\{A^j\}} \frac{f(A)}{g(A)}.$$


Now we turn to the additional cases. 
Since $f$ is weakly  subadditive and $g$ is weakly  super-additive,  by the property of disjoint-pair  Lov\'asz extension, it can be checked that 
$f^L(\vec x)\le \sum_{j=1}^mf^L(t_j\vec1_{A^j})$ and $g^L(\vec x)\ge \sum_{j=1}^mg^L(t_j\vec1_{A^j})$. Together with the symmetry assumption for $f$ and $g$, for $\vec x=\sum_{j=1}^m t_j\vec 1_{A^j} $, 
$$\frac{f^L(\vec x)}{g^L(\vec x)}\le\max_{1\le j\le m} \frac{f^L(t_j\vec1_{A^j})}{g^L(t_j\vec1_{A^j})}
=\max_{1\le j\le m} \frac{f(A^j)}{g(A^j)}$$
and thus $$\lambda_m:=\inf_{\gen(X)\ge m}\sup\limits_{\vec x\in X} \frac{f^L(\vec x)}{g^L(\vec x)}\le\min_{\{A^j\}\in \tilde{P}_{m}(V)}\max_{i=1,\cdots,m} \frac{f(A^i)}{g(A^i)}.$$
Suppose $(\lambda_m,\vec x)$ is an eigenpair of $(f^L,g^L)$, and $\vec x$ has $k_m$  nodal domains $A^1,\cdots,A^{k_m}$. Taking $X=\mathrm{span}(\vec x|_{A^1},\cdots,\vec x|_{A^{k_m}})$, we have $\dim X=k_m$. Let  $A^j_\pm=\{v\in A^j:\pm x_v>0\}$ and denote by $\vec 1_{A^j}:=\vec 1_{A^j_+}-\vec 1_{A^j_-}$,  $j=1,\cdots,k_m$. 

Define a pre-order relation $\prec$ on $\R^n$: $\vec x\prec \vec y$ if $\triangle(\vec x)\subset \overline{\triangle(\vec y)}$ where 
\begin{equation}\label{eq:piece-cone-strong-comonotonic}
\triangle(\vec x):=\{\vec x'\in\R^n:x'_i<x'_j\Leftrightarrow  x_i<x_j,~\pm x_i>0\Leftrightarrow \pm x_i’>0, \forall i,j\in V_l,\,\forall l=1,\cdots,k\}.    
\end{equation} The pre-order induces an equivalence relation $\approx $ on $\R^n$:  $\vec x\approx  \vec y$ if $\triangle(\vec x)=\triangle(\vec y)$. 
It can be verified that:
\begin{itemize}
    \item $\vec x\approx  \vec y$ implies $\nabla f^L(\vec x)=\nabla f^L(\vec y)$;
    \item $\vec x\prec\vec y$ implies $\nabla f^L(\vec x)\supset\nabla f^L(\vec y)$;
    \item If $\vec x\prec\vec y$ and $(\lambda,\vec y)$ is an eigenpair of $(f^L,g^L)$, then $(\lambda,\vec x)$ is also an eigenpair.
    
Proof: The condition $\vec x\prec\vec y$ implies $\nabla f^L(\vec x)\supset\nabla f^L(\vec y)$ and $\nabla g^L(\vec x)\supset\nabla g^L(\vec y)$. Together with the assumption that $(\lambda,\vec y)$ is an eigenpair of $(f^L,g^L)$, we have $\vec 0\in\nabla f^L(\vec y)-\lambda\nabla g^L(\vec y) \subset \nabla f^L(\vec x)-\lambda\nabla g^L(\vec x)$ meaning that $(\lambda,\vec x)$ is also an eigenpair.
\end{itemize} 
Since the sets  $A^1,\cdots,A^{k_m}$ form a nodal domain decomposition of $\vec x$,  $(\lambda_m,\vec x|_{A^i})$ must be an eigenpair of $(f^L,g^L)$, $\forall i=1,\cdots,k_m$. It is clear that $\vec 1_{A^i}\in \overline{\triangle(\vec x|_{A^i})}$, which implies $\vec 1_{A^i}\prec \vec x|_{A^i}$. Thus $(\lambda_m,\vec 1_{A^i})$ is an eigenpair. Therefore, 
$$\lambda_m=\max_{i=1,\cdots,k_m} \frac{f^L(\vec1_{A^i})}{g^L(\vec1_{A^i})}=\max_{i=1,\cdots,k_m} \frac{f(A^i)}{g(A^i)}\ge \min_{\{A^j\}\in \tilde{P}_{k_m}(V)}\max_{i=1,\cdots,k_m} \frac{f(A^i)}{g(A^i)}.$$ 
\end{proof}

 Similar to Theorem \ref{thm:Cheeger-type}, we have  $$\min_{\{A^j\}\in \tilde{P}_{m}(V)}\max_{A\in \Sigma\{A^j\}} \frac{f(A)}{g(A)} \ge \inf_{\mathrm{cat}(X)\ge m}\sup\limits_{\vec x\in X} \frac{f^M(\vec x)}{g^M(\vec x)}$$
where the min-max value on the right hand side indicates  an eigenvalue of $(f^M,g^M)$.  

\vspace{0.16cm}

\textbf{The setting under the piecewise multilinear extension introduced in  Definition  \ref{defn:piece-multilinear}}

\vspace{0.16cm}

Considering a tuple of finite sets $V:=(V_1,\cdots,V_k)$, we rewrite $A\subset V$ if $A=(A_1,\cdots,A_k)$ with $A_i\in\power(V_i)$,  $i=1,\cdots,k$. For $A,B\subset V$, we have  the union $A\vee B=(A_1\cup B_1,\cdots,A_k\cup B_k)$. We say that $A$ and $B$ are disjoint if $A_i\cap B_i=\varnothing$,  $\forall i$.  Redefine  $$\tilde{P}_{m}(V)=\{\{A^j\}_{j=1}^m\subset \power(V_1)\times\cdots\times\power(V_k):   A^1,\cdots,A^m\text{ are pairwise disjoint}\},$$
where  $A^j=(A^j_1,\cdots,A^j_k)\in \power(V_1)\times\cdots\times\power(V_k)$.   For given $\{A^j\}_{j=1}^m\in \tilde{P}_{m}(V)$, denote by $\Sigma\{A^j\}$ the smallest family  containing $\{A^j\}$ and  closed under the union.

\begin{theorem}\label{thm:Cheeger-type2}
For 
$f,g:\tilde{P}_1(V)\to\R_+$,  we have
\begin{equation}\label{eq:minmax-union-}
\min_{\{A^j\}\in \tilde{P}_{m}(V)}\max_{A\in \Sigma\{A^j\}} \frac{f(A)}{g(A)} \ge \inf_{\gen(X)\ge m}\sup\limits_{\vec x\in X} \frac{f^M(|\vec x|)}{g^M(|\vec x|)}
\ge \max_{\{A^j\}\in \tilde{P}_{n+1-m}(V)}\min_{A\in \Sigma\{A^j\}} \frac{f(A)}{g(A)}
\end{equation}
where the absolute value $|\vec x|$ is taken component-wise. 
If we further assume that  $f$ is  submodular and $g$ is   supermodular, then \begin{equation}\label{eq:lovasz-k-way-Cheeger}
\min_{\{A^j\}\in \tilde{P}_{m}(V)}\max_{i=1,\cdots,m} \frac{f(A^i)}{g(A^i)} \ge \inf_{\gen(X)\ge m}\sup\limits_{\vec x\in X} \frac{f^L(|\vec x|)}{g^L(|\vec x|)}:=\lambda_m\ge \min_{\{A^j\}\in \tilde{P}_{k_m}(V)}\max_{i=1,\cdots,m} \frac{f(A^i)}{g(A^i)}    
\end{equation}
where $k_m$ is the number of 
 the nodal domains of an eigenvector w.r.t. the eigenvalue  $\lambda_m$ of the function pair $(f^L(|\cdot|),g^L(|\cdot|))$. Here $f^L$ represents the original Lov\'asz extension of $f$.

We can replace $f^L(|\cdot|)$ by $f^L$ in \eqref{eq:lovasz-k-way-Cheeger}, when we further suppose that $f$ is symmetric, i.e.,  $f(A_1,\cdots,A_k)=f(V_1\setminus A_1,\cdots,V_k\setminus A_k)$, $\forall A:=(A_1,\cdots,A_k)\subset V$.  
\end{theorem}


Theorem \ref{thm:Cheeger-type2} is a variant  analog of Theorem
\ref{thm:Cheeger-type}. By these results, we immediately obtain the $k$-way
Cheeger inequality and the $k$-way dual  Cheeger inequality for the  graph 1-Laplacian.

For a graph $(V,E)$, the $k$-way Cheeger constant \cite{Miclo08,LGT12}
\begin{equation}\label{eq:k-wayCheeger}
h_k:=\min_{\text{ disjoint } S_1, \cdots, S_k} \max_{1\le i\le k}\frac{|\partial S_i|}{\vol(S_i)}, \end{equation}
and the $k$-way dual Cheeger constant \cite{Liu15}
\begin{equation}\label{eq:k-waydualCheeger}
h^+_k:=\max\limits_{\text{ disjoint } (V_1,V_2),\ldots,(V_{2k-1},V_{2k})}\min\limits_{1\le i\le k}\frac{2|E(V_{2i-1},V_{2i})|}{\vol(V_{2i-1}\cup V_{2i})},
\end{equation}
are investigated  systematically. Both Theorem \ref{thm:Cheeger-type} and Theorem   \ref{thm:Cheeger-type2}  imply: 
 \begin{cor}\label{cor:1-Lap-Cheeger}
For an eigenpair $(\lambda_k,\vec x)$  of the graph 1-Laplacian \cite{Chang16} , where $\lambda_k$ is the $k$-th  minimax eigenvalue, 
$$  h_{m(x)}  \le \lambda_k\le h_k,\,\,\,\,\,\,\,\forall\, k, $$
in which $m(\vec x)$ is the number of nodal domains of $\vec x$.  

For an eigenpair  $(\lambda_k^+,\vec x)$  of the signless 1-Laplacian  \cite{CSZ16},  where  $\lambda_k^+$ is the $k$-th  minimax eigenvalue, 
$$  1-h^+_{m'(x)}\le \lambda^+_k\le 1-h^+_k,\,\,\,\,\,\,\,\forall\, k, $$
in which  $m'(\vec x)$ is the number of connected components of the  support set of $\vec x$. 
\end{cor}
\begin{remark}\label{example:K_5-1-Lap-c-h}
It is known that $\lambda_1=h_1$ and $\lambda_2=h_2$. However, $\lambda_3$ can be strictly smaller than $h_3$. In fact, for the complete graph $K_5$ on five vertices,  by Proposition 8 in \cite{CSZ17}, the  eigenvalues of the    1-Laplacian on $K_5$ are $0,\frac34,1$. Note that the clique covering number of $K_5$ is $1$, and then we can apply Theorem 1 in \cite{Zhang18} to derive that the multiplicity of the eigenvalue  $1$ is $2$. Thus,  $\lambda_4=\lambda_5=1$,  $\lambda_3=\lambda_2=\frac34$, $\lambda_1=0$.  But it is easy to check that $h_1=0$,  $h_2=\frac 34$ and   $h_3=h_4=h_5=1$. Altogether, we get  
$\lambda_3=\frac34<h_3=1$.  
\end{remark}

In addition, we have a result involving the  $p$-Laplacian and its signless version:
\begin{pro}\label{pro:coincide-bipart}
The spectrum of the $p$-Laplacian \cite{TudiscoHein18} and the spectrum of the signless $p$-Laplacian \cite{BS18} on a graph coincide if and only if the graph is bipartite.
\end{pro}
\begin{proof}It is known that the multiplicity of the eigenvalue 0 of the  $p$-Laplacian $\Delta_p$ equals  the number of connected components of the graph. 
And it is not difficult to check that the multiplicity of the eigenvalue 0 of the signless $p$-Laplacian $\Delta_p^+$ equals the number of bipartite components of the graph. Therefore, if the spectra of $\Delta_p$ and  $\Delta_p^+$ coincide, the graph must be bipartite. 

Conversely, for a bipartite graph with the vertex parts $V_1$ and $V_2$, we take  $\varphi:\R^n\to\R^n$ as $\varphi(\vec x)_i=x_i$ if $i\in V_1$ and $\varphi(\vec x)_i=-x_i$ if $i\in V_2$. Then $\varphi$ is an odd homeomorphism, and we can apply  Proposition \ref{pro:odd-homeomorphism} to obtain that the  spectra  (counting multiplicity) of $\Delta_p$ and  $\Delta_p^+$ coincide. 
\end{proof}

\subsection{Saddle point problems and von Neumann type min-max theorems }
\label{sec:saddle}

 We continue the study of the powerful min-max methods and saddle-point problems. As 
  von Neumann's minimax theorem has been applied widely and investigated  deeply in game theory, 
 it should be helpful to establish some extension theory for it. 
  
The  saddle point problem for a function $F:X\times Y\to \R$ is to find $(\vec x^*,\vec y^*)\in X\times Y$ such that $$\inf\limits_{\vec y\in Y}\sup\limits_{\vec x\in X}F(\vec x,\vec y)=\sup\limits_{\vec x\in X}\inf\limits_{\vec y\in Y}F(\vec x,\vec y) $$
in which $X$ and $Y$ are continua like convex sets,
while the discrete saddle point problem for $f:\A\times \B\to \R$ is to find $(A^*,B^*)\in\A\times\B$ satisfying
$$\min\limits_{B\in\B}\max\limits_{A\in\A} f(A,B)=\max\limits_{A\in\A} \min\limits_{B\in\B}f(A,B)$$
where $\A$ and $\B$ are finite set-families.
We will connect these two via extension approaches. 
The following result shows 
that the discrete saddle-point problem can be equivalently transformed to a continuous version by our extension method.


\begin{theorem}\label{thm:min-max-ABxy}
Given $\A\subset \power(V_1)\times\cdots\times\power(V_k)$ and $\B\subset \power(V_1)\times\cdots\times\power(V_l)$, suppose that  $(\A,\D_\A)$ and $(\B,\D_\B)$ are perfect domain pairs. If  $f:\A\times \B\to\R$ and $g:\A\times \B\to\R_+$ satisfy  
\begin{equation}\label{eq:f/g-min-max}
\min\limits_{B\in\B}\max\limits_{A\in\A} \frac{f(A,B)}{g(A,B)}=\max\limits_{A\in\A} \min\limits_{B\in\B}\frac{f(A,B)}{g(A,B)},
\end{equation}
 then we have
\begin{equation}\label{eq:AB-xy}
\min\limits_{B\in\B}\max\limits_{A\in\A} \frac{f(A,B)}{g(A,B)}=\inf\limits_{\vec y\in\D_\B}\sup\limits_{\vec x\in\D_\A}\frac{f^M(\vec x,\vec y)}{g^M(\vec x,\vec y)}=\max\limits_{A\in\A}\min\limits_{B\in\B} \frac{f(A,B)}{g(A,B)}=\sup\limits_{\vec x\in\D_\A}\inf\limits_{\vec y\in\D_\B}\frac{f^M(\vec x,\vec y)}{g^M(\vec x,\vec y)}.
\end{equation}
Moreover, $(A^*,B^*)$ is a saddle point of $f/g$ if and only if $(\vec 1_{A^*},\vec 1_{B^*})$ is a saddle point of $f^M/g^M$.
\end{theorem}
\begin{proof}
Note that  
for any $B\in\B$, $(\A\times \{B\},\D_\A\times\{\vec1_B\})$ is a  perfect domain pair.  
Then, we are able to  apply  Theorem \ref{thm:optimal-identity-fg} to get  $$\max\limits_{A\in\A}\frac{f(A,B)}{g(A,B)}=\sup\limits_{\vec x\in\D_\A}\frac{f^M(\vec x,\vec1_B)}{g^M(\vec x,\vec1_B)}$$ and thus
\begin{equation}\label{eq:AB-xy1}
\min\limits_{B\in\B}\max\limits_{A\in\A} \frac{f(A,B)}{g(A,B)}=\min\limits_{B\in\B}\sup\limits_{\vec x\in\D_\A}\frac{f^M(\vec x,\vec1_B)}{g^M(\vec x,\vec1_B)}\ge \inf\limits_{\vec y\in\D_\B}\sup\limits_{\vec x\in\D_\A}\frac{f^M(\vec x,\vec y)}{g^M(\vec x,\vec y)}.
 \end{equation}
Similarly, we have 
$$\max\limits_{A\in\A}\min\limits_{B\in\B} \frac{f(A,B)}{g(A,B)}\le  \sup\limits_{\vec x\in\D_\A}\inf\limits_{\vec y\in\D_\B}\frac{f^M(\vec x,\vec y)}{g^M(\vec x,\vec y)}.$$
And together with  the basic min-max inequality $$\inf\limits_{\vec y\in\D_\B}\sup\limits_{\vec x\in\D_\A}\frac{f^M(\vec x,\vec y)}{g^M(\vec x,\vec y)}\ge \sup\limits_{\vec x\in\D_\A}\inf\limits_{\vec y\in\D_\B}\frac{f^M(\vec x,\vec y)}{g^M(\vec x,\vec y)},$$ 
we obtain 
$$\min\limits_{B\in\B}\max\limits_{A\in\A} \frac{f(A,B)}{g(A,B)}\ge \inf\limits_{\vec y\in\D_\B}\sup\limits_{\vec x\in\D_\A}\frac{f^M(\vec x,\vec y)}{g^M(\vec x,\vec y)}\ge \sup\limits_{\vec x\in\D_\A}\inf\limits_{\vec y\in\D_\B}\frac{f^M(\vec x,\vec y)}{g^M(\vec x,\vec y)}\ge \max\limits_{A\in\A}\min\limits_{B\in\B} \frac{f(A,B)}{g(A,B)},$$
which confirms  \eqref{eq:AB-xy}. 
By the definition of saddle points, we get 
$\frac{f(A,B^*)}{g(A,B^*)}\le \frac{f(A^*,B^*)}{g(A^*,B^*)}\le \frac{f(A^*,B)}{g(A^*,B)}$, $\forall (A,B)\in\A\times\B$, and $\frac{f(A^*,B^*)}{g(A^*,B^*)}=\min\limits_{B\in\B}\max\limits_{A\in\A} \frac{f(A,B)}{g(A,B)}=\max\limits_{A\in\A} \min\limits_{B\in\B}\frac{f(A,B)}{g(A,B)}$. 
Then
\begin{align*}
\sup\limits_{\vec x\in\D_\A}\frac{f^M(\vec x,\vec 1_{B})}{g^M(\vec x,\vec 1_{B})}=\max\limits_{A\in\A}\frac{f^M(\vec 1_{A^*},\vec 1_{B})}{g^M(\vec 1_{A^*},\vec 1_{B})}&\le \frac{f^M(\vec 1_{A^*},\vec 1_{B^*})}{g^M(\vec 1_{A^*},\vec 1_{B^*})}
\\&\le \min\limits_{B\in\B}\frac{f^M(\vec 1_{A^*},\vec 1_{B})}{g^M(\vec 1_{A^*},\vec 1_{B})} = \inf\limits_{\vec y\in\D_\B}\frac{f^M(\vec 1_{A^*},\vec y)}{g^M(\vec 1_{A^*},\vec y)},
\end{align*}
and together with Eq.~\eqref{eq:AB-xy}, $(\vec 1_{A^*},\vec 1_{B^*})$ is a saddle point of $f^M/g^M$. The other direction is similar.
\end{proof}

\begin{remark}
We note that the condition  \eqref{eq:f/g-min-max} is equivalent to that $f/g$ possesses a saddle point. Indeed, for any finite families $\A$ and $\B$, and $h:\A\times \B\to\R$,  $\min\limits_{B\in\B}\max\limits_{A\in\A} h(A,B)=\max\limits_{A\in\A} \min\limits_{B\in\B}h(A,B)$ if and only if there exists $(A^*,B^*)\in\A\times \B$ s.t. $h(A,B^*)\le h(A^*,B^*)\le h(A^*,B)$, $\forall A\in\A,B\in\B$.
\end{remark}

\begin{remark}\label{remark:AB-xy}
 We can take $\A=(\power(V_1)\setminus\{\varnothing\})\times\cdots\times(\power(V_k)\setminus\{\varnothing\})$,  $\B=(\power(V_{k+1})\setminus\{\varnothing\})\times\cdots\times(\power(V_{k+l})\setminus\{\varnothing\})$ in Theorem \ref{thm:min-max-ABxy}, and then  $\D_\A=\R_+^{\#V_1+\cdots+\#V_k}$ and $\D_\B=\R_+^{\#V_{k+1}+\cdots+\#V_{k+l}}$. 
We can also take $\A=(\power_2(V_1)\setminus\{(\varnothing,\varnothing)\})\times\cdots\times(\power_2(V_k)\setminus\{(\varnothing,\varnothing)\})$,  $\B=(\power_2(V_{k+1})\setminus\{(\varnothing,\varnothing)\})\times\cdots\times(\power_2(V_{k+l})\setminus\{(\varnothing,\varnothing)\})$, $\D_\A=(\R^{\#V_1}\setminus\vec0)\times\cdots\times (\R^{\#V_k}\setminus\vec0)$,  $\D_\B=(\R^{\#V_{k+1}}\setminus\vec0)\times\cdots\times (\R^{\#V_{k+l}}\setminus\vec0)$, and   adopt the  multiple integral extension (Definition \ref{def:multiple-integral}) instead of the piecewise multilinear  extension (Definition \ref{defn:piece-multilinear}).  
\end{remark}

\begin{remark}
The converse of Theorem \ref{thm:min-max-ABxy} is false, i.e., 
$$\inf\limits_{\vec y\in\D_\B}\sup\limits_{\vec x\in\D_\A}\frac{f^M(\vec x,\vec y)}{g^M(\vec x,\vec y)}= \sup\limits_{\vec x\in\D_\A}\inf\limits_{\vec y\in\D_\B}\frac{f^M(\vec x,\vec y)}{g^M(\vec x,\vec y)}$$
doesn't imply
$$
\min\limits_{B\in\B}\max\limits_{A\in\A} \frac{f(A,B)}{g(A,B)}=\max\limits_{A\in\A} \min\limits_{B\in\B}\frac{f(A,B)}{g(A,B)}.$$
It means that there is some discrete saddle point problem (with no discrete solution)  possessing a continuous  solution in the sense of piecewise multilinear extension.  See the following examples.
\end{remark}

\begin{example}\label{ex:saddle-conterexample}
We continue the investigation of Example \ref{ex:spectral-radius}. Consider a path graph on three vertices, i.e.,  $V=\{1,2,3\}$ and $E=\{\{1,2\},\{2,3\}\}$. Denote its adjacency matrix by $W$. Note that $f^Q(\vec x,\vec y)=\vec x^\top W\vec y$ and $g^Q(\vec x,\vec y)=\vec x^\top \vec y$.  On  one hand, by the 
Krein-Rutman theorem\footnote{It is also known as Birkhoff–Varga formula or Collatz-Wielandt theorem.}(or by Theorem  \ref{thm:quadratic-saddle}), 
$$\inf\limits_{\vec x\in\R^3_+}\sup\limits_{\vec y\in\R^3_+}\frac{\vec x^\top W\vec y}{\vec x^\top \vec y}=\sup\limits_{\vec y\in\R^3_+}\inf\limits_{\vec x\in\R^3_+}\frac{\vec x^\top W\vec y}{\vec x^\top \vec y}=\lambda_{\max}(W)=\sqrt{2}.$$
 On the other hand, $\inf\limits_{A\subset V}\sup\limits_{B\subset V} \frac{\#E(A,B)}{\#(A\cap B)}=2>1=\sup\limits_{B\subset V}\inf\limits_{A\subset V} \frac{\#E(A,B)}{\#(A\cap B)}$.

\end{example}

\begin{example}\label{ex:saddle-Two-Person-Zero-Sum}
Given $V_1=\{1,\cdots,n\}$, $V_2=\{1,\cdots,m\}$, and a payoff matrix $C=(c_{ij})_{n\times m}$, let $f(A,B)=\sum_{i\in A,j\in B}c_{ij}$ and $g(A,B)=\#A\cdot\#B$, $\forall A\subset V_1$,  $B\subset V_2$. Then $f^Q(\vec x,\vec y)=\sum_{i=1}^n\sum_{j=1}^mc_{ij}x_iy_j$ and $g^Q(\vec x,\vec y)=(\sum_{i=1}^nx_i)(\sum_{j=1}^my_j)$. It follows from  von Neumann's minimax theorem that 
$$ \min\limits_{\sum_i p_i=1,p_i\ge 0}\max\limits_{\sum_i q_i=1,q_i\ge 0}\sum_{i=1}^n\sum_{j=1}^mc_{ij}p_iq_j= \max\limits_{\sum_i q_i=1,q_i\ge 0}\min\limits_{\sum_i p_i=1,p_i\ge 0}\sum_{i=1}^n\sum_{j=1}^mc_{ij}p_iq_j, $$
which can be reformulated as 
\begin{equation}\label{eq:two-person-zero-sum}
\inf\limits_{\vec x\in\R^n_+}\sup\limits_{\vec y\in\R^m_+}\frac{\vec x^\top C\vec y}{(\sum_{i=1}^nx_i)(\sum_{j=1}^my_j) }=\sup\limits_{\vec y\in\R^m_+}\inf\limits_{\vec x\in\R^n_+}\frac{\vec x^\top C\vec y}{(\sum_{i=1}^nx_i)(\sum_{j=1}^my_j) }.
\end{equation}
This equality can be obtained from  Theorem  \ref{thm:quadratic-saddle} directly.  But according to  the theory of  two-person zero-sum games, it is easy to give a payoff matrix $C$ such that $\min\limits_{A\subset V_1}\max\limits_{B\subset V_2} \frac{f(A,B)}{g(A,B)}>\max\limits_{B\subset V_2} \min\limits_{A\subset V_1}\frac{f(A,B)}{g(A,B)}$. 
\end{example}

\begin{remark}\label{remark:von-generalize-optimization}
We show that  Theorem \ref{inthm:min-max-ABxy} is also a generalization of Theorem B in \cite{JostZhang-PL}. 
Indeed, taking $\B=\{V\}$ as a singleton,  fixing  $\vec y$, and restricting $f$ and $f^M$ to their first components, we can verify Theorem B in \cite{JostZhang-PL}. 
\end{remark}

Theorems  \ref{thm:quadratic-saddle},  \ref{inthm:min-max-ABxy} and \ref{thm:min-max-ABxy} indicate that when one wants to solve a combinatorial saddle point problem, it is better to consider its continuous extension. The extended solution of the continuous saddle point problem is more  
flexible than the pure solution of the original discrete saddle point problem. This suggests a new  explanation  why one considers also mixed strategies instead of only  pure-strategy Nash equilibria.

\vspace{0.16cm}

\textbf{Piecewise bilinear extension and von Neumann's min-max theorem}

\vspace{0.16cm}

In order to show von Neumann’s convex-concave  min-max theorem in its full generality, we slightly enlarge the scope of the piecewise bilinear extension:

For $\A,\B\subset\power(V)$ or $\power_2(V)$, and $f:\A\times\B\to\R$, define $f^Q:\D_\A\times\D_\B\to\R$ as  a composition of Lov\'asz extensions in the following way:  $f^Q(\vec x,\vec y)=\tilde{f}^L_y(\vec x)$, with $\tilde{f}_y:\A\to\R$ defined as  $\tilde{f}_y(A):=f_A^L(\vec y)$, where $f_A:\B\to\R$ is defined by $f_A(B)=f(A,B)$. Here the Lov\'asz extension refers to the original version or the disjoint-pair version. 

\begin{remark}
Let the operator $\mathcal{L}_i$  be the (disjoint-pair) Lov\'asz extension acting on the $i$-th component, while we regard the other components as fixed parameters.

Precisely,  $\mathcal{L}_1f(x,B)$ is the (disjoint-pair) Lov\'asz extension of $A\mapsto f(A,B)$, for fixed $B\in\B$.

Similarly,  $\mathcal{L}_2f(A,y)$ is the (disjoint-pair) Lov\'asz extension of $B\mapsto f(A,B)$, for fixed $A\in\A$.

It is easy to check that $\mathcal{L}_1$ and $\mathcal{L}_2$ are independent of each other, and thus we have the commutative  diagram:
$$\xymatrix{f(A,B)\ar[r]^{\mathcal{L}_1}\ar[d]^{\mathcal{L}_2}  & \mathcal{L}_1f(x,B)\ar[d]^{\mathcal{L}_2} \\ \mathcal{L}_2f(A,y)\ar[r]^{\mathcal{L}_1} & f^Q(x,y)}$$
where 
$$f^Q(\vec x,\vec y)=\mathcal{L}_1\mathcal{L}_2 f(\vec x,\vec y)=\mathcal{L}_2\mathcal{L}_1 f(\vec x,\vec y).$$
Therefore, the restriction of $f^Q$ to each component is the (disjoint-pair) Lov\'asz extension of some function. Similarly, we can define a slight generalization of the  piecewise multilinear extension of $f:\A_1\times\cdots\times\A_k\to\R$ by
$$f^M(\vec x^1,\cdots,\vec x^k)=\mathcal{L}_1\mathcal{L}_2\cdots\mathcal{L}_k f (\vec x^1,\cdots,\vec x^k).$$
In summary, the piecewise multilinear extension can be seen as a composition of several (disjoint-pair)  Lov\'asz extensions. 
\end{remark}

\begin{pro}\label{pro:min-max-submodular}
Suppose that $f^Q(\vec x,\vec y)$ and  $g^Q(\vec x,\vec y)$ are  piecewise bilinear   extensions of $f,g:\A\times \B\to\R$ with $\A\subset \power(V_1)$ (or $\A\subset \power_2(V_1)$) and $\B\subset \power(V_2)$ (or $\B\subset \power_2(V_2)$), where $f$ satisfies   the following conditions:
\begin{itemize}
    \item $f$ is submodular  on its first  component;
    \item $f$ is  supermodular on its  second component.
\end{itemize}
 Then 
\begin{equation}\label{eq:min-max-submodular}
\min\limits_{\vec x\in \mathrm{cone}(\overline{C_\A})}\sup\limits_{\vec y\in \mathrm{cone}(C_\B)}\frac{f^Q(\vec x,\vec y)}{g^Q(\vec x,\vec y)}=\sup\limits_{\vec y\in \mathrm{cone}(C_\B)}\min\limits_{\vec x\in \mathrm{cone}(\overline{C_\A})}\frac{f^Q(\vec x,\vec y)}{g^Q(\vec x,\vec y)}
\end{equation}
but the discrete saddle point problem $\min\limits_{A\in\A}\max\limits_{B\in\B} \frac{f(A,B)}{g(A,B)}=\max\limits_{B\in\B} \min\limits_{A\in\A}\frac{f(A,B)}{g(A,B)}$ may have no  solution, where $C_\A\times C_\B$ is a bounded convex set such that $g^Q$ is bilinear on $  \overline{C_\A}\times  C_\B$  
with no zeros.
\end{pro}

\begin{proof}
Since $g$ is modular on its first component and $f$ is submodular on its first component, we obtain that $g^Q$ is a linear function of $\vec x$, and  $f^Q$ is convex with respect to  $\vec x$. Without loss of generality, we may assume $g^Q(\vec x,\vec y)>0$, $\forall (\vec x,\vec y)\in \overline{C_\A}\times  C_\B$. Consequently, for any $\vec x,\vec x'\in \overline{C_\A}$ and $0\le t\le 1$,
$$\frac{f^Q(t\vec x +(1-t)\vec x',\vec y)}{g^Q(t\vec x +(1-t)\vec x',\vec y)}\le  \frac{tf^Q(\vec x,\vec y)+(1-t)f^Q(\vec x',\vec y)}{tg^Q(\vec x,\vec y)+(1-t)g^Q(\vec x',\vec y)}\le \max\left\{\frac{f^Q(\vec x,\vec y)}{g^Q(\vec x,\vec y)},\frac{f^Q(\vec x',\vec y)}{g^Q(\vec x',\vec y)}\right\}$$
meaning that $f^Q/g^Q$ is quasi-convex on $\vec x\in \overline{C_\A}$. Similarly, $f^Q/g^Q$ is quasi-concave on $\vec y\in C_\B$. Also, it is clear that $f^Q/g^Q$ is continuous on $\overline{C_\A}\times  C_\B$. 
Sion's  min-max theorem (Theorem \ref{thm:strong-min-max})  yields $$\inf\limits_{\vec x\in \overline{C_\A}}\sup\limits_{\vec y\in C_\B}\frac{f^Q(\vec x,\vec y)}{g^Q(\vec x,\vec y)}=\sup\limits_{\vec y\in C_\B}\inf\limits_{\vec x\in \overline{C_\A}}\frac{f^Q(\vec x,\vec y)}{g^Q(\vec x,\vec y)}$$
which is equivalent to \eqref{eq:min-max-submodular} by the zero-homogeneity  of $f^Q/g^Q$.  

For the discrete saddle point problem,  one  can find many examples from  two-person zero-sum games (see Example \ref{ex:saddle-Two-Person-Zero-Sum}). 
\end{proof}

The assumption in  Proposition \ref{pro:min-max-submodular} is satisfied in most of the interesting cases. 
For example, if $\A=\power(V)$, we can always take $\mathrm{cone}(C_\A)$ as $X_{\sigma}:=\{\vec x\in \R^n_+:x_{\sigma(1)}<\cdots<x_{\sigma(n)}\}$ for any permutation $\sigma\in S_n$,  and if we further assume that $g$ is modular on its first component, then $\mathrm{cone}(C_\A)$ can be chosen as the first quadrant  $\R^n_+$, where $n=\#V$. 

Also, for $\A=\power_2(V)$, we can always take $\mathrm{cone}(C_\A)=\triangle(\vec x)$ (see \eqref{eq:piece-cone-strong-comonotonic}) for any given $\vec x\in\R^n\setminus\{\vec0\}$. 

\begin{theorem}[Sion's min-max  theorem \cite{Sion58}]\label{thm:strong-min-max}
Let $X$ be a compact convex set, and let $Y$ be a convex set. Let $F:X\times Y\to \R$ be such that:
\begin{itemize}
\item  $F$ is
 upper semi-continuous and quasi-concave on $Y$ for each $\vec x\in X$;
\item  $F$ is
 lower  semi-continuous and quasi-convex on $X$ for each $\vec y\in X$.
\end{itemize}
Then $\inf\limits_{x\in X}\sup\limits_{y\in Y}F(\vec x,\vec y)=\sup\limits_{y\in Y}\inf\limits_{x\in X}F(\vec x,\vec y)$.
\end{theorem}



\begin{proof}[Proof of Theorem  \ref{thm:quadratic-saddle}]
Theorem  \ref{thm:quadratic-saddle} under the condition (a) is a direct consequence of Theorem  \ref{thm:min-max-ABxy}. 

We prove Theorem \ref{thm:quadratic-saddle} under the condition (b) by employing   Proposition \ref{pro:min-max-submodular}. Since $g$ is modular on each component, $g$ must be bilinear on $\R^n\times \R^m$. It follows from  $g\ge 0$ on $\power(V_1)\times \power(V_2)$ that $g^Q\ge 0$ on $\R^n_{\ge0}\times \R^m_{\ge0}$. Precisely, $g^Q(\vec x,\vec y)=\sum_{i=1}^n\sum_{j=1}^mg(i,j) x_iy_j$, where $g(i,j):=g(\{i\},\{j\})\ge0$.  For any $\vec x\in\R^n_{\ge0}\setminus\{\vec 0\}$, there exists $i\in V_1$ such that $x_i>0$. By the assumption that  $g(\{i\},V_2)>0$, there exists $j\in V_2$ satisfying $g(i,j)>0$.  Accordingly, for any $\vec y\in \R^m_+$,  we have $g^Q(\vec x,\vec y)\ge g(i,j)x_iy_j>0$. Therefore, $g^Q$ is positive on $(\R^n_{\ge0}\setminus\{\vec 0\})\times \R^m_+$, and $f^Q/g^Q$ is well-defined and continuous on $(\R^n_{\ge0}\setminus\{\vec 0\})\times \R^m_+$.  
Hence, by taking  $C_\A=\{\vec x\in\R^n_+:x_1+\cdots+x_n=1\}$ and $C_\B=\{\vec x\in\R^m_+:x_1+\cdots+x_m=1\}$, we can apply Proposition \ref{pro:min-max-submodular} to derive \eqref{eq:continuous-ex}. 
\end{proof}

 {
\textbf{A general min-max relation}

\vspace{0.06cm} 

Recall that a function is {\sl quasi-linear} if it is quasi-convex and quasi-concave. As an extension of both Theorem \ref{thm:tilde-H-f-M} and  Theorem \ref{thm:min-max-ABxy}, we present the following general min-max relation.

 \begin{theorem}\label{thm:tilde-H-f-mm}
Let $H:\R^n_{\ge0}\setminus\{\vec 0\}\to\R\cup\{\pm\infty\}$ be a  zero-homogeneous and \textbf{quasi-linear}  function. For any functions $f_1,\cdots,f_n:\A\times\B\to \R_{\ge0}$, we have the min-max inequality:
 \begin{align*}
    \min\limits_{A\in \A}\max\limits_{B\in \B}H(f_1(A,B),\cdots,f_n(A,B))&\ge\inf\limits_{\vec x\in \D_\A}\sup\limits_{\vec y\in \D_\B} H(f^M_1(\vec x,\vec y),\cdots,f^M_n(\vec x,\vec y))\\&\ge\sup\limits_{\vec y\in \D_\B}\inf\limits_{\vec x\in \D_\A} H(f^M_1(\vec x,\vec y),\cdots,f^M_n(\vec x,\vec y))
\\&\ge  \max\limits_{B\in \B}  \min\limits_{A\in \A}H(f_1(A,B),\cdots,f_n(A,B))
 \end{align*}
 where $(\A,\D_\A)$ and $(\B,\D_\B)$ are perfect domain pairs. 
\end{theorem}
In summary, when we want to solve a combinatorial saddle point problem, it is better to work on the piecewise  multilinear  extension. In fact, suppose that the corresponding  continuous saddle point problem has a solution, which we will call 
a {\sl weak solution}  of the original problem for convenience. Then, we can return to the original  combinatorial saddle-point problem, as in the following cases: 
 \begin{itemize}
\item[Case 1.] 
If there is a solution to the original combinatorial saddle point problem, 
then we can construct  such a solution based on a weak solution,  which is of course good news.
        \item[Case 2.] 
    If the original  combinatorial saddle-point problem  has no solution,  then we can accept  a weak  solution because it makes sense on its own.    (For  example, this suggests a new  explanation  why one considers also mixed strategies instead of only  pure-strategy Nash equilibria in a two-person-zero-sum game.)
 \end{itemize}
}        

\section{Applications in   various   areas}\label{sec:application}

\subsection{Tur\'an problem and Motzkin-Straus theorem}\label{sec:Turan}
The classical Tur\'an theorem (weak version) states that for any $K_{\omega+1}$-free graph $G=(V,E)$,
\begin{equation}\label{eq:Turan}
\#E\le (1-\frac{1}{\omega})\frac{(\#V)^2}{2},
\end{equation}
where $\omega$ is  the maximal clique number of $G$. It has  many combinatorial proofs, from which the extremal graph theory started its history.

\begin{lemma}\label{lemma:simple-}
Let $\mathbf{f},\mathbf{g}:\R^n\to\R$ be smooth 
functions such that  
$\mathbf{g}$ is positive  on $\R^n_{\ge0}\setminus\{\vec0\}$. 
For a maximizer (resp. minimizer) $\vec x$  of
$\frac{\mathbf{f}}{\mathbf{g}}|_{\R^n_{\ge0}\setminus\{\vec0\}}$ (if it exists),  let $\vec v$ be such that $\vec x+\vec v\in \R^n_{\ge0}\setminus\{\vec0\}$, $\mathrm{supp}(\vec v)\subset \mathrm{supp}(\vec x)$, $\R\ni t\mapsto \mathbf{g}(\vec x+t\vec v)$ is constant, and $\frac{\partial}{\partial y_i}\frac{\partial}{\partial y_j}\mathbf{f}(\vec y)=0$,  $\forall i,j\in\mathrm{supp}(\vec v)$, $\forall\vec y\in\R^n$.  If we further assume that $\mathbf{f}$ is real analytic, then $\vec x+\vec v$ is also a maximizer (resp. minimizer) of  $\frac{\mathbf{f}}{\mathbf{g}}|_{\R^n_{\ge0}\setminus\{\vec0\}}$.
\end{lemma}

\begin{proof}
Claim:  Let $\vec x$ be a critical point of  $\frac{\mathbf{f}}{\mathbf{g}}|_{\R^n_{\ge0}\setminus\{\vec0\}}$ 
and let $\vec v\in\R^n$ be such that $\langle\nabla\mathbf{g}(\vec x),\vec v \rangle=0$ and $\mathrm{supp}(\vec v)\subset \mathrm{supp}(\vec x)$, then $\langle\nabla\mathbf{f}(\vec x),\vec v \rangle=0$. 

Proof of the claim: By the assumption,  $\mathrm{supp}(\vec x)=\{i\in\{1,\cdots,n\}:x_i>0\}\ne\varnothing$. For any $i\in \mathrm{supp}(\vec x)$, we have $\frac{\partial}{\partial x_i}\frac{\mathbf{f}(\vec x)}{\mathbf{g}(\vec x)}=0$, $\forall i\in \mathrm{supp}(\vec x)$. Thus, $\frac{\partial}{\partial x_i}\mathbf{f}(\vec x)=\frac{\mathbf{f}(\vec x)}{\mathbf{g}(\vec x)}\frac{\partial}{\partial x_i}\mathbf{g}(\vec x)$ for  any $i\in \mathrm{supp}(\vec x)$. By the condition that $v_i=0$ whenever $i\not\in \mathrm{supp}(\vec x)$, we have 
\begin{align*}
\langle\nabla\mathbf{f}(\vec x),\vec v \rangle&=\sum_{i=1}^nv_i\frac{\partial}{\partial x_i}\mathbf{f}(\vec x)=\sum_{i\in \mathrm{supp}(\vec x)}v_i\frac{\partial}{\partial x_i}\mathbf{f}(\vec x)
\\&=\sum_{i\in \mathrm{supp}(\vec x)}v_i\frac{\mathbf{f}(\vec x)}{\mathbf{g}(\vec x)}\frac{\partial}{\partial x_i}\mathbf{g}(\vec x)=\frac{\mathbf{f}(\vec x)}{\mathbf{g}(\vec x)}\langle\nabla\mathbf{g}(\vec x),\vec v \rangle=0.
\end{align*}

Now we prove the lemma. 
It follows from  $\mathbf{g}(\vec x+t\vec v)=\mathbf{g}(\vec x)$ $\forall t\in\R$ that $\langle\nabla\mathbf{g}(\vec x),\vec v \rangle=0$, and thus by the above claim, we have $\langle\nabla\mathbf{f}(\vec x),\vec v \rangle=0$. Since $\mathbf{f}$ is a real analytic function, $t\mapsto \mathbf{f}(\vec x+t\vec v)$ must be real analytic. Note that $\frac{d}{dt}|_{t=0}\mathbf{f}(\vec x+t\vec v)=\langle\nabla\mathbf{f}(\vec x),\vec v \rangle=0$, and for any $k\ge 2$, 
\begin{align*}
\frac{d^k}{dt^k}|_{t=0}\mathbf{f}(\vec x+t\vec v)&=\sum_{i_1,\cdots,i_k=1}^nv_{i_1}\cdots v_{i_k}\frac{\partial^k f(\vec x)}{\partial x_{i_1}\cdots\partial x_{i_k}}
\\&=\sum_{i_1,\cdots,i_k\in\mathrm{supp}(\vec v)}v_{i_1}\cdots v_{i_k}\frac{\partial^k f(\vec x)}{\partial x_{i_1}\cdots\partial x_{i_k}} = 0
\end{align*}
where the last equality is due to the condition that $\partial_i\partial_j\mathbf{f}=0$,  $\forall i,j\in\mathrm{supp}(\vec v)$.

Therefore, the real analytic function $ t\mapsto \mathbf{f}(\vec x+t\vec v)$ is constant. This implies that $\mathbf{f}(\vec x+\vec v)=\mathbf{f}(\vec x)$, and hence $\frac{\mathbf{f}(\vec x+\vec v)}{\mathbf{g}(\vec x+\vec v)}=\frac{\mathbf{f}(\vec x)}{\mathbf{g}(\vec x)}$, meaning that $\vec x+\vec v$ is also a maximizer  of  $\frac{\mathbf{f}}{\mathbf{g}}|_{\R^n_{\ge0}\setminus\{\vec0\}}$. 

The case of minimizer is similar.
\end{proof}

\begin{pro}\label{pro:Q-inequality} For $f,g:\power(V)^2\to \R_+$, there holds
\begin{equation}\label{eq:quodratic-extension}
    \max\limits_{A}\frac{f(A,A)}{g(A,A)}\le \max\limits_{\vec x\in \R^V_{\ge0}} \frac{f^Q(\vec x,\vec x)}{g^Q(\vec x,\vec x)}\le \max\limits_{A\subset B}\frac{f(A,B)}{g(A,B)}=\max\limits_{\R^V_{\ge0}\ni\vec x,\vec y\text{ comonotonic}}\frac{f^Q(\vec x,\vec y)}{g^Q(\vec x,\vec y)}.
\end{equation}

Now we further assume that  $g(A,B)=\tilde{g}(A)\tilde{g}(B)$ for some modular function $\tilde{g}:\power(V)\to\R$, and $f$ is  modular on both  components. Suppose that  there exists $C>0$ satisfying    $f(\{i\},\{i\})=C(\tilde{g}^2(\{i\})-\tilde{g}(\{i\}))$ and $f(\{i\},\{j\})=Cg(\{i\},\{j\})$ whenever $f(\{i\},\{j\})>0$.  Then the left inequality in \eqref{eq:quodratic-extension} is indeed an equality. 
\end{pro}

\begin{proof}
The inequality \eqref{eq:quodratic-extension} is a direct consequence of 
Theorem \ref{thm:Turan-general}.  
For the equality case, 
we set $C=1$, because otherwise we can use $Cf$ instead of $f$. Since $g(A,B)=\tilde{g}(A)\tilde{g}(B)$ and $\tilde{g}$ is modular, we have $g^Q_\Delta(\vec x)=\langle \vec u,\vec x\rangle^2$, where   $\vec u=(\tilde{g}(\{1\}),\cdots,\tilde{g}(\{n\}))\in\R^n_+$. By the assumption that  $f(A,B)$ is modular on each component, its  piecewise bilinear extension $f^Q$ must be multilinear, and thus  $f^Q_\triangle(\vec x)=\vec x^T M \vec x$ where $M=(f(\{i\},\{j\}))_{n\times n}$. 

For any $\vec v$ satisfying $\langle \vec u,\vec v\rangle=0$, $g^Q_\Delta(\vec x+\vec v)=g^Q_\Delta(\vec x)$. 
Let $\vec x$ be a maximizer of $f_\triangle^Q/g_\triangle^Q$ on $\R^n_{\ge 0}\setminus\{\vec0\}$. If $f(i,j):=f(\{i\},\{j\})= 0$ and $x_ix_j>0$ for some $i\ne j$, taking $\vec v$ defined as $v_i=-x_i$,  $v_j=x_i\frac{u_i}{u_j}$ and $v_l=0$ for $l\ne i,j$, 
then Lemma \ref{lemma:simple-} can be applied to deduce  that 
$\vec x+\vec v$ is also a maximizer of $f_\triangle^Q/g_\triangle^Q$ on $\R^n_{\ge 0}\setminus\{\vec0\}$.  Taking $\vec x:=\vec x+\vec v$ and repeating the process, we finally obtain a subset $A\subset V$  satisfying $\mathrm{supp}(\vec x)=A$ and  $f(i,j)>0$ for $i\ne j$ in $ A$.  Therefore, $f^Q_\Delta(\vec x)/g^Q_\Delta(\vec x)=\frac{\vec x^TM\vec x}{(\vec x^T\vec u)^2}=1-\frac{\sum_{i\in A} u_ix_i^2}{(\sum_{i\in A} x_iu_i)^2}\le 1-\frac{1}{ \sum_{i\in A} u_i}$ and the equality holds if and only if $x_i=\text{Const}$ for $i\in A$. In consequence, $\vec 1_A$ is a maximizer of $f^Q_\Delta/g^Q_\Delta$. The proof is completed. 
\end{proof}

According to Proposition \ref{pro:Q-inequality}  
and Table \ref{tab:double-Lov}, 
we get the identity
\begin{equation}\label{eq:M-S-weak}
\max\limits_{A\in\power(V)\setminus\{\varnothing\}}\frac{\#E(A,A)}{(\#A)^2}=\sup\limits_{\vec x\in\R^n_+}\frac{2\sum_{i\sim j}x_ix_j}{\|\vec x\|_1^2}=\max\limits_{\vec x\ne 0} \frac{\sum\limits_{i,j\in V\text{ s.t. }\{i,j\}\in E}x_ix_j}{\|\vec x\|_1^2}.
\end{equation}
It is very interesting that  \eqref{eq:M-S-weak} reduces to the Motzkin-Straus theorem immediately by the (weak) Tur\'an theorem \eqref{eq:Turan}. In fact, applying \eqref{eq:Turan} to the subgraph $G|_A$ induced by $A$ implies that $\frac{\#E(A,A)}{(\#A)^2}$ achieves its maxima at some maximum clique, which means $\max\limits_{A\in\power(V)\setminus\{\varnothing\}}\frac{\#E(A,A)}{(\#A)^2}=\frac{2{\omega\choose 2}}{\omega^2}=(1-\frac1\omega)$. In consequence, the original Motzkin-Straus theorem $$ \max\limits_{x_i\ge0,\sum_i x_i=1}2\sum_{\{i,j\}\in E}x_ix_j=1-\frac1\omega$$ is proved by virtue of \eqref{eq:M-S-weak}. 

In addition, since the maximum clique number of $(V,E)$ equals  the independence number of the complement graph $(V,E^c)$, one can see that the Motzkin-Straus identity
is equivalent to the following representation of the  independence number:
\begin{equation}\label{eq:M-S-independence}
\alpha(G)=\max\limits_{\vec x\in\R^V\setminus\{\vec 0\}}\frac{\|\vec x\|_1^2}{\|\vec x\|_1^2-2\sum_{ij\in E^c}x_ix_j}.
\end{equation}

Similarly, for a simple graph $(V,E)$,  let $H$ be the collection of all $k$-cliques in $(V,E)$. Then we obtain a special $k$-uniform hypergraph $(V,H)$, and   its Lagrangian 
satisfies\footnote{This equality might be known to experts, although we didn't find a reference. } 
 $$\lambda(H):=\sup\limits_{\vec x\ne\vec0}\frac{ \sum_{\{i_1,\cdots,i_k\}\in H}x_{i_1}\cdots x_{i_k}}{\|\vec x\|_1^k} =  
 \max\limits_{U\subset V,\,U\ne\varnothing}\frac{ \#\{\{i_1,\cdots,i_k\}\in H:\{i_1,\cdots,i_k\}\subset U\}}{(\#U)^k}.$$
We employ Theorem \ref{thm:Turan-general} to give a proof here. First, by Theorem \ref{thm:Turan-general}, the LHS is larger than or equal to the RHS. To show the converse, we let $F(\vec x)=\sum_{\text{clique }\{i_1,\cdots,i_k\}}x_{i_1}\cdots x_{i_k}$ and $G(\vec x)=(x_1+\cdots+x_n)^k$ for $\vec x=(x_1,\cdots,x_n)\in \R^n_{\ge 0}\setminus\{\vec 0\}$. Let $\vec x$ be a maximizer of $F/G$ on $\R^n_{\ge 0}\setminus\{\vec 0\}$. 
Note that $G(\vec x+\vec v)=G(\vec x)$ for any $\vec v$ with $\langle\vec 1,\vec v\rangle=0$. 
If there exist $i,j\in \mathrm{supp}(\vec x)$ such that there is no clique $\{i_1,\cdots,i_k\}$ containing both $i$ and $j$, then taking $\vec v$ defined by $v_i=-x_i$, $v_j=x_i$ and $v_l=0$ for $l\ne i,j$, we can apply Lemma \ref{lemma:simple-} to derive that 
$\vec x+\vec v$ is also a maximizer of $F/G$. 
Substituting $\vec x:=\vec x+\vec v$, we have $ x_i=0$. Repeating the process, we can finally obtain that there is a clique $U$  
such that $x_i=0$ whenever $i\in V\setminus U$, and $x_i>0$ whenever $i\in U$. Since $H$ collects all the $k$-cliques of a graph, $H|_U:=\{h\in H:h\subset U\}$  collects all the $k$-cliques in $U$. Putting everything together, 
it is easy to see that Maclaurin's  inequality  can be applied  to get $x_i=\text{const}$ for $i\in U$. Consequently, $\vec 1_U$ is a maximizer of $F/G$. The proof is completed.

Similar to the above discussion,  we can apply Lemma \ref{lemma:simple-}, the generalized mean inequality and Maclaurin's  inequality, to derive the following equality
$$\sup\limits_{\vec x\ne\vec0}\frac{ \sum_{\{i_1,\cdots,i_k\}\in H}x_{i_1}\cdots x_{i_k}-\tau \sum_{i\in V}|x_i|^k}{\|\vec x\|_1^k} =  
 \max\limits_{U\subset V,\,U\ne\varnothing}\frac{ \#\{h\in H:h\subset U\}-\tau\#U}{(\#U)^k}$$
where $\tau$ is a nonnegative real number. 

\subsection{$p$-Laplacians generated by Lov\'asz extension}
\label{sec:p-Lap}


Given a family    $\{w_e\}_{e\in E}$ of positive numbers, and two families of functions
\begin{equation*}
    \{f_e:\power(V)\to\R_{\ge0}\}_{e\in E}\qquad\text{and}\qquad\{g_e :\power(V)\to\R_{\ge0}\}_{e\in E}
\end{equation*}satisfying 
 $f_e(\emptyset)=f_e(V)=0=g_e(\emptyset)$ and $g_e(A)=\sum_{i\in A}g_e(\{i\})$ for all $A\subset V$,  let
 \begin{equation*}
    f(A):=\sum_{e\in E} w_e f_e(A)\qquad \text{and}\qquad g(A):=\sum_{e\in E} g_e(A).
 \end{equation*}
 We define the \textbf{Cheeger constant} as
\begin{equation*}h:=\min_{A\subset V,\,A\not\in\{\emptyset,V\}}\frac{f(A)}{\min \{g(A),g(V\setminus A)\}}.\end{equation*}

The following $p$-homogeneous eigenvalue problem
 \begin{equation}\label{eq:eigenvalue-problem}
\vec 0\in\nabla \sum_{e\in E} w_{e}(f_e^L(\vec x))^p-\lambda \nabla \sum_{e\in E} g_e^L(|\vec x|^p)
 \end{equation}
can unify many analogs of  $p$-Laplacian eigenvalue problem for graphs or hypergraphs (see Remark \ref{remark:p-Lap-} for details), where each $w_{e}$ is a positive constant,  and $|\vec x|^p=(|x_1|^p,\cdots,|x_n|^p)$. 


According to Proposition \ref{pro:mountain-pass},  the second eigenvalue of the eigenvalue problem \eqref{eq:eigenvalue-problem} 
can be characterized as
 \begin{equation}\label{eq:lambda_2p-Laplacian}
     \lambda=\inf\limits_{\vec x\text{ nonconstant}}\frac{\sum_{e\in E} w_{e}(f_e^L(\vec x))^p}{\min\limits_{c\in\R}\sum_{e\in E} g_e^L(|\vec x-c\vec1|^p)}.
 \end{equation}

\begin{example}\label{example:p-Lap}
Given a simple graph $G=(V,E)$, and taking  $w_{e}=1$,  consider the discrete functions $f_e,g_e:\power(V)\to\R$ defined by $$f_e(A)=\begin{cases} 1, &\text{ if } e\text{ has one end point in }A\text{ and the other in }V\setminus A,\\
0, &\text{ otherwise},
\end{cases}$$
and $g_e(A)=\#(e\cap A)$. By the original Lov\'asz extension, we have $f^L_e(\vec x)=|x_i-x_j|$ and $g^L_e(\vec x)=x_i+x_j$, where $\{i,j\}=e\in E$. Then
$$ \sum_{e\in E} (f_e^L(\vec x))^p=\sum_{\{i,j\}\in E}|x_i-x_j|^p\;\text{ and }\;  \sum_{e\in E} g_e^L(|\vec x|^p)=\sum_{\{i,j\}\in E}(|x_i|^p+|x_j|^p)=\sum_{i\in V}\deg(i)|x_i|^p.$$
So in this case,  \eqref{eq:eigenvalue-problem} reduces to   the eigenvalue problem of the normalized $p$-Laplacian on a graph. 
\end{example}

We may call $\vec x\mapsto \nabla \sum_{e\in E}  w_{e}(f_e^L(\vec x))^p$ the {\sl Lov\'asz
  $p$-Laplacian} induced by $ \{f_e:\power(V)\to\R_{\ge0}\}_{e\in E}$, because it is based on Lov\'asz extension, and it generalizes  the graph  $p$-Laplacian. Then, we may call  \eqref{eq:eigenvalue-problem} the eigenvalue problem of  Lov\'asz
  $p$-Laplacians for $ \{f_e\}_{e\in E}$ and $ \{g_e\}_{e\in E}$. Assume that $g(\{i\}):=\sum_{e\in E}g_e(\{i\})>0$ for any $i\in V$,  then the following Cheeger inequality holds. 
  \begin{theorem}\label{thm:Cheeger-main}
Under the above setting, we  have
\begin{equation}\label{eq:Cheeger-right}
(\frac{2}{c})^{p-1}\frac{h^p}{p^p}\le  \lambda\le 2^{p-1}Ch,
\end{equation}
where $ C:=\max\limits_{e,\,A}f_e(A)^{p-1}$ and 
$$c=\max\limits_{i\in V}\frac{\sum_{e\in E_i}w_{e}}{g(i)},\text{ and }E_i:=\{e\in E:\exists S\subset V \text{ s.t. }f_e(S\setminus \{i\})\ne f_e(S) \}$$
\end{theorem}

\begin{proof}
Let $A\in\power(V)\setminus\{\varnothing,V\}$ be such that
\begin{equation*}h=\frac{f(A)}{\min \{g(A),g(V\setminus A)\}}.\end{equation*}
By taking the nonconstant vector $\vec x=\vec 1_A$ in \eqref{eq:lambda_2p-Laplacian}, we get
\begin{equation*}\lambda\le \frac{\sum_{e\in E} w_{e}(f_e^L(\vec 1_A))^p}{\min\limits_{c\in\R}\sum_{e\in E} g_e^L(|\vec 1_A-c\vec1|^p)}.\end{equation*}

Since $g_e$ is modular and $g=\sum_{e\in E}g_e$, we have \begin{equation*}\sum_{e\in E} g_e^L(|\vec 1_A-c\vec1|^p)=\sum_{i\in V}g(i)|(\vec 1_A)_i-c|^p\end{equation*}
and \begin{equation*}\min\limits_{c\in\R}\sum_{i\in V}g(i)|(\vec 1_A)_i-c|^p=\frac{g(A)g(V\setminus A)}{(g(A)^{\frac{1}{p-1}}+g(V\setminus A)^{\frac{1}{p-1}})^{p-1}}.\end{equation*} Therefore, 
\begin{align*}
    \lambda&\le \frac{\sum_{e\in E} w_{e}(f_e^L(\vec1_A))^p}{g(A)g(V\setminus A)\big/(g(A)^{\frac{1}{p-1}}+g(V\setminus A)^{\frac{1}{p-1}})^{p-1} }\\
    &=\sum_{e\in E} w_{e}(f_e(A))^p\left(\sqrt[p-1]{\frac{1}{g(A)}}+\sqrt[p-1]{\frac{1}{g(V\setminus A)}}\right)^{p-1}\\
    &\le \max\limits_{e,A}f_e(A)^{p-1}\sum_{e\in E} w_{e}f_e(A) 2^{p-1}\frac{1}{\min\{g(A),g(V\setminus A)\}}\\ &:=\frac{2^{p-1}C
f(A)}{\min \{g(A),g(V\setminus A)\}} =2^{p-1}Ch.
\end{align*}

Let's move on to   the lower bound of $\lambda$, i.e., the left hand side of the inequality  \eqref{eq:Cheeger-right}. For simplicity, we identify a vector $\vec x\in\R^V$ with the function $\vec x: V\to\ \R$. Put $\deg(j)=\sum_{e\in E_j}w_{e}$ for $j\in V$. 
Then $\deg(j)\le cg(j)$, $\forall j$. 

Below we should adopt a new form of the original  Lov\'asz extension.

The Lov\'asz extension  \eqref{eq:Lovasz-PL} can be re-written as
\begin{equation}\label{eq:Lovasz-PL1}
f^L(\vec x)=\sum_{i=0}^{n-1} (x_{(i+1)}-x_{(i)})f(V_{i}(\vec x)),
\end{equation}
in which  $x_{(0)}:=0$ and $V_0(\vec x):=V$. 
Alternatively, we can also write
\begin{equation*} 
f^L(\vec x)=\sum_{i=0}^{k-1} (x_{[i+1]}-x_{[i]})f(V_{[i]}(\vec x)),
\end{equation*}
where \begin{equation*}k:=1+\sum_{i=0}^{n-1}\mathrm{sign}(|f(V_{i}(\vec x))-f(V_{i+1}(\vec x))|)\end{equation*} and $\{[1],\ldots,[k]\}\subset V$ satisfy:
\begin{itemize}
    \item $x_{[1]}<\ldots<x_{[k]}$, $x_{[0]}=0$, 
    \item $V_{[i]}(\vec x):=\{j\in V: x_j>x_{[i]}\}$ for $i\ge 1$, $V_{[0]}(\vec x):=V$,
    \item $f(V_{[i]}(\vec x))\ne f(V_{[i+1]}(\vec x))$ for all $i=0,\ldots,k-1$.
\end{itemize}
We call the set  $\{[1],\ldots,[k]\}$ a \textbf{simple index set} for $f$ at $\vec x$. For a fixed vector $\vec x$, let $\mathrm{si}(f)=\{i:f(V_{[i]}(\vec x))\ne 0\}$. It is clear that 
\begin{equation}\label{eq:Lovasz-PL2}
f^L(\vec x)=\sum_{i\in\mathrm{si}(f)} (x_{[i+1]}-x_{[i]})f(V_{[i]}(\vec x)).
\end{equation}
Given $p\ge 1$, for $\vec x\in \R_{\ge0}^n$, using \eqref{eq:Lovasz-PL2}, and noting that $f_e$ is nonnegative, we have
\begin{align}
&\sum\limits_{e\in E}w_{e}f^L_e(\vec x^p)
\\=&\sum\limits_{e\in E}w_{e}\sum\limits_{i\in\mathrm{si}(f_e)} (x_{[i+1]_e}^p-x_{[i]_e}^p)f_e(V_{[i]_e}(\vec x^p))\notag 
\\\le &\sum\limits_{e\in E}w_{e}\sum\limits_{i\in\mathrm{si}(f_e)}p(x_{[i+1]_e}-x_{[i]_e})(\frac{x_{[i+1]_e}^p+x_{[i]_e}^p}{2})^{\frac{1}{p'}}f_e(V_{[i]_e}(\vec x)) \label{eq:first-p-Lap}
\\\le &\frac{p}{2^{\frac{1}{p'}}}\left(\sum\limits_{e\in E}w_{e}\sum\limits_{i\in\mathrm{si}(f_e)}f_e(V_{[i]_e}(\vec x))^p(x_{[i+1]_e}-x_{[i]_e})^p\right)^{\frac1p}\label{eq:second-p-Holder}
\\ &~\times \left(\sum\limits_{e\in E}w_{e}\sum\limits_{i\in\mathrm{si}(f_e)}(x_{[i+1]_e}^p+x_{[i]_e}^p)\right)^{\frac{1}{p'}}\notag 
\\\le &\frac{p}{2^{\frac{1}{p'}}}\left(\sum\limits_{e\in E}w_{e}\left(\sum\limits_{i\in\mathrm{si}(f_e)}f_e(V_{[i]_e}(\vec x))(x_{[i+1]_e}-x_{[i]_e})\right)^p\right)^{\frac1p}\left(\sum_{j\in V}\widetilde{\deg}(j)x_{j}^p\right)^{\frac{1}{p'}}\notag
\\\le &(\frac{c}{2})^{\frac{1}{p'}}p\left(\sum\limits_{e\in E}w_{e}(f^L_e(\vec x))^p\right)^{\frac1p}\left(\sum_{j\in V}g(j)x_{j}^p\right)^{\frac{1}{p'}}\label{eq:third-deg}
\end{align}
where  $p'$ is the H\"older conjugate of $p$, i.e., $\frac{1}{p}+\frac{1}{p'}=1$, and   
$\{[1]_e,\ldots,[i]_e,\ldots\}$ is a simple index set for $f_e$ at $\vec x$, and
\begin{equation*}\widetilde{\deg}(j):=\sum_{e\in \widetilde{E}_j}w_{e},\text{ and }\widetilde{E}_j:=\{e\in E:j=[i]_e \text{ or }j=[i+1]_e\text{ for some }i\in \mathrm{si}(f_e)\}.\end{equation*} 
The first inequality \eqref{eq:first-p-Lap} uses an inequality in \cite{Amghibech},  the second inequality  \eqref{eq:second-p-Holder} uses  H\"older's inequality, and 
 the last inequality \eqref{eq:third-deg} is according to the  fact that
 \begin{equation}\label{eq:counting-deg}\sum_{j\in V: x_j\ge t}\widetilde{\deg}(j)\le \sum_{j\in V: x_j\ge t}\deg(j)\le c\sum_{j\in V: x_j\ge t}g(j) \text{ for any }t\in\R.\end{equation}
In fact, for any $e\in \widetilde{E}_j$, we may suppose that $j=[i]_e$, and then  $f_e(V_{[i-1]_e}(\vec x))\ne f_e(V_{[i]_e}(\vec x))$. Clearly, if $j'\ne j$ and $x_{[i]_e}\le x_{j'}< x_{[i+1]_e}$, then  $e\not\in \widetilde{E}_{j'}$ because $j'\ne [i']_e$ for any $i'$. That is, $e$ is counted exactly   once in  $\widetilde{E}_{j}$ over all $j$ with $x_j\in[x_{[i]_e},x_{[i+1]_e})$. We shall prove that there exists $j'$ with $x_{[i]_e}\le x_{j'}< x_{[i+1]_e}$ such that $e\in E_{j'}$. 
Suppose the contrary, that for any $j'$ with $x_{[i]_e}\le x_{j'}< x_{[i+1]_e}$,  $e\not\in E_{j'}$. Denote by $\{j_1,\cdots,j_l\}=\{j:x_{[i]_e}\le x_{j}< x_{[i+1]_e}\}=V_{[i-1]_e}\setminus V_{[i]_e}$.  It follows from $e\not\in E_{j_1}\cup\cdots\cup E_{j_l}$ that  $f(V_{[i-1]_e}(\vec x))=f(V_{[i-1]_e}(\vec x)\setminus \{j_1\})=\cdots =f(V_{[i-1]_e}(\vec x)\setminus \{j_1,\cdots,j_l\})=f(V_{[i]_e}(\vec x))$, which is a contradiction. In consequence, $e$ is counted at least once in  $E_{j}$ over all $j$ with $x_j\in[x_{[i]_e},x_{[i+1]_e})$, and the inequality \eqref{eq:counting-deg} is then proved. Therefore, $$\sum_{j\in V}\widetilde{\deg}(j)x_{j}^p=\int_0^\infty \sum_{j\in V: x_j^p\ge t}\widetilde{\deg}(j) dt\le c\int_0^\infty\sum_{j\in V: x_j^p\ge t}g(j)dt=c\sum_{j\in V}g(j)x_{j}^p,$$
which implies  \eqref{eq:third-deg}.

Consequently, for $\vec x\in \R_{\ge0}^n\setminus\{\vec 0\}$,
\begin{equation*}
 \frac{\sum\limits_{e\in E}w_ef^L_e(\vec x^p)}{\|\vec x\|_{p,g}^p}:= \frac{\sum\limits_{e\in E}w_ef^L_e(\vec x^p)}{\sum_{j\in V}g(j)x_{j}^p}\le p(\frac{c}{2})^{\frac{1}{p'}}\left(\frac{\sum\limits_{e\in E}w_e(f^L_e(\vec x))^p}{\sum_{j\in V}g(j)x_{j}^p}\right)^{\frac1p}.
\end{equation*}
Similarly, for $\vec x\in\R_{\le0}^n\setminus\{\vec 0\}$, denoting by $|\vec x|=(|x_1|,\ldots,|x_n|)$, and taking  $\tilde{f}_e(S)=f_e(V\setminus S)$, $\forall S\subset V$,  we have
\begin{equation*}
 \frac{\sum\limits_{e\in E}w_e\tilde{f}^L_e(|\vec x|^p)}{\|\vec x\|_{p,g}^p}= \frac{\sum\limits_{e\in E}w_ef^L_e(-|\vec x|^p)}{\|-|\vec x|\|_{p,g}^p}= \frac{\sum\limits_{e\in E}w_ef^L_e(-|\vec x|^p)}{\sum_{j\in V}g(j)|x_{j}|^p}\le p(\frac{c}{2})^{\frac{1}{p'}}\left(\frac{\sum\limits_{e\in E}w_e(f^L_e(\vec x))^p}{\sum_{j\in V}g(j)|x_{j}|^p}\right)^{\frac1p}.
\end{equation*}

For $\vec x\in\R^n\setminus (\R_{\ge0}^n\cup \R_{\le0}^n)$, let $\vec x=\vec x_+ + \vec x_-$ be such that $( x_+)_i=\max\{x_i,0\}$ and $( x_-)_i=\min\{x_i,0\}$ for all $i\in V$. Accordingly, we derive that
\begin{align*}
    \frac{\sum\limits_{e\in E}w_e(f^L_e(\vec x))^p}{\sum_{j\in V}g(j)|x_{j}|^p}
 &=   \frac{\sum\limits_{e\in E}w_e|f^L_e(\vec x_+)+f^L_e(\vec x_-)|^p}{\sum_{j\in V}g(j)|x_{+,j}|^p+\sum_{j\in V}g(j)|x_{-,j}|^p}\\
 &\ge  \min\left\{ \frac{\sum\limits_{e\in E}w_e(f^L_e(\vec x_+))^p}{\|\vec x_+\|_{p,g}^p}, \frac{\sum\limits_{e\in E}w_e(f^L_e(\vec x_-))^p}{\|\vec x_-\|_{p,g}^p}\right\}\\
  &\ge \frac{1}{p^p}(\frac 2c)^{\frac{p}{p'}}\min\left\{ \frac{\sum\limits_{e\in E}w_ef^L_e(\vec x^p_+)}{\|\vec x_+^p\|_{1,g}}, \frac{\sum\limits_{e\in E}w_e\tilde{f}^L_e(|\vec x_-|^p)}{\|\vec x_-^p\|_{1,g}}\right\}^p\\
  &\ge \frac{(2/c)^{p-1}}{p^p}\min\left\{\frac{f(A_+)}{g(A_+)}, \frac{\tilde{f}(A_-)}{g(A_-)}\right\}^p
  \\&\ge \frac{(2/c)^{p-1}}{p^p}\min\left\{\frac{\hat{f}(A_+)}{g(A_+)}, \frac{\hat{f}(A_-)}{g(A_-)}\right\}^p,
\end{align*}
for some nonempty subset $A_\pm\subset \mathrm{supp}(\vec x_\pm)$ provided by Theorem A in \cite{JostZhang-PL} or Theorem \ref{thm:optimal-identity-fg}, where $\hat{f}(S):=\min\{f(S),f(V\setminus S)\}= \min\{f(S),\tilde{f}(S)\}$.

For any nonconstant vector $\hat{\vec x}$, since $g$ is a volume function,  there exists $c\in\{\hat{ x}_1,\cdots,\hat{ x}_n\}\subset \R$ such that  the vector 
$\vec x=\hat{\vec x}-c\vec 1$ satisfies 
 \begin{equation*}
      g(\mathrm{supp}(\vec x_+))\le \frac12 g(V)\qquad\text{and}\qquad g(\mathrm{supp}(\vec x_-))\le \frac12 g(V).
 \end{equation*}
Then, $f_e^L( \vec x)=f_e^L(\hat{\vec x})-cf_e(V)=f_e^L(\hat{\vec x})$ provided by $f_e(V)=0$. Therefore, we have
 \begin{align*}
     \frac{\sum_{e\in E} w_e(f_e^L(\hat{\vec x}))^p}{\min\limits_{c\in\R}\sum_{e\in E} g_e^L(|\hat{\vec x}-c\vec1|^p)}\ge \frac{\sum_{e\in E} w_e(f_e^L(\vec x))^p}{\sum_{e\in E} g_e^L(|\vec x|^p)}
     \ge \frac{(2/c)^{p-1}}{p^p}\cdot\frac{\hat{f}(A)^p}{g(A)^p}
 \end{align*}
 for some nonempty subset $A\subset \mathrm{supp}(\vec x_+)\subset \mathrm{supp}(\vec x)$ or $A\subset \mathrm{supp}(\vec x_-)\subset \mathrm{supp}(\vec x)$. In consequence,  $g(A)\le \frac12 g(V)$ (i.e. $g(A)\le g(V\setminus A)$).
It is easy to see that
\begin{align*}
\min_{A\ne\emptyset,V}\frac{\hat{f}(A)}{\min \{g(A),g(V\setminus A)\}}&=  \min_{A\ne\emptyset,V}\frac{\min \{f(A),f(V\setminus A)\} }{\min \{g(A),g(V\setminus A)\}}
\\&= \min_{A\ne\emptyset,V}\frac{f(A) }{\min \{g(A),g(V\setminus A)\}}=h.
\end{align*}
Therefore,  $\lambda\ge \frac{(2/c)^{p-1}}{p^p}h^p$. The proof is completed.
\end{proof}

\begin{example}\label{example:chemical-p-Lap}
Given a  chemical  hypergraph $(V,E)$ (see \cite{JostMulas19} for the definition) satisfying  $e_{in}\ne\emptyset\ne e_{out}$ and $\#(e_{in}\cup e_{out})\ge 2$ for all $e\in E$, let $f_{e}:\power(V)\to \R$ be defined by
\begin{equation*}
f_{e}(A)=\begin{cases}
1,&\text{ if } e_{in}\cap A\ne\emptyset\ne e_{out}\setminus A\text{ \; or\; }e_{out}\subset  A\subset V\setminus e_{in},\\
0,&\text{ otherwise},
\end{cases}
\end{equation*}
where $e_{in}:=\{\text{inputs of }e\}$ and  $e_{out}:=\{\text{outputs of }e\}$. 
Then the Lov\'asz extension of $f_{e}$ is determined by 
\begin{equation*}f_{e}^L(\vec x)=\left|\max\limits_{i\in e_{in}}x_i-\min\limits_{j\in e_{out}}x_j\right|.\end{equation*}
And the associated  $p$-Laplacian $\Delta_p$ induced by the Lov\'asz extension is defined as
$$\Delta_p\vec x= \nabla \sum_{e\in E}|\max\limits_{i\in e_{in}}x_i-\min\limits_{j\in e_{out}}x_j|^p$$
which satisfies
\begin{equation*}\frac 1p\langle \Delta_p \vec x,\vec x \rangle=\sum_{e\in E}\left|\max\limits_{i\in e_{in}}x_i-\min\limits_{j\in e_{out}}x_j\right|^p.\end{equation*}
\end{example}
The  $p$-Laplacian $\Delta_p$ introduced in Example \ref{example:chemical-p-Lap} can be described and computed approximately by  the following steps:
\begin{enumerate}
\item Given a vector $\vec x\in\R^V=\R^n$, for each hyperedge $e\in E$, let $i_e=\mathop{\mathrm{argmax}}\limits_{i\in e_{in}}x_i$ and $j_e=\mathop{\mathrm{argmin}}\limits_{j\in e_{out}}x_j$.
    \item Construct the weighted graph $G_{\vec x}$ on the vertex set $V$ by adding edges $\{i_e,j_e\}$ having weight $w(i_e,j_e)=w(e):=1$. 
  \item The  $p$-Laplacian $\Delta_p$ is defined to be the usual $p$-Laplacian $\Delta_p[G_{\vec x}]$ w.r.t. the graph $G_{\vec x}$, and $\Delta_p\vec x:=\Delta_p[G_{\vec x}]\vec x$. 
\end{enumerate}

Although  the related energy function 
\begin{equation*}
\vec x\mapsto\frac{\sum_{e\in E}|\max\limits_{i\in e_{in}}x_i-\min\limits_{j\in e_{out}}x_j|^p}{\sum_{i\in V}\deg(i)|x_i|^p}
\end{equation*}
is not smooth in general, it has some  features that are similar to the graph case. Amazingly, this eigenvalue problem is very similar to the graph case since we have a relative  isoperimetric  inequality for that. In detail, the second smallest eigenvalue of such a   $p$-Laplacian and the Cheeger constant
\begin{equation*}h:=\min\limits_{A\in\power(V)\setminus\{\emptyset,V\}}\frac{\#(\partial A)}{\min\{\vol(A),\vol(V\setminus A)\}}\end{equation*}
satisfy  Cheeger's inequality, where we adopt  the volume  $\vol(A):=\sum_{e\in E}\#(e\cap A)=\sum_{i\in A}\deg(i)$, the degree $\deg(i):=\# \{e\in E: i\in e\}$, and the boundary set \begin{equation*}\partial A:=\{e\in E: e_{in}\cap A\ne\emptyset\ne e_{out}\setminus A\text{ \; or\; }e_{out}\subset  A\subset V\setminus e_{in}\}.\end{equation*}

Also, the $p$-Laplacian on  chemical hypergraphs satisfies a nodal domain property, which is very similar to the graph case shown in Proposition \ref{pro:inertia-p-Lap}.

\begin{proof}
[Proof of Proposition \ref{pro:inertia-p-Lap}]
Let $U\subset V$ be  a maximal independent set.  Then for any $\vec x\in\R^n$ with $x_i=0$ $\forall i\in V\setminus U$, we have   $|x_i-x_j|^p=|x_i|^p+|x_j|^p$ whenever  $\{i,j\}\in E$.  Therefore,  for any $\vec x\in\R^n$ satisfying $\mathrm{supp}(\vec x)\subset U$, we get 
$$\frac{F(\vec x)}{G(\vec x)}:=\frac{\sum_{\{i,j\}\in E}w_{ij}|x_i-x_j|^p}{\sum_{i\in V}\deg_i|x_i|^p}=\frac{\sum_{\{i,j\}\in E}w_{ij}(|x_i|^p+|x_j|^p)}{\sum_{i\in V}\deg_i|x_i|^p}=1.$$
By Theorem \ref{thm:nodal-inertia-bound}, we obtain that $\alpha\le\alpha_1 \le\min\{\#\{\lambda_i\le 1\},\#\{\lambda_i\ge 1\}\}$. 

It can be verified that the  connected components of the support set of an eigenvector form a family of nodal domains in the sense of Definition  \ref{def:nice-nodal-domain}. Hence, Theorem \ref{thm:nodal-inertia-bound} and Theorem \ref{thm:inertia-nodal} can be directly applied to $p$-Laplacian to get the upper bound $\min\{ k+r-1, n-k+r\}$ for the number of nodal domains.
\end{proof}

As a  consequence of Theorem \ref{thm:Cheeger-main}, and as an analog of  Proposition \ref{pro:inertia-p-Lap},  we have
\begin{theorem}
\label{thm:p-Lap-Cheeger}
Under the above setting, we have the following Cheeger inequality
\begin{equation}\label{eq:Cheeger-chemical-hyper-p}
2^{p-1}\frac{h^p}{p^p}\le \lambda_2(\Delta_p)\le 2^{p-1}h.
\end{equation}
Also, we have the inertia bound $\alpha\le\min\{\#\{\lambda_i(\Delta_p)\le 1\},\#\{\lambda_i(\Delta_p)\ge 1\}\} $.   
And for any eigenvector $\vec x$ w.r.t. $\lambda_i(\Delta_p)$ whose multiplicity is $r$, the number  of connected components of the support set of $\vec x$ is smaller than or equal to $\min\{ i+r-1, n-i+r\}$. Here the independence number $\alpha$ and the connected components can be defined on the underlying graph\footnote{Two vertices are connected by an edge in the underlying graph if and only if there exists $h\in H$ with $h_{in}\cup h_{out}\supset\{i,j\}$.} induced by the chemical hypergraph.
\end{theorem}

\begin{remark}\label{remark:p-Lap-} Theorem \ref{thm:Cheeger-main} generalizes and enhances the relevant results in the recent references  \cite{LM18,Yoshida19}. Moreover, Theorem \ref{thm:p-Lap-Cheeger} includes the following special cases: 
\begin{itemize}
    \item Taking $p=2$ and letting  $e_{in}=e_{out}$ for any $e\in E$, we get Louis hypergraph Laplacian  \cite{Louis15} and the Cheeger inequality therein. 
\item Taking $(V,E)$ as a graph (i.e., $e_{in}=e_{out}$ and $\#e_{in}=2$),
  Theorem \ref{thm:p-Lap-Cheeger} implies the Cheeger inequality for the graph $p$-Lapalcian  \cite{TudiscoHein18}. \item Taking $p=1$ and letting  $e_{in}=e_{out}$ for any $e\in E$, we get the total variation on hypergraphs  \cite{TVhyper-13}.
\end{itemize}
\end{remark}

In general, letting  $e:=e_{in}=e_{out}$, one can obtain
$$\frac{\sum_{e\in E}|\max\limits_{i\in e}x_i-\min\limits_{j\in e}x_j|^p}{\sum_{i\in V}\deg(i)|x_i|^p}=\frac{\sum_{e\in E}\max\limits_{i,j\in e}|x_i-x_j|^p}{\sum_{i\in V}\deg(i)|x_i|^p}.$$
If we further take $p=1$ and let the edge set be  $E=\{N(i):i\in V\}$, where $N(i)$ is  the 1-neighborhood of $i$,  then we recover the equality of Cheeger constants w.r.t. the vertex-boundary \cite{JostZhang-PL}.

  Mulas\cite{Mulas20}  generalizes the graph Cheeger inequalities to the case of  $k$-uniform hypergraphs, using the normalized Laplacian for hypergraphs  \cite{JostMulas19}. 
From a different perspective, we indeed provide in this section  a way of defining a $p$-Laplacian from a Lov\'asz type extension of a Cheeger quantity so that the Cheeger inequality emerges automatically.

\subsection{Tensors and their eigenvalues}\label{sec:tensor}

Eigenvalues for tensors have been defined by Lim \cite{Lim05} and Qi \cite{Qi05} in different ways; see also  the presentations in \cite{Sturmfels16,Qi18}. Here, in line with our general procedure, we approach the eigenvalue problem of tensors through Rayleigh quotients. We consider $d$-dimensional $n\times \cdots \times n$-tensors, that is, arrays of the form $A=(a_{i_1i_2\dots i_d})$ where each entry takes its values in $\mathbb{R}$ and the indices $i_1, i_2,\dots ,i_d$ range from $1$ to $n$. We assume that the tensor $A$ is symmetric, that is, each entry $a_{i_1i_2\dots i_d}$ is invariant under permutations of the indices. We write $A(\vec x,\dots ,\vec x)=\sum_{i_1,\dots ,i_d=1}^n a_{i_1i_2\dots i_d}x_{i_1}\mdot\mdot\mdot x_{i_d}$ for $ \vec x=(x_{1},\dots ,x_{n})\in \mathbb{R}^d$, and we can then also define $A(\vec x(1),\dots ,\vec x(d))$. 
We then consider the quotient
\begin{equation}
 \vec x=(x_1,\dots ,x_n)\mapsto   \frac{\sum_{i_1,\dots ,i_d=1}^n a_{i_1i_2\dots i_d}x_{i_1}\mdot \mdot \mdot x_{i_d}}{\sum_i x_i^d}.
\end{equation}
Its critical points then satisfy the eigenvalue equation
\begin{equation}
  \sum_{i_2,\dots ,i_d=1}^n a_{ii_2\dots i_d}x_{i_2}\mdot \mdot \mdot x_{i_d}  =\lambda x_i^{d-1}
\end{equation}
for all $i=1,\dots, n$ and some $\vec x\neq \vec 0$ and some $\lambda\in \mathbb{R}$. Note that the first index is excluded from the sum, but by symmetry of $A$, we could have as well taken any other index for that role. This is the eigenvalue equation of \cite{Qi05,Sturmfels16,Qi18}. 

\begin{defn}[H-eigenvalue]
Continuing   Example \ref{ex:tensor}, 
for two  order-$k$  $n$-dimensional tensors   $C:=(c_{i_1,\cdots,i_k})$ and $D:=(d_{i_1,\cdots,i_k})$, the H-eigenvalue problem of $(C,D)$ is to find a pair $(\lambda,\vec x)\in\R\times (\R^n\setminus \{\vec0\})$ satisfying $C\vec x^{k-1}=\lambda D\vec x^{k-1}$, where  $C\vec x^{k-1}:=(\sum_{i_2,\cdots,i_k=1}^nc_{i,i_2,\cdots,i_k}x_{i_2}\cdots x_{i_k})_{i=1}^n$.
\end{defn}

\begin{defn}
The adjacency tensor $A$ of a $k$-uniform hypergraph $(V,E)$ is a non-negative  symmetric tensor such that $a_{i_1,\cdots,i_k}>0$ $\Leftrightarrow$ $\{i_1,\cdots,i_k\}\in E$. Denote the $i$-th eigenvalue w.r.t. a non-negative  diagonal tensor $D$ by $\lambda_i:=\inf\limits_{\gen(S)\ge i}\sup\limits_{x\in S}\frac{\langle A\vec x^{k-1},\vec x\rangle}{\langle D\vec x^{k-1},\vec x\rangle}$. 
\end{defn}

\begin{proof}[Proof of Proposition \ref{pro:inertia-k-uniform}]
Let $U$ be a maximal independent set with $\#U=\alpha$. 
For any $\vec x\in \R^U$, $x_{i_1}\cdots x_{i_k}=0$ whenever $\{i_1,\cdots,i_k\}$ is a hyperedge. And $a_{i_1,\cdots,i_k}=0$ if $\{i_1,\cdots,i_k\}$ is not a hyperedge. Thus,  
$$\frac{F(\vec x)}{G(\vec x)}:=\frac{\langle A\vec x^{k-1},\vec x\rangle}{\langle D\vec x^{k-1},\vec x\rangle}=\frac{\sum a_{i_1,\cdots,i_k}x_{i_1}\cdots x_{i_k}}{\sum d_{i_1,\cdots,i_k}x_{i_1}\cdots x_{i_k}}=0.$$
Applying Theorem \ref{thm:nodal-inertia-bound} to the function pair $(F,G)$, we get  $\alpha\le\alpha_0 \le\min\{\#\{\lambda_i\le 0\},\#\{\lambda_i\ge 0\}\}$.

We will check that the set of connected components of the support set of an eigenvector is a family of nodal  domains in the sense of Definition  \ref{def:nice-nodal-domain}. Indeed, for connected components  $U_1,\cdots,U_k$
 of the support set $\supp(\vec x)$, $a_{i_1,\cdots,i_k}=0$ if
 $\{i_1,\cdots,i_k\}$ intersects two of these components. Thus it can be
 checked that  $\frac{\partial}{\partial x_v}F(\vec
 x|_{U_j})=\frac{\partial}{\partial x_v}F(\vec x)$  for any $v\in U_j$. By the
 condition that $D$ is a nonnegative diagonal matrix, the function $G(\vec
 x):=\langle D\vec x^{k-1},\vec x\rangle$ also  possesses the property
 $\frac{\partial}{\partial x_v}G(\vec x|_{U_j})=\frac{\partial}{\partial
   x_v}G(\vec x)$ for any $v\in U_j$. Since $(\lambda_i,\vec x)$ is an
 eigenpair, i.e., $\nabla F(\vec x)=\lambda_i\nabla G(\vec x)$, we have
 $\frac{\partial}{\partial x_v}F(\vec
 x|_{U_j})=\lambda_i\frac{\partial}{\partial x_v}G(\vec x|_{U_j})$ for any
 $v\in U_j$. By the Euler identity for homogeneous functions, we immediately get $F(\vec x|_{U_j})=\lambda_iG(\vec x|_{U_j})$. Finally, it is not difficult to show that for any $t_1,\cdots, t_k$, $F(\sum_{j=1}^kt_j\vec x|_{U_j})=\lambda_iG(\sum_{j=1}^kt_j\vec x|_{U_j})$. Therefore, Theorem \ref{thm:nodal-inertia-bound} can be applied to give   the upper bound $\min\{ k+r-1, n-k+r\}$ for the number of nodal domains.
\end{proof}

We can of course more generally take certain norms in the Rayleigh quotient. 
For instance, we could adopt  the norm $\|\vec x(j)\|_{p_j}$ for the $j$-th argument and consider as in \cite{Lim05}
\begin{equation}\label{eq:Rayleigh-norms}
 (\vec x(1),\dots ,\vec x(d))\mapsto   \frac{|A(\vec x(1),\dots ,\vec x(d))|}{\|\vec x(1)\|_{p_1} \dots  \|\vec x(d)\|_{p_d}}.
\end{equation}
and take its stationary points as eigenvectors. Also, there are many meaningful   optimization problems for \eqref{eq:Rayleigh-norms} with constraints. For example, by Theorem \ref{thm:optimal-identity-fg}, we have an interesting equality for the graph maxcut problem: $$\max\limits_{S\subset V}\#\partial S=\max\limits_{S\cap T=\varnothing}\#E(S,T)=\max\limits_{|x|^\top  |y|=0}\frac{A(\vec x,\vec y)}{\|\vec x\|_\infty\|\vec y\|_\infty},$$
where $|\vec x|:=(|x_1|,\cdots,|x_n|)$, and $A$ is the adjacency matrix of a graph $(V,E)$ with $V=\{1,\cdots,n\}$. 

We now consider the case where all entries $a_{i_1i_2\dots i_d}\in \{0,1\}$. Such a tensor can be seen as representing a simplicial complex with vertex set $V=\{1,\dots ,n\}$ and where the vertices $i_1, \dots ,i_d$ form a $(d-1)$-simplex iff $a_{i_1i_2\dots i_d}=1$. We can then define a tensor Laplacian 
\begin{equation}
    (\Delta x)_i= -\frac{1}{\deg_i}\sum_{i_2, \dots ,i_d=1}^n a_{i_1i_2\dots i_d} x_{i_2} \mdot \mdot \mdot  x_{i_d}+ (x_i)^{d-1}
\end{equation}
for $i=1,\dots ,n$, where $\deg_i$ is the number of entries $a_{ii_2\dots i_d}=1$ when $i_2,\dots ,i_d=1,\dots, n$, or equivalently, the number of $(d-1)$-simplices containing $i$. The eigenvalue equation for this Laplacian then is
\begin{equation}
   -\frac{1}{\deg_i}\sum_{i_2, \dots ,i_d=1}^n a_{i i_2\dots i_d} x_{i_2} \mdot \mdot \mdot  x_{i_d}+ (x_i)^{d-1} =\lambda (x_i)^{d-1} \text{ for all }i
\end{equation}
for some $\vec x\neq \vec0$ and some real eigenvalue $\lambda$. 
Clearly, $\lambda=0$ is an eigenvalue for the constant eigenfunction. 

Equivalently, we can write of course
\begin{equation}\label{eq:complex-tensor-Laplacian}
   \sum_{i_2, \dots ,i_d=1}^n a_{i i_2\dots i_d} x_{i_2} \mdot \mdot \mdot x_{i_d}- (1-\lambda) \deg_i(x_i)^{d-1} =0  \text{ for all }i.
\end{equation}
This comes from the Rayleigh quotient
\begin{equation}
 \vec   x\mapsto \frac{\sum_i\deg_i x_i^d-\sum_{i_1,i_2, \dots ,i_d=1}^n a_{i_1i_2\dots i_d}  x_{i_1} x_{i_2} \mdot \mdot \mdot   x_{i_d}}{\sum_i\deg_i x_i^d}. 
\end{equation}
This Laplace operator can be generalized to arbitrary symmetric tensors with nonnegative entries when we put $\deg_i=\sum_{i_2, \dots ,i_d=1}^n a_{ii_2\dots i_d}$.

We then have the following analog of  Proposition \ref{pro:inertia-k-uniform}.
\begin{pro}\label{pro:inde-nodal-tensor}
The independence number of a $(d-1)$-dim simplicial complex on $V$ is defined as $\alpha=\max\{\#U:U\subset V\text{ s.t. }U\text{ contains no }(d-1)\text{-dim  simplex}\}$. Let $\lambda_i$ be the $i$-th minimax  eigenvalue of the  eigenvalue problem  \eqref{eq:complex-tensor-Laplacian}. Then $\alpha\le\min\{\#\{\lambda_i\le 1\},\#\{\lambda_i\ge 1\}\}$. Moreover, for any eigenvector $\vec x$ w.r.t. $\lambda_i$ whose multiplicity is $r$, the number  of connected components of the support  of $\vec x$ is smaller than or equal to $\min\{ i+r-1, n-i+r\}$. 
\end{pro}

\subsection{Signed (hyper-)graphs}\label{sec:signed-graph}
Spectral theory for signed graphs has many important applications. 
A  breakthrough  of Huang \cite{Huang19} asserts that any induced subgraph of an $n$-dimensional hypercube on a set of $2^{n-1}+1$  vertices has maximum degree at least $\sqrt{n}$. This confirms the Sensitivity Conjecture in the field of computer science. 
In this section, we  use Theorem \ref{thm:fg-Huang} to obtain more results on signed graphs.   

We first generalize the concept of a signed graph to allow for edge weights. 
\begin{defn}[weighted signed graph]
A weighted signed graph is a pair $(V,W)$ of the vertex set   $V=\{1,2,\cdots,n\}$  and the adjacency matrix   $W=(w_{ij})_{n\times n}$, where $w_{ij}=w_{ji}$  and  $w_{ii}=0$, $\forall i,j\in V$.

If $w_{ij}\in\{0,1,-1\}$ for any $i,j\in V$, we call such a $(V,W)$  a signed graph.

If $w_{ij}\ge 0$,  $(V,W)$ is called a weighted graph. And if $w_{ij}\in\{0,1\}$, we get a simple graph.
\end{defn}

\begin{theorem}\label{thm:generalized-inertia-Huang}
For a weighted graph $(V,W)$ with $\#V=n$, we put $S(W)=\{W'=(w_{ij}')_{n\times n}:(V,W')\text{ is a weighted signed graph with }|w_{ij}'|=w_{ij},\forall i,j\in V\}$. Then we have
\begin{equation}
    \min\limits_{U\subset V,\#U=k}\max\limits_{i\in U}\deg_U(i)\ge \max\limits_{W'\in S(W)}\max\{\lambda_k(W'),-\lambda_{n-k+1}(W')\}
\end{equation}
where $\deg_U(i):=\sum_{j\in U}w_{ij}$ is the degree of the vertex $i$ of the induced subgraph $(U,W|_U)$.
\end{theorem}

\begin{proof}
Taking $f(A,B)=\sum_{i\in A,j\in B}w_{ij}$ and $g(A,B)=\#(A\cap B)$ for $A,B\subset V$ in 
Theorem \ref{thm:fg-Huang}, and by Example \ref{ex:spectral-radius}, we have $\max\limits_{i\in U}\deg_U(i)=\max\limits_{A\subset B}\frac{f(A,B)}{g(A,B)}$,  $f^Q_\Delta(\vec x)=\langle W\vec x,\vec x\rangle$ and $g^Q_\Delta(\vec x)=\langle \vec x,\vec x\rangle$. For any $W'\in S(W)$, taking  $F'(\vec x)=\langle W'\vec x,\vec x\rangle$ and   $G'(\vec x)=\langle \vec x,\vec x\rangle$, we have  $\lambda_i(W')=\lambda_i(F',G')$ and $\lambda_{n-i+1}(W')=\lambda_i'(W')=\lambda_i'(F',G')$ (by the classical min-max theorem). Therefore, the proof is completed by Theorem \ref{thm:fg-Huang}.
\end{proof}


\begin{remark}
Theorem \ref{thm:generalized-inertia-Huang} implies the inertia bound for the  independent number. In fact, let $d_k=  \min\limits_{U\subset V,\#U=k}\max\limits_{i\in U}\deg_U(i)$ and $s_k=\max\limits_{W'\in S(W)}\max\{\lambda_k(W'),-\lambda_{n-k+1}(W')\}$. Then both $(d_k)_{k\ge 1}$ and $(s_k)_{k\ge 1}$ are non-decreasing sequences with  $s_k\le d_k$, $\forall k$. 

Clearly, $d_k=0$ $\Leftrightarrow$ there is an independent set of  $k$ elements $\Leftrightarrow$ $\alpha\ge k$. So, it follows from $d_{\alpha}=0$ that $s_\alpha\le0$,  which means $\max\{\lambda_{\alpha}(W),-\lambda_{n-\alpha+1}(W)\}\le 0$, i.e, $\lambda_1\le\cdots\le \lambda_{\alpha}\le 0\le \lambda_{n-\alpha+1}\le\cdots\le \lambda_n$. Therefore, we get the inertia bound $\alpha\le \min\{\#\{\lambda_i\le 0\},\#\{\lambda_i\ge 0\}\}$. 
\end{remark}

\begin{remark}
Following  Huang's idea, we can use Theorem \ref{thm:generalized-inertia-Huang} to get a very slight generalization of Huang's theorem in the following way:
\begin{itemize}
    \item[Step 1.] Let $(V,W)$ be a weighted graph such that there exists $W'\in S(W)$ satisfying $W'^2=\lambda I$ with $\lambda>0$. Then $n:=\#V$ is even, and $\min\limits_{U\subset V,\#U=\frac n2+1}\max\limits_{i\in U}\deg_U(i)\ge \sqrt{\lambda}$.

Proof: Note that the eigenvalues of $W'$ are  $\pm \sqrt{\lambda}$. Combining
this with the fact that $\mathrm{trace}(W')=0$, we obtain that $n$ is even and $W'$ is similar to $\mathrm{diag}(\underbrace {\sqrt{\lambda},\cdots,\sqrt{\lambda}}\limits_{n/2},\underbrace {-\sqrt{\lambda},\cdots,-\sqrt{\lambda}}_{n/2})$.  Hence, $\lambda_{\frac n2+1}(W')=\sqrt{\lambda}$ and $\lambda_{\frac n2}(W')=-\sqrt{\lambda}$. By Theorem \ref{thm:generalized-inertia-Huang}, we complete the proof of Step 1.

 \item[Step 2.] Let $(\tilde{V},\tilde{W})$ be the Cartesian product of the  weighted graph $(V,W)$  and the path graph on two vertices  with the edge weight $w$. Then there exists 
$\tilde{W}'\in S(\tilde{W})$ satisfying  $\tilde{W}'^2=(\lambda+w^2)I$. 

Proof: By the basic  property of the Cartesian product, we have  $\tilde{W}=\left(\begin{matrix}
W & wI\\
wI & W\end{matrix}\right)$. Now, let  $\tilde{W}'=\left(\begin{matrix}
W' & wI\\
wI &-W'\end{matrix}\right)$. Then $$\tilde{W}'^2=\left(\begin{matrix}
W'^2+w^2I & O\\
O &W'^2+w^2I\end{matrix}\right)=(\lambda+w^2)I$$ and $\tilde{W}'\in S(\tilde{W})$. h

 \item[Step 3] An $n$-dimensional weighted hypercube is the  Cartesian product of $n$ path graphs on two vertices with edge weights $w_1,\cdots,w_n$, respectively. 
Any induced subgraph of an $n$-dimensional weighted hypercube $(V,W)$ on a set of $2^{n-1}+1$  vertices has maximum degree at least $\sqrt{w_1^2+\cdots+w_n^2}$. 

Proof: By Step 2, it immediately follows from  mathematical induction on $n$ that  there exists $W'\in S(W)$ satisfying $W'^2=(w_1^2+\cdots+w_n^2)I$. And then by Step 1, the proof is completed. 
\end{itemize}
\end{remark}

Furthermore,  Theorem \ref{thm:fg-Huang} implies a similar estimate for signed weighted hypergraphs. 

\begin{defn}[signed weighted  hypergraph]
A  signed weighted hypergraph is a pair $(V,W)$ of the vertex set   $V=\{1,2,\cdots,n\}$  and its adjacency $k$-order tensor   $W=(w_{i_1,\cdots,i_k})_{n\times n}$, where $w_{i_1,\cdots,i_k}=w_{\sigma(i_1),\cdots,\sigma(i_k)}$ for any permutation $ \sigma\in S_k$,  $\forall i_1,\cdots,i_k\in V$. If  $w_{i_1,\cdots,i_k}\ge 0$, we call  $(V,W)$   a weighted hypergraph for simplicity. 
\end{defn}

\begin{theorem}\label{thm:generalized-inertia-hyper-Huang}
Given a weighted hypergraph $(V,W)$ with $\#V=n$, denote by $S(W)=\{W'=(w_{i_1,\cdots,i_k}')_{n\times n}:(V,W')\text{ is a signed weighted  hypergraph with }|w_{i_1,\cdots,i_k}'|=w_{i_1,\cdots,i_k},\forall i_1,\cdots,i_k\in V\}$. Then we have
\begin{equation}\label{eq:tensor-hypergraph-spec}
    \min\limits_{U\subset V,\#U=m}\max\limits_{i\in U}\deg_U(i)\ge \max\limits_{W'\in S(W)}\max\{\lambda_m(W'),-\lambda_{m}'(W')\}
\end{equation}
where $\deg_U(i):=\sum_{i_1,\cdots,i_{k-1}\in U}w_{i,i_1,\cdots,i_{k-1}}$ is the degree of the vertex $i$ of the sub-hypergraph $(U,W|_U)$. 
\end{theorem}
\begin{proof}
Taking $f(A_1,\cdots,A_k)=\sum_{i_1\in A_1,\cdots,i_k\in A_k}w_{i_1,\cdots,i_k}$ and $g(A_1,\cdots,A_k)=\#(A_1\cap\cdots\cap A_k)$ for $A_1,\cdots,A_k\subset V$, 
it is not difficult to check that
$$\max\limits_{i\in U}\deg_U(i)=\max\limits_{A\subset U}\frac{\sum_{i\in A}\deg_U(i)}{\#A} =\max\limits_{\;\text{ chain }A_1,\cdots,A_k\subset U}\frac{f(A_1,\cdots,A_k)}{g(A_1,\cdots,A_k)}.$$
Also, we have $f^Q_\triangle(\vec x)=\sum_{i_1,\cdots,i_k}w_{i_1,\cdots,i_k}x_{i_1}\cdots x_{i_k}$ and $g^Q_\triangle(\vec x)=\sum_i x_i^k$, and $\lambda_i(W)=\lambda_i(f^Q_\triangle,g^Q_\triangle)$. Finally, we are able to apply Theorem \ref{thm:fg-Huang} to get \eqref{eq:tensor-hypergraph-spec}, as the remaining part  is similar to that of Theorem \ref{thm:generalized-inertia-Huang}.
\end{proof}

\subsection{Spectral theory on simplicial complexes}
\label{sec:simplicial-complex}
In this section, we use the extension theory and the spectral theory for function pairs to give some preliminary investigations on (nonlinear)  eigenvalue problems for simplicial complexes. 

We shall work on an abstract simplicial complex $K$ with the vertex set $V=\{1,\cdots,n\}$. For any $\sigma=\{i_0,\cdots,i_d\}\in K$, we use $[\sigma]:=[i_0,\cdots,i_d]$ to indicate  the  oriented  $d$-dimensional simplex which is  formed by $\sigma$. Let $S_d$ be the collection of all simplexes in $K$ of  dimension $d$, and let  $[S_d]=\{[\sigma]:\sigma\in S_d\}$ be the set of all oriented  $d$-simplexes.   

The $d$-th chain group $C_d(K)$ of $K$ is a vector space with the  basis $[S_d]$. The boundary map $\partial_d:C_d(K)\to C_{d-1}(K)$  is a linear operator 
 defined by    $\partial_d[i_0,\cdots,i_d]=\sum_{j=0}^d(-1)^j[i_0,\cdots,i_{j-1},i_{j+1},\cdots,i_d]$, which can  also be represented by the incidence matrix $B_d$ of dimension $\#S_{d-1}\times \#S_d$. Clearly, the elements of the matrix $B_d$ belong to  $\{-1,0,1\}$.

 The $d$-th cochain group $C^d(K)$ is  defined as the dual of the chain group $C_d(K)$. The  simplicial coboundary map $\delta_d:C^d(K)\to C^{d+1}(K)$ is a linear operator generated by $(\delta_df)([i_0,\cdots,i_{d+1}])=\sum_{j=0}^{d+1}(-1)^jf([i_0,\cdots,i_{j-1},i_{j+1},\cdots,i_{d+1}])$ for any $f\in C^d(K)$. It is obvious that $\delta_d=B_{d+1}^\top$. We  use both the \textbf{incidence matrices}  and  the  \textbf{coboundary operators} to express the Laplace matrices/operators (see \cite{HorakJost}): 
\begin{enumerate}[-]
\item the $d$-th up Laplace  operator  $L^{up}_d:=\delta_d^*\delta_d=B_{d+1} B_{d+1}^\top$ 
\item the $d$-th down Laplace  operator  $L^{down}_d:=\delta_{d-1}\delta_{d-1}^*=B_d^\top B_d$ 
\item the $d$-th Hodge  Laplace  operator  $L_d:=L^{up}_d+L^{down}_d=\delta_d^*\delta_d+\delta_{d-1}\delta_{d-1}^*=B_d^\top B_d+B_{d+1} B_{d+1}^\top$ 
\end{enumerate}
It is known that the spectra of
these matrices  encode many qualitative properties  of the associated simplicial complex. The overall aim of this section is to bring forward the study of the nonlinear eigenvalue problems on simplicial complexes.  
We introduce the following $p$-Laplace operators on $C^d(K)$:  
\begin{enumerate}[-]
\item the $d$-th up Laplace  operator  $L^{up}_{d,p}:=\delta_d^*\alpha_{p}\delta_d$,  and for $f\in C^d(K)$,  $L^{up}_{d,p}f=B_{d+1} \alpha_p(B_{d+1}^\top f)$, where $\alpha_p:(t_1,t_2,\cdots)\mapsto (|t_1|^{p-2}t_1,|t_2|^{p-2}t_2,\cdots)$ for $p>1$, and  $\alpha_1:(t_1,t_2,\cdots)\mapsto \{(\xi_1,\xi_2,\cdots):\xi_i\in\mathrm{Sgn}(t_i)\}$.
\item the $d$-th down Laplace  operator  $L^{down}_{d,p}:=\delta_{d-1}\alpha_p\delta_{d-1}^*$, and for $f\in C^d(K)$,  $L^{down}_{d,p}f=B_d^\top \alpha_p (B_df)$ 
\item the $d$-th  Laplace  operator  $L_{d,p}:=\delta_d^*\alpha_p\delta_d+\delta_{d-1}\alpha_p\delta_{d-1}^*$, and for $f\in C^d(K)$, $L_{d,p}f=B_d^\top \alpha_p (B_df)+B_{d+1} \alpha_p (B_{d+1}^\top f)$ 
\end{enumerate}

\begin{pro}\label{pro:nonzero-p-Lap}
The nonzero
eigenvalues of the up $p$-Laplacians are in  one-to-one correspondence  with those of the 
down  $p^*$-Laplacians:
$$\{\lambda^{\frac 1p}:\lambda\text{ is a nonzero eigenvalue of }L_{d,p}^{up}\}=\{\lambda^{\frac {1}{p^*}}:\lambda\text{ is a nonzero eigenvalue of }L_{d+1,p^*}^{down}\}$$
\end{pro}

\begin{proof}
For $p>1$, we have 
$L^{up}_{d,p} f=\frac1p\nabla_f\| B_{d+1}^\top f\|_p^p$  and $L^{down}_{d+1,p^*} g=\frac{1}{p^*}\nabla_g\| B_{d+1} g\|_{p^*}^{p^*}$.  Then, the eigenvalues of $L_{d,p}^{up}$ coincide with those of $(\| B_{d+1}^\top \cdot\|_p^p,\|\cdot\|_p^p)$. 
We refer to the proof of Proposition \ref{pro:p-Lap-dual} for the rest. 
\end{proof}

By taking $p=2$, it is easy to see that  Proposition \ref{pro:nonzero-p-Lap}
generalizes the well-known relation between up and down  Laplacians, that is, the nonzero
eigenvalues of $L^{up}_d$ and  $L^{down}_{d+1}$ coincide. 

So, we can concentrate on the up $p$-Laplacian  for investigating the
spectra of simplicial complexes. 

We construct the {\sl underlying anti-signed graph} $G_{up}^-(S_d)$ on $S_d$ with 
the edge set 
$$\left\{\{[\tau],[\tau']\}:[\tau],[\tau']\in [S_d],\exists [\sigma]\in [S_{d+1}]\text{ s.t. }\tau,\tau'\subset\sigma\right\}$$
and the sign of an edge  $\{[\tau],[\tau']\}$ is $\mathrm{sgn}([\tau],[\tau']):=\mathrm{sgn}([\tau],\partial[\sigma])\cdot\mathrm{sgn}([\tau'],\partial[\sigma])$, where $\sigma\in S_{d+1}$ and $\tau,\tau'\subset \sigma$. 
\begin{remark}
One can of course take
$\mathrm{sgn}([\tau],[\tau'])=-\mathrm{sgn}([\tau],\partial[\sigma])\cdot\mathrm{sgn}([\tau'],\partial[\sigma])$
to get the so-called  {\sl underlying signed graph} $G_{up}(S_d)$  on
$S_d$. But in the following results, we mostly use the underlying anti-signed
graph  $G_{up}^-(S_d)$, since it is more convenient for  proving a Cheeger-type inequality. This construction is very natural and  can be observed from the definition of (up/down)  combinatorial  Laplacian matrices of a simplicial complex. A similar idea was already used to define the signed adjacency matrix of a triangulation on a surface \cite{FST08}. 
\end{remark}

To get more concise and more useful results,  we will work with the {\sl normalized up $p$-Laplace operator}   $\Delta^{up}_{d,p}$, whose eigenvalues are determined by the function pair $(\| B_{d+1}^\top \cdot\|_p^p,\|\cdot\|_{p,\deg}^p)$, where $\|f\|_{p,\deg}^p=\sum_{\tau\in S_d}\deg_\tau |f(\tau)|^p$. 

\begin{pro}\label{pro:maximal eigenvalue-balance}
The eigenvalues of $\Delta_{d,p}^{up}$ lie in $[0,(d+2)^{p-1}]$. In addition, for $p>1$, the spectrum of $\Delta_{d,p}^{up}$ contains $(d+2)^{p-1}$, if and only if the underlying anti-signed graph on $S_d$ has a  balanced component.  Moreover, the multiplicity of  $(d+2)^{p-1}$ equals the number of balanced components of the underlying anti-signed graph.  
\end{pro}
\begin{proof}
The upper bound $(d+2)^{p-1}$ of the eigenvalues of $\Delta_{d,p}^{up}$ is provided by H\"older's  inequality. For the equality case, it is not difficult to verify that  there exists a sub-partition $[S_d]^+\sqcup [S_d]^-$  of $[S_d]$ such that for any  $[\tau],[\tau']\in [S_d]^+\sqcup [S_d]^-$, $\mathrm{sgn}([\tau],[\tau'])=-1$ if and only if  $\#(\{[\tau],[\tau']\}\cap [S_d]^+)=1$. Then, we can switch the set $[S_d]^+$ to make  all edges in $[S_d]^+\sqcup [S_d]^-$  positive, 
meaning that the induced subgraph  $[S_d]^+\sqcup [S_d]^-$ of the  underlying anti-signed graph  is switching equivalent to an  all-positive  signed graph.  On the multiplicity,  we shall concentrate on $G_{up}^-(S_d)$, and all the verifications are standard. The proof is completed.
\end{proof}

\begin{remark}
It is clear that the underlying anti-signed graph on $S_d$ is balanced if and only if the underlying signed graph on $S_d$ is antibalanced. Moreover, a   graph (i.e., the case of  $d=0$) is antibalanced means that it is bipartite, and then,  Proposition \ref{pro:maximal eigenvalue-balance} is indeed an extension of the  fact that the spectrum of the normalized $p$-Lapalcian on a graph contains $2^{p-1}$ if and only if the graph is bipartite. 
\end{remark}

Several
problems in spectral theory for simplicial complexes  arise when trying to generalize the classical spectral results
that are known for graphs, such as the Cheeger inequality. Inspired by the resent results on simplicial complexes,  signed graphs and oriented hypergraphs \cite{AtayLiu,LLPP,Mulas20,SKM14}, we present the following Cheeger-type constants.

Given $A,A'\subset S_d$ that are disjoint, let $|E_+(A,A')|=\#\{\{[\tau],[\tau']\}:[\tau]\in A,[\tau']\in A',\mathrm{sgn}([\tau],[\tau'])=1\}$ and $|E_-(A)|=\#\{\{[\tau],[\tau']\}:[\tau],[\tau']\in A,\mathrm{sgn}([\tau],[\tau'])=-1\}$. Let 
$$\beta(A,A')=\frac{2\left(|E_-(A)|+|E_-(A')|+|E_+(A,A')|\right)+|\partial(A\sqcup A')|}{\vol(A\sqcup A')}$$
where $|\partial A|$ is the number of the edges of $G_{up}^-(S_d)$ that cross $A$ and $S_d\setminus A$,   $\vol(A)=\sum_{\tau\in A}\deg_\tau$ and $\deg_\tau=\#\{\sigma\in S_{d+1}:\tau\subset \sigma\}$. 

Then we introduce the $k$-th  Cheeger constant   on $S_d$:
$$h_k(S_d)=\min\limits_{\text{disjoint } A_1,A_2,\ldots,A_{2k-1},A_{2k}\text{ in }S_d}\max\limits_{1\le i\le k}\beta(A_{2i-1},A_{2i}).$$
It is interesting that $h_k(S_d)=0$ if and only if $G_{up}^-(S_d)$ has  $k$ balanced components.

\begin{remark}
 For $d=0$,  the constant $h_k(S_0)$ reduces to the $k$-way Cheeger constant of a graph  \cite{LGT12}.
\end{remark}

\begin{theorem}\label{thm:anti-signed-Cheeger}
For any simplicial complex and every $d\ge 0$,
\begin{equation}\label{eq:Cheeger-1-complex}
\frac{ h_1(S_d)^2}{2(d+1)}\le d+2-\lambda_n(\Delta^{up}_d)\le 2h_1(S_d),
\end{equation}
where $n=\#S_d$. 
Moreover, there exists an absolute constant $C$  such that for any simplicial complex, and any $k\ge 1$, 
\begin{equation}\label{eq:Cheeger-k-complex}
\frac{ h_k(S_d)^2}{Ck^6(d+1)}\le d+2-\lambda_{n+1-k}(\Delta^{up}_d)\le 2h_k(S_d).
\end{equation}

\end{theorem}

\begin{proof}
We first show that 
$$ d+2-\lambda_{n-i+1}(\Delta^{up}_d)=(d+1)\lambda_i(\Delta(G_{up}^-(S_d))),\;\;\;i=1,\cdots,n. $$
In fact, it can be immediately derived by the identity regarding the  Rayleigh quotients:
$$d+2-\frac{\sum\limits_{\sigma\in S_{d+1}}\left(\sum_{\tau\in S_d,\tau\subset\sigma}\mathrm{sgn}([\tau],\partial[\sigma])f(\tau)\right)^2}{\sum_{\tau\in S_d}\deg_\tau f(\tau)^2}=(d+1)\frac{\sum_{[\tau]\sim[\tau']}\left(f(\tau)-\mathrm{sgn}(\tau,\tau')f(\tau'))\right)^2}{\sum_{\tau\in S_d}\widetilde{\deg}_\tau f(\tau)^2}$$
where $[\tau]\sim[\tau']$ represents an  edge in the underlying  anti-signed graph  $G_{up}^-(S_d)$, and $\widetilde{\deg}_\tau=(d+1)\deg_\tau$ is the degree of $\tau$ in  $G_{up}^-(S_d)$. 

Note that $\frac{1}{d+1}h_k(S_d)$ also  indicates the $k$-th Cheeger constant of the signed graph $G_{up}^-(S_d)$. By the Cheeger inequality and the higher order Cheeger inequalities in   \cite{AtayLiu}, we have  $\frac{\lambda_1(\Delta(G_{up}^-(S_d))}{2}\le \frac{h_1(S_d)}{d+1}\le \sqrt{2\lambda_1(\Delta(G_{up}^-(S_d))}$. And there exists  an absolute constant $C$  such that for any signed graph and any $k\ge 1$, $\frac{\lambda_k(\Delta(G_{up}^-(S_d))}{2}\le \frac{h_k(S_d)}{d+1}\le Ck^3 \sqrt{\lambda_k(\Delta(G_{up}^-(S_d))}$. In consequence, we obtain
$$ \frac{d+2-\lambda_{n}(\Delta^{up}_d)}{2}\le h_1(S_d)\le \sqrt{2(d+1)(d+2-\lambda_{n}(\Delta^{up}_d))}$$
and
$$\frac{d+2-\lambda_{n+1-k}(\Delta^{up}_d)}{2}\le h_k(S_d)\le Ck^3\sqrt{(d+1)(d+2-\lambda_{n+1-k}(\Delta^{up}_d))} .$$
Then, we have verified \eqref{eq:Cheeger-1-complex} and \eqref{eq:Cheeger-k-complex}. 
\end{proof}

By Theorem \ref{thm:anti-signed-Cheeger}, $\lambda_n(\Delta^{up}_d)=d+2$ if and only if  $h_1(S_d)=0$,  if and only if the underlying anti-signed graph $G_{up}^-(S_d)$ has a  balanced component.  

In contrast to Proposition \ref{pro:maximal eigenvalue-balance} on the multiplicity of $(d+2)^{p-1}$ for $\Delta_{d,p}^{up}$,  the multiplicity of $1$ for $\Delta_{d,1}^{up}$ has a quite different characterization. To state this, we show the following concepts and results.

A balanced (resp. antibalanced) clique $S$ is a subset of $S_d$ such that $S$ induces a balanced (resp. antibalanced) complete subgraph in $G_{up}^-(S_d)$. Similar to Theorem 1 in \cite{Zhang18}, we can prove the following:

\begin{pro}
The maximum eigenvalue of $\Delta_{d,1}^{up}$ is  $1$, and the multiplicity of the eigenvalue  $1$, denoted by $m_1(S_d)$,  satisfies the sandwich inequality
$$\widetilde{\alpha}(S_d)\le m_1(S_d)\le \widetilde{\kappa}(S_d)$$
where $\widetilde{\alpha}(S_d)=\max\{p+2q:\exists \text{ pairwise non-adjacent }p\text{ balanced cliques and }q\text{ antibalanced cliques}\}$, $\widetilde{\kappa}(S_d)=\min\{p+2q:\exists~p\text{ balanced cliques and }q\text{ antibalanced cliques covering }S_d\}$.  
\end{pro}

According to Theorem \ref{thm:nodal-inertia-bound}, we also have the inertia bound and the nodal domain theorem:

\begin{pro}
Let $\alpha$ be the independence number of $G_{up}^-(S_d)$. Then, 
$$    \alpha\leq \min\{\#\{i:\lambda_i(\Delta_{d,p}^{up})\leq 1\},\#\{i:\lambda_i(\Delta_{d,p}^{up})\geq 1\}\}.$$
For any eigenfunction $f$ w.r.t. $\lambda_i(\Delta_{d,p}^{up})$ whose multiplicity is $r$, the number  of connected components of the support set of $f$ is smaller than or equal to $\min\{ i+r-1, n-i+r\}$. 
\end{pro}

Next, we show some  results on the smallest non-trivial eigenvalue of $\Delta^{up}_{d,p}$.

\begin{pro}\label{pro:lambda(d+1)}
Given a simplicial complex $K$, for any  $ 0\le d<\dim K$, and $p\ge 1$,  $\lambda_{d+1}(L^{up}_{d,p})=\lambda_{d+1}(\Delta^{up}_{d,p})=0$. 
\end{pro}

\begin{proof}
It suffices to prove that the multiplicity of the eigenvalue zero is larger than or equal to $d+1$.

By Theorem 3.1 in \cite{HorakJost}, we can derive that  the multiplicity of the eigenvalue zero of $L^{up}_d$ is
$\dim \mathrm{Ker}(B_{d+1}^\top)=\dim \mathrm{Image}(B_{d}^\top)+\dim \tilde{H}^{d}(K,\R)=\mathrm{rank}(B_{d})+\mathrm{rank}(\tilde{H}^{d}(K,\R))\ge \mathrm{rank}(B_{d})$. 

Since every $(d+1)$-simplex has $(d+2)$ sub-simplices of  dimension $d$, the incidence matrix $B_{d}$ has at least $(d+2)$ nonzero columns. And based on this fact, we can further verify that $\mathrm{rank}(B_{d})\ge d+1$.  Finally, it is obvious that the multiplicities of the eigenvalue zero of $L^{up}_{d,p}$,    $\Delta^{up}_{d,p}$ and  $L^{up}_d$ coincide. The proof is completed.
\end{proof}

It is well-known  that $\lambda_2(L_0^{up})>0$ if and only if $\mathrm{rank}(\tilde{H}^{0}(K,\R))=0$, i.e., $K$ is connected.  For $\lambda_{d+2}(L^{up}_d)$ with $d\ge 1$, we have
\begin{pro}\label{pro:lambda(d+2)}
Given a pure  simplicial complex $K$, $ 1\le d<\dim K$ and $p\ge 1$, we have   $\lambda_{d+2}(\Delta^{up}_{d,p})>0$ (or  $\lambda_{d+2}(L^{up}_{d,p})>0$) if and only if $K$ is a simplex of dimension $(d+1)$. 
\end{pro}

\begin{proof}
Without loss of generality, we only prove the case of $p=2$. 
For any  $ 0\le d<\dim K$,  $\lambda_{d+2}(L^{up}_d)>0$ if and only if $\mathrm{rank}(B_{d})=d+1$ and $\mathrm{rank}(\tilde{H}^{d}(K,\R))=0$. If $K$ is 
a simplex of dimension $(d+1)$, it is easy to check that $\mathrm{rank}(B_{d})=d+1$,  $\mathrm{rank}(\tilde{H}^{d}(K,\R))=0$, and $\lambda_{d+2}(L^{up}_d)=d+2$.

For the converse,   by the proof of   Proposition \ref{pro:lambda(d+1)}, we can verify that the number of 
$(d+1)$-simplexes in $K$ is one. Since  $K$ is pure, $K$ must be a simplex of dimension $(d+1)$.
\end{proof}

Let $I_d=\dim \mathrm{Image}(B_d^\top)+1=\mathrm{rank}(B_d)+1$ and let $k_d=\dim \mathrm{Ker}(B_{d+1}^\top)+1$. Then, 
$$\lambda_{I_d}(\Delta^{up}_{d,p})=\min\limits_{x\bot \mathrm{Image}(B_d^\top)}\frac{\|B_{d+1}^\top\vec x\|_p^p}{\min\limits_{y\in \mathrm{Image}(B_d^\top)}\|\vec x+\vec y\|_{p,\deg}^p}$$ 
and 
$\lambda_{d+2}(\Delta^{up}_{d,p})\le \lambda_{I_d}(\Delta^{up}_{d,p})\le \lambda_{k_d}(\Delta^{up}_{d,p})$, where $\|\vec x\|_{p,\deg}^p:=\sum_{\tau\in S_d}\deg_\tau |x_\tau|^p$. Clearly, $\lambda_{k_d}(\Delta^{up}_{d,p})$ is the  smallest non-vanishing (nonzero) eigenvalue of the normalized $d$-th up $p$-Laplacian. We call $\lambda_{I_d}(\Delta^{up}_{d,p})$  the  first (smallest) non-trivial eigenvalue of the $p$-Laplacian $\Delta^{up}_{d,p}$. 

It is interesting  that for any $p\ge 1$, $\lambda_{I_d}(\Delta^{up}_{d,p})= \lambda_{k_d}(\Delta^{up}_{d,p})$ if and only if $\tilde{H}^{d}(K,\R)=0$. 
Also, similar to Proposition \ref{pro:lambda(d+2)}, for a  pure  simplicial complex, the equality $\lambda_{d+2}(\Delta^{up}_{d,p})= \lambda_{I_d}(\Delta^{up}_{d,p})$ holds  if and only if $K$ is a $(d+1)$-simplex or $\mathrm{rank}(\tilde{H}^{d}(K,\R))\ne0$.

\begin{remark}
For the case of $p=2$, the smallest non-trivial eigenvalue $\lambda_{I_d}(\Delta^{up}_{d})$ of the  normalized  up Laplacian has been used to derive a Cheeger inequality in \cite{SKM14}.
\end{remark} 

In the table below, we show the relations among the spectra of the normalized Laplacians on $S_d$,  and on the underlying signed graphs $G_{up}^-(S_d)$ as well as $G_{up}(S_d)$ associated to  $S_d$. It can be seen that their eigenvalues  $\lambda_n(G_{up}^-(S_d))\ge \cdots\ge \lambda_1(G_{up}^-(S_d))$, $\lambda_1(\Delta^{up}_d)\le\cdots\le \lambda_n(\Delta^{up}_d)$ and $\lambda_1(G_{up}(S_d))\le\cdots\le \lambda_n(G_{up}(S_d))$ satisfy the simple equalities:  $\lambda_{n+1-k}(G_{up}^-(S_d))=\frac{1}{d+1}(d+2-\lambda_k(\Delta^{up}_d))$, $\lambda_k(G_{up}(S_d))=\frac{1}{d+1}(d+\lambda_k(\Delta^{up}_d))$, and $\lambda_{n+1-k}(G_{up}^-(S_d))=2-\lambda_k(G_{up}(S_d))$, where $k=1,\cdots,n$ and $n=\#S_d$.

In summary, we use $\Delta(G_{up}(S_d))$ (resp., $\Delta(G_{up}^{-}(S_d))$) to denote the normalized Laplacians on the signed graph  $G_{up}(S_d)$ (resp., anti-signed graph $G_{up}^{-}(S_d)$). Then,  the eigenvalues of these operators have the relation:
$$
\begin{matrix}
\text{Specturm of }\Delta_d^{up} & ~& \text{Specturm of }\Delta(G_{up}(S_d)) & ~& \text{Specturm of }\Delta(G_{up}^{-}(S_d)) \\
\\
   &~&   &~&    \\
 0 &~&  \frac{d}{d+1}  &~&  \frac{d+2}{d+1} \\
\vdots  &~& \vdots  &~&  \vdots \\ \lambda &\Longleftrightarrow& \frac{\lambda+d}{d+1}  &\Longleftrightarrow&  \frac{d+2-\lambda}{d+1} \\
\vdots  &~& \vdots  &~&  \vdots \\ d+2 &~& 2 &~&  0 \\
\end{matrix}
$$
that is,  $\lambda$ is an eigenvalue of $\Delta_d^{up}$ if and only if $\frac{\lambda+d}{d+1} $ is an eigenvalue of $\Delta(G_{up}(S_d))$ if and only if $\frac{d+2-\lambda}{d+1} $ is an eigenvalue of  $\Delta(G_{up}^{-}(S_d))$.  
In addition, the multiplicity of the eigenvalue $0$ of $\Delta_d^{up}$ is larger than  or equal to $d+1$, while the multiplicity of the eigenvalue $d+2$ of $\Delta_d^{up}$ agrees  with the number of balanced components of $G_{up}^{-}(S_d)$.


By Theorem \ref{piecewise-linear-vertex},  there exists an  extreme point $\vec x$ which is also an eigenvector associated to the eigenvalue $\lambda_{I_d}(\Delta^{up}_{d,1})$. Indeed, based on the concepts and results in Section \ref{sec:structure-eigenspace}, one can check that the extreme points of the function pair $(\|B_{d+1}^\top\cdot\|_1,\|\cdot\|_{1,\deg})$ belong to $ \mathrm{cone}\{-N,\cdots,-1,0,1,\cdots,N\}^{\#S_d}$ for some positive integer $N$ (if $d=0$, one can take $N=1$). This means that $\lambda_{I_d}(\Delta^{up}_{d,1})$ can be expressed as a combinatorial optimization, or equivalently, an integer  programming with constraint on $\{-N,\cdots,-1,0,1,\cdots,N\}^n$, and thus we would like to call $$h(S_d):=
\min\limits_{x\bot^1 \mathrm{Image}(B_d^\top)}\frac{\|B_{d+1}^\top\vec x\|_1}{\|\vec x\|_{1,\deg}}=\lambda_{I_d}(\Delta^{up}_{d,1})$$
the Cheeger constant on $S_d$, where $\vec x\bot^1\vec y$ indicates that $\vec x$ is $\|\cdot\|_{1,\deg}$-orthogonal to $\vec y$ (see Section \ref{sec:structure-eigenspace} for the definition).

We have a  combinatorial explanation of the Cheeger constant $h(S_d)$ using the language of multi-sets  in combinatorics. A multiset can be formally defined as a  pair $(S,m)$, where $S$ is the underlying set of the multiset, formed from its distinct elements, and $m:S\to\mathbb{Z}$ is an integer-valued function, giving the {\sl multiplicity}. For convenience, we usually write $S$ instead of $(S,m)$, and we use $|S|
:=\sum_{s\in S}|m(s)|$ to indicate the {\sl size} of the multiset $S$. 

Now we concentrate on the underlying set $S_d$. We use $S\subset_N S_d$ to
indicate that $S$ is a multiset on  the underlying set $S_d$ with
multiplicities in $\{-N,\cdots,0,\cdots,N\}$. For such an $S$, let its
coboundary  $\partial^*_{d+1}S$ be the  multiset on the underlying set $S_{d+1}$ such that each $\sigma\in S_{d+1}$ has the multiplicity $\sum_{\tau\in S_d}m(\tau)\mathrm{sgn}([\tau],\partial[\sigma])$, where $m(\tau)$ is the multiplicity of $\tau$ in $S$. Denote by $\vol(S)=\sum_{\tau\in S_d}\deg_\tau |m(\tau)|$  the volume  of the multiset $S$.  

It should be noted that $\tilde{H}^{d}(K,\R)\ne 0$ if and only if  $h(S_d)=0$.  More precisely, 
according to Theorems \ref{thm:smallest-nonzero} and \ref{piecewise-linear-vertex}, as well as the results in Section \ref{sec:structure-eigenspace}, there exists $N\in\mathbb{Z}_+$ such that
\begin{align}
h(S_d)&=\min\limits_{\substack{S\subset_N S_d \\ S\neq \partial^*_{d}(T),\forall T\subset_N S_{d-1}}}\frac{|\partial^*_{d+1} S|}{\min\limits_{S':\partial^*_{d+1}S'=\partial^*_{d+1}S}\vol(S')}\label{eq:combinatorial-h(S_d)}
\\&\xlongequal[]{\text{if } \tilde{H}^{d}(K,\R)= 0}\min\limits_{\substack{S\subset_N S_d \\ \partial^*_{d+1}S\ne \varnothing}}\frac{|\partial^*_{d+1} S|}{\min\limits_{S':\partial^*_{d+1}S'=\partial^*_{d+1}S}\vol(S')}>0, \notag  
\end{align}
In order to further  understand the formula \eqref{eq:combinatorial-h(S_d)},  below we show an equivalent reformulation of \eqref{eq:combinatorial-h(S_d)} using the language of norms on cochain groups. 


The norm $\|\cdot\|_{1,\deg}$ on $C^{d}(K)$ induces a quotient norm on
$C^{d}(K)/\mathrm{image}(\delta_{d-1})$, which will be  denoted by $\|\cdot\|$
for simplicity.  More precisely, for any equivalence class $[\vec x]\in C^{d}(K)/\mathrm{image}(\delta_{d-1})$,  let $\| [\vec x]\|=\inf\limits_{x'\in [x]}\|\vec x'\|_{1,\deg}$. Then 
$$h(S_d)=
\min\limits_{0\ne [x]\in C^{d}(K)/\mathrm{image}(\delta_{d-1})}\frac{\|\delta_d\vec x\|_1}{\| [\vec x]\|}=\min\limits_{0\ne [x]\in C^{d}(K,\mathbb{Z})/\mathrm{image}(\delta_{d-1})}\frac{\|\delta_d\vec x\|_1}{\| [\vec x]\|}$$
and it is interesting that in the case of $\tilde{H}^{d}(K,\R)= 0$, 
$$h(S_d)=\min\limits_{y\in \mathrm{image}(\delta_{d})}\frac{\|\vec y\|_1}{\|\vec y\|_{\mathrm{fil}}}=\frac{1}{\max\limits_{y\in \mathrm{image}(\delta_{d})}\|\vec y\|_{\mathrm{fil}}/\|\vec y\|_1}=\frac{1}{\|\delta_d^{-1}\|_{\mathrm{fil}}}$$
where $\|\vec y\|_{\mathrm{fil}}:=\inf\limits_{x\in\delta_d^{-1} (y)}\|\vec x\|_{1,\deg}$ is the filling norm of $\vec y$, and $\|\delta_d^{-1}\|_{\mathrm{fil}}$ is called the  filling profile by Gromov (see Section 2.3 in \cite{Gromov}). 

{
\begin{remark}
Steenbergen, Klivans and  Mukherjee \cite{SKM14} introduced the following  Cheeger constant 
$$h^d(K):=\min\limits_{\varphi\in C^d(K,\mathbb{Z}_2)\setminus\mathrm{Im\,}\delta}\frac{\|\delta\varphi\|}{\min\limits_{\psi\in\mathrm{Im\,}\delta}\|\varphi+\psi\|}$$
which satisfies
$$h^d(K)=0\Longleftrightarrow \tilde{H}^d(K,\mathbb{Z}_2)\ne0,\;\; \forall d\ge 0,$$
where $\|\cdot\|$ is the Hamming norm on $C^d(K,\mathbb{Z}_2)$ (i.e. $l^1$-norm on $\mathbb{Z}_2^{n}$ with $n=\#S_d$) . According to the examples and theorems in \cite{DK12,GS15,GW16,PRT15,SKM14}, all the Cheeger constants defined using cohomology (or homology) with $\mathbb{Z}_2$-coefficients cannot derive a general two-side Cheeger inequality like the graph setting. It is worth noting that our Cheeger constant can be reformulized as 
$$h(S_d)=\min\limits_{\varphi\in C^d(K,\mathbb{Z})\setminus\mathrm{Im\,}\delta}\frac{\|\delta\varphi\|_1}{\min\limits_{\psi\in\mathrm{Im\,}\delta}\|\varphi+\psi\|_{1,\deg}}$$
where we use  $\mathbb{Z}$-coefficients instead of the  $\mathbb{Z}_2$-coefficients, and we use the  (weighted) $l^1$-norm instead of the Hamming norm. It is clear that 
$$h(S_d)=0 \Longleftrightarrow \tilde{H}^d(K,\R)\ne0,\;\; \forall d\ge 0.$$
\end{remark}}

For the case of $d=0$, we can take $N=1$, and then $h(S_0)$ reduces to the usual Cheeger constant on graphs. The following  preliminary result indicates that such a   constant $h(S_d)$ is probably a good candidate for  Cheeger-type  inequalities.

\begin{pro}\label{pro:rough-Cheeger}
Suppose that $\deg_\tau>0$, $\forall \tau\in S_d$. Then, $$\frac{h^2(S_d)}{\#S_{d+1}}\le \lambda_{I_d}(\Delta_d^{up})\le \vol(S_d)h(S_d)$$
and for any $p\ge 1$,
$$\frac{h^p(S_d)}{|\#S_{d+1}|^{p-1}}\le \lambda_{I_d}(\Delta_{d,p}^{up})\le \vol(S_d)^{p-1}h(S_d).$$
\end{pro}
\begin{proof}
For simplicity, we denote  $h=h(S_d)$ and take $\lambda=\lambda_{I_d}(\Delta_d^{up})$. 
We shall prove $\frac{\min\limits_{\tau\in S_{d}}\deg_\tau}{\#S_{d+1}}h^2\le \lambda\le \vol(S_d)h^2$. 

Let  $k=\mathrm{rank}(B_{d})$. Then $\lambda$ and $h$ are the $(k+1)$-th min-max eigenvalues of the $d$-th up Laplacian and  the $d$-th up 1-Laplacian, respectively. We only need to prove that, for any $k\ge 1$, 
$$\sqrt{\frac{1}{\sum\limits_{\tau\in S_d}\deg_\tau} \lambda_k}\le h_k\le \sqrt{\frac{\#S_{d+1}}{\min\limits_{\tau\in S_{d}}\deg_\tau} \lambda_k}.$$
In fact, it is easy to see that 
$$\min\limits_\tau\deg_\tau  \le \frac{\|\vec x\|_{1,\deg}^2}{\|\vec x\|_{2,\deg}^2}\le \sum_{\tau\in S_d}\deg_\tau \; \text{ 
and }\; 1\le \frac{\|B_{d+1}^\top\vec x\|_1^2}{\|B_{d+1}^\top\vec x\|_2^2}\le \#S_{d+1}.$$
Hence
$$\frac{1}{\sum_{\tau\in S_d}\deg_\tau}\frac{\|B_{d+1}^\top\vec x\|_2^2}{\|\vec x\|_{2,\deg}^2} \le\frac{\|B_{d+1}^\top\vec x\|_1^2}{\|\vec x\|_{1,\deg}^2}\le   \frac{\#S_{d+1}}{\min\limits_\tau\deg_\tau}\frac{\|B_{d+1}^\top\vec x\|_2^2}{\|\vec x\|_{2,\deg}^2}.$$
The proof of $\frac{h^2(S_d)}{\#S_{d+1}}\le \lambda_{I_d}(\Delta_d^{up})\le \vol(S_d)h(S_d)$ is then  completed by noting that $h\le 1\le \deg_\tau$,  $\forall \tau\in S_d$.   The case of $\Delta_{d,p}^{up}$ is similar.
\end{proof}

\begin{remark}\label{remark:down-Cheeger}
 We can also define the down Cheeger constant 
$$h_{down}(S_d):=\min\limits_{x\bot^1 \mathrm{Image}(B_{d+1})}\frac{\|B_{d}\vec x\|_1}{\|\vec x\|_{1,\deg}}=\lambda_{I_{d+1}}(\Delta^{down}_{d,1})$$
which  possesses a combinatorial  reformulation that is similar to  \eqref{eq:combinatorial-h(S_d)}. 

Consider a $d$-dimensional  combinatorial manifold $K$, that is, a $d$-dimensional topological manifold possessing  a simplicial complex structure. As a manifold, we assume that $K$ is connected and  has no  boundary. Then, the down adjacency relation  induces a graph on $S_d$, and we have the Cheeger inequality: 
$$\frac{h^2_{down}(S_d)}{2}\le \lambda_{2}(\Delta_d^{down}) 
\le 2h_{down}(S_d).$$

\end{remark}

\begin{defn}
Let $M$ be a $d$-dimensional orientable compact closed Riemannian manifold. A triangulation $T$  of $M$ is {\sl $C$-uniform} if there exists $C>1$ such that for any two $d$-simplexes $\triangle$ and $\triangle'$ in the triangulation $T$, $$\frac1C <\frac{\mathrm{diam}(\triangle)}{\mathrm{diam}(\triangle')}<C\;\;\text{ and }\;\;\frac1C <\frac{\mathrm{diam}(\triangle)}{\mathrm{vol}(\triangle)^{\frac1d}}<C
.$$

A triangulation $T$  of $M$ is {\sl uniform} if there exist $N>1$ and  $C>1$  such that either the number of vertices of $T$ is smaller than $N$, or $T$  is $C$-uniform. The constants $N$ and $C$ are called the {\sl uniform parameters} of the triangulation.
\end{defn}
\begin{theorem}\label{thm:Cheeger-manifold-complex}
Let $M$ be an  orientable, compact, closed Riemannian manifold of dimension $(d+1)$. Let  $K$ be a simplicial complex which is  combinatorially equivalent to a uniform triangulation of $M$. Then, there is a Cheeger inequality
$$ \frac{h^2(S_d)}{C}\le \lambda_{I_{d}}(\Delta_{d}^{up}) \le C\cdot h(S_d), $$
where $C$ is a uniform constant which  is independent of the choice of $K$. In
addition, $h(S_d)>0$ if and only if $H_1(K)=0$ (or equivalently, $H_1(M)=0$).
\end{theorem}
\begin{proof}
By Proposition \ref{pro:rough-Cheeger}, $\lambda_{I_{d}}(\Delta_{d}^{up})=0$ if and only if  $h(S_d)=0$. So, it suffices to assume that $h(S_d)>0$, i.e., $\tilde{H}^d(M)=\tilde{H}^d(K)=0$. Since $M$ and $K$ are of  dimension $(d+1)$, Poincar\'e duality implies that  $\tilde{H}_1(M)=\tilde{H}^d(M)=0$.  

We may assume  without loss of generality that $M$ is  simply connected, and the  triangulation is $C$-uniform for some $C>1$, and $S_d(K)$ has $n$ elements, where $n$ is a sufficiently large integer. 

For any $\epsilon>0$, there exist $N>0$ such that any $C$-uniform triangulation with at least $N$ facets satisfies  $\frac{1}{3C^2}\epsilon<\mathrm{diam}(\triangle)<\epsilon$, $\forall \triangle$. Here, we also regard the uniform triangulation as a uniform  $\epsilon$-net.

\begin{enumerate}
\item[Claim 1] For the down Cheeger constant, we have
$$ \frac{d+2}{4}h^2_{down}(S_{d+1})\le \lambda_{I_d}(\Delta_d^{up}) \le (d+2)h_{down}(S_{d+1}). $$

Proof: This is derived by the Cheeger inequality
$$ \frac{h^2_{down}(S_{d+1})}{2}\le \lambda_{2}(\Delta_{d+1}^{down}) \le 2h_{down}(S_{d+1}) $$
proposed in Remark \ref{remark:down-Cheeger}, 
and the duality property 
$\lambda_{I_d}(\Delta_d^{up})=\frac{d+2}{2}\lambda_{2}(\Delta_{d+1}^{down})$. 
\item[Claim 2] The Cheeger constant $h(S_d)$ and the down Cheeger constant  $h_{down}(S_{d+1})$ satisfy $h(S_d)\sim  h_{down}(S_{d+1})$,  
i.e., there exists a uniform constant  $C>1$ such that 
$\frac1C h_{down}(S_{d+1})\le h(S_d)\le C  h_{down}(S_{d+1})$.

The proof is divided into the following two claims.
\begin{enumerate}
\item[Claim 2.1] $\frac1\epsilon h_{down}(S_{d+1})\sim  h(M)$

Proof: Let $G$ be the graph with $n:=\#S_{d+1}$ vertices located in the barycenters of all $(d+1)$-simplexes, such that two vertices form an edge in $G$ if and only if these two $d$-simplexes are down adjacent. We may call $G$ the underlying graph of the  triangulation. 

Note that $ h_{down}(S_{d+1})$ also indicates the Cheeger constant of the  unweighted underlying graph $G$. 
An approximation  approach developed in \cite{TMT20,TrillosSlepcev-16} implies  that the Cheeger constant of a   uniform triangulation should approximate the
Cheeger constant of the manifold when we equip the edges of the underlying graph of the triangulation with  appropriate  weights (related to $\epsilon$).   In fact, since $G$ is  a underlying graph of the triangulation, we may assume that   $G$ is embedded in the manifold $M$, and the  distribution of the vertices of $G$ is uniform\footnote{The vertices of $G$ are well-distributed on $M$.}. Then, according to the approximation theorems in \cite{TMT20,TrillosSlepcev-16}, by adding  appropriate  weights (related to $\epsilon$)\footnote{The weight of an edge $\{u,v\}$ is determined by the distance of $u$ and $v$ in $M$, which is about $O(\epsilon)$.} on $G$, the Cheeger constant of $G$ (with appropriate  edge weights) would approximate $h(M)$ (i.e., the difference of $h(M)$ and the Cheeger constant of the weighted graph $G$ is bounded by  $h(M)/2$ whenever $\epsilon$ is sufficiently small). 
We can then adopt the same approximation  approach in
\cite{TMT20,TrillosSlepcev-16} (i.e., a slight modification of the  approximation theorem 
in \cite{TMT20,TrillosSlepcev-16,TrillosSlepcev-15}) 
 to derive that $ \frac1\epsilon h_{down}(S_{d+1})\sim  h(M)$.

\item[Claim 2.2] $\frac1\epsilon h(S_d)\sim  h(M)$ whenever $H_1(M)=0$.

Proof: It is well-known that  $H_1(M)=0$ if and only if $H^d(M)=0$ if and only if $\mathrm{Ker}(\delta_d)=\mathrm{Im}(\delta_{d-1})$,  since $M$ is a compact closed manifold of dimension $(d+1)$. Thus,  $$h(S_d)= \min\limits_{x\not\in  \mathrm{Ker}(\delta_{d})}\frac{\sum\limits_{\sigma\in S_{d+1}}|\sum\limits_{\tau\in S_d}\mathrm{sgn}(\tau,\partial\sigma)x_\tau|}{\min\limits_{z\in \mathrm{Ker}(\delta_d)}\sum\limits_{\tau\in S_d}2|x_\tau+z_\tau|}.$$
By the duality theorem  (see Lemma \ref{lem:dual} and Proposition \ref{pro:p-Lap-dual}), we further have
$$h(S_d)=\frac{\max\limits_{\sigma \mathop{\sim}\limits^{\text{down}} \sigma'}\frac12|y_\sigma-y_{\sigma'}|}{\min\limits_{t\in\R}\max\limits_{\sigma\in S_{d+1}}|y_\sigma+t|}$$
and by Theorem \ref{piecewise-linear-vertex}, there is no difficulty to check that the  optimization in the right hand side coincides with  
$$\min\limits_{\min\limits_{\sigma} y_\sigma+\max\limits_\sigma y_\sigma=0}\frac{\max\limits_{\sigma \mathop{\sim}\limits^{\text{down}} \sigma'}|y_\sigma-y_{\sigma'}|}{2\max\limits_{\sigma}|y_\sigma|}=\frac{1}{\mathrm{diam}(G)}$$
where $\mathrm{diam}(G)$ indicates the combinatorial diameter of $G$. We remark here that we indeed rewrite   $h(S_d)$ as the smallest nontrivial eigenvalue of the $\infty$-Laplacian, which agrees with  $1/\mathrm{diam}(G)$.  This argument is similar to a theorem  in  \cite{JLM-99}.

Finally, since the  triangulation is $C$-uniform, it is easy to see that 
$$\frac1\epsilon h(S_d)= \frac{1}{\epsilon\cdot \mathrm{diam}(G)}\sim \frac{1}{ \mathrm{diam}(M)} .$$
Hence,  $\frac1\epsilon h(S_d)\sim h(M)$.
\end{enumerate}
\end{enumerate}
The proof is then  completed by combining all the statements above.
\end{proof}
\begin{remark}
\begin{itemize}
\item The constant $C$ in  Theorem \ref{thm:Cheeger-manifold-complex} depends
  on the uniform parameters of the triangulation, and the ambient
  manifold. We hope that it is possible to find a new approach to get a uniform
  constant only depends on the dimension $d$.
\item Under the same condition of Theorem \ref{thm:Cheeger-manifold-complex},
  we further have $\frac{\lambda_{k_{d}}(\Delta_{d,1}^{up})^2}{C}\le
  \lambda_{k_{d}}(\Delta_{d}^{up}) \le
  C\lambda_{k_{d}}(\Delta_{d,1}^{up})$. This inequality  coincides with the
  Cheeger inequality in Theorem \ref{thm:Cheeger-manifold-complex} if and only
  if $H_1(M)=0$.

\item A modification of the proof can deduce that $\frac{1}{\mathrm{diam}(G)}\sim \lambda_2(G)$ whenever $G$ can be uniformly embedded  into such a typical manifold, where $\lambda_2(G)$ is the second smallest  eigenvalue of the normalized  Laplacian on $G$. To some extent, this 
can be regarded as a higher dimensional analog of the main result 
in \cite{LouderSouto}. 
{
\item Inspired by the approximation theory for Laplacians on triangulations  of manifolds proposed by Dodziuk \cite{Dodziuk76} and  Dodziuk-Patodi \cite{DP76}, we hope that it is possible to develop an approximation theory for our Cheeger constants on triangulations of manifolds. }

\end{itemize}

\end{remark}

Motivated by the above results and  discussions, we then present the following open problem for Cheeger inequalities on simplicial complexes.

\textbf{Conjecture}: There exists $C_d>0$ which only depends on $d\in\mathbb{N}$, such that
$$\frac{h^2(S_d)}{C_d}\le \lambda_{I_d}(\Delta_d^{up}) \le C_dh(S_d),\;\;\text{ and }\;\; \frac{h^2_{down}(S_d)}{C_d}\le \lambda_{I_d}(\Delta_d^{down}) \le C_dh_{down}(S_d).$$

\subsection{Other applications on extension and duality}

We show new equalities based on the theory of duality in Section
\ref{sec:dual}, and by employing these equalities, we immediately get  the
dual optimization of the inner problem in the Dinkelbach-type scheme
\cite{JostZhang-PL}, and the dual formulation of the $l^p$-polarization
(Chebyshev)  constant \cite{AmbrusNietert}.  In addition, applying the dual
principle to Lov\'asz  extension, we obtain  new equivalent continuous
representations of the Cheeger constant, maxcut, dual Cheeger  quantity on a graph.  
\begin{pro}\label{pro:dual-inner-problem}
Let $F:\R^m\to[0,+\infty)$ be a positive-definite and one-homogeneous convex function, and let $T:\R^m\to\R^n$ be a linear transformation. For any convex body $\mathbb{B}\subset \R^n$ that contains $\vec0$ as its inner point, and for any $\vec u\in \R^n$, we have
$$\min\limits_{x\in \mathbb{B}}(F(T\vec x)-\vec x\cdot\vec u)=-\min\limits_{F^*(\vec y)\le 1}h_{\mathbb{B}}(\vec u-T^\top\vec y)\,\text{ and }\,\max\limits_{x\in \mathbb{B}}(F(T\vec x)-\vec x\cdot\vec u)=  \max\limits_{F^*( y)\le 1}h_{\mathbb{B}}(T^\top\vec y-\vec u)  ,$$
where $h_{\mathbb{B}}$ is the support function of ${\mathbb{B}}$, and $F^*$ is the dual function of $F$.
\end{pro}

\begin{proof}
We only need to prove the following equivalent statement:

Let $F:\R^m\to[0,+\infty)$ and $G:\R^n\to[0,+\infty)$ be positive-definite and one-homogeneous convex functions. For any matrix of order $m\times n$, and for any $\vec u\in \R^n$, we have
$$\min\limits_{G(x)\le 1}(F(T\vec x)-\vec x\cdot\vec u)=-\min\limits_{F^*(\vec y)\le 1}G^*(\vec u-T^\top\vec y)$$
and
$$\max\limits_{G(x)\le 1}(F(T\vec x)-\vec x\cdot\vec u)=  \max\limits_{F^*( y)\le 1}G^*(T^\top\vec y-\vec u)  .$$

The proof is direct. In fact, by the definition of duality, there holds
\begin{align*}
\min\limits_{G(x)\le 1}(F(T\vec x)-\vec x\cdot\vec u)&=\min\limits_{G(x)\le 1}\left( \max\limits_{F^*( y)\le 1}  T\vec x\cdot \vec y-\vec x\cdot\vec u\right)
=\min\limits_{G(x)\le 1} \max\limits_{F^*( y)\le 1}  \vec x\cdot (T^\top\vec y-\vec u)\\&=\max\limits_{F^*( y)\le 1} \min\limits_{G(x)\le 1}  \vec x\cdot (T^\top\vec y-\vec u)=\max\limits_{F^*( y)\le 1} (-\max\limits_{G(x)\le 1}  \vec x\cdot (\vec u-T^\top\vec y))\\&=\max\limits_{F^*( y)\le 1}-G^*(\vec u-T^\top\vec y)=-\min\limits_{F^*( y)\le 1}G^*(\vec u-T^\top\vec y)    
\end{align*}
and
\begin{align*}
\max\limits_{G(x)\le 1}(F(T\vec x)-\vec x\cdot\vec u)&=\max\limits_{G(x)\le 1}\left( \max\limits_{F^*( y)\le 1}  T\vec x\cdot \vec y-\vec x\cdot\vec u\right)
=\max\limits_{G(x)\le 1} \max\limits_{F^*( y)\le 1}  \vec x\cdot (T^\top\vec y-\vec u)\\&=\max\limits_{F^*( y)\le 1} \max\limits_{G(x)\le 1}  \vec x\cdot (T^\top\vec y-\vec u)=\max\limits_{F^*( y)\le 1}G^*(T^\top\vec y-\vec u) .   
\end{align*}
The proof is completed.
\end{proof}

\begin{example}
In the Dinkelbach-type scheme, 
we work on  a convex optimization 
$$\vec  x^{k+1}\in \argmin\limits_{\vec x\in \mathbb{B}} \{F_1(\vec x)+r^k G_2(\vec x) -(\langle \vec u^k,\vec x\rangle+r^k \langle \vec v^k,\vec x\rangle) + H_{\vec x^k}(\vec x)\},$$ 
and by Proposition \ref{pro:dual-inner-problem}, 
the equivalent dual problem of this optimization  is 
$$\vec y^{k+1}\in \argmin\limits_{\vec y\in\Omega _k}\|\vec u^k+r^k\vec v^k-\vec y\|_2^2,\;\;\;\; \vec x^{k+1}=\frac{\vec u^k+r^k\vec v^k-\vec y^{k+1}}{\|\vec u^k+r^k\vec v^k-\vec y^{k+1}\|_2}$$
where we take $\mathbb{B}$ as the $l^2$-ball, and $\Omega_k$ is the dual convex body of $\{\vec x:F_1(\vec x)+r^kG_2(\vec x)+H_{\vec x^k}(\vec x)\le 1\}$. 

\end{example}

\begin{remark}
Another equivalent formulation of Proposition \ref{pro:dual-inner-problem} can be written as 
\begin{equation}\label{eq:max-dual-FGu}
\max\limits_{x\ne0}\frac{F(T\vec x)-\vec x\cdot\vec u}{G(\vec x)}=\max\limits_{y\ne0}\frac{G^*(T^\top\vec y-\vec u) }{F^*(\vec y)},\;\;\forall \vec u    
\end{equation}
and
$$ \min\limits_{x\ne0}\frac{F(T\vec x)-\vec x\cdot\vec u}{G(\vec x)}= -\min\limits_{y\ne0}\frac{G^*(\vec u-T^\top\vec y) }{F^*(\vec y)}\;\;\text{ whenever }\vec u\not\in \mathrm{int}(T^\top \nabla F(\vec 0)).$$
Also, \eqref{eq:max-dual-FGu} can be formulated as
\begin{align*}
\max\limits_{x\in \mathbb{B}_G}F(T\vec x)-\vec x\cdot\vec u&=\max\limits_{x\in \mathrm{Ext}(\mathbb{B}_G)}F(T\vec x)-\vec x\cdot\vec u
\\&=\max\limits_{y\in \mathbb{B}_{F^*}}G^*(T^\top\vec y-\vec u)= \max\limits_{y\in \mathrm{Ext}(\mathbb{B}_{F^*})}G^*(T^\top\vec y-\vec u),\;\;\forall \vec u ,
\end{align*}
$$$$
where  $\mathrm{Ext}(\mathbb{B}_G)$ and $\mathrm{Ext}(\mathbb{B}_{F^*})$ are extreme  sets of the convex bodies $\mathbb{B}_G:=\{\vec x:G(\vec x)\le1\}$ and  $\mathbb{B}_{F^*}:=\{\vec y:F^*(\vec y)\le1\}$, respectively.  
This allows  us to prove many results  in a  short  and elegant way. 
\end{remark}

\begin{example}
We may simply call the vertex $p$-Laplacian   on an oriented  hypergraph the hypergraph $p$-Laplacian, and we call the (hyper-)edge  $p$-Laplacian   on an oriented  hypergraph the dual hypergraph $p$-Laplacian. By Proposition \ref{pro:p-Lap-dual}, we only need to concentrate on  the hypergraph $p$-Laplacian for $p\in [1,2]$. 
\end{example}

\begin{example}
Given $\vec v_i\in\R^n$, $i=1,\cdots,m$, and $p\ge 1$ with $p^*$ as its H\"older conjugate, we have
$$\max\limits_{\|\vec x\|_p\le 1}\sum_{i=1}^m|\vec v_i\cdot\vec x|= \max\limits_{\varepsilon_i\in\{-1,1\}}\|\sum_{i=1}^m\varepsilon_i\vec v_i\|_{p^*}.$$
For $p=2$, the above equality reveals a dual form of the $l^1$-Chebyshev  constant  (Proposition 3  in \cite{AmbrusNietert}).
We can similar obtain a dual form of  the $l^p$-polarization  (Chebyshev)  constant via the inequality
$$\max\limits_{\|\vec x\|_2\le 1}\sum_{i=1}^m|\vec v_i\cdot\vec x|^p= \max\limits_{\sum_{i=1}^m|\varepsilon_i|^{p^*}=1}\|\sum_{i=1}^m\varepsilon_i\vec v_i\|_2^p.$$
\end{example}

\begin{example}
Let $G=(V,E)$ be a simple graph without bipartite component, then the dual Cheeger constant of $G$ possesses  the new  continuous representation:
$$h_1^+(G)=1-\min\limits_{x\ne 0}\frac{\sum_{i\sim j}|x_i+x_j|}{\deg_i|x_i|}=1-\min\limits_{\vec y:\exists i\text{ s.t. }\sum_{e\ni i}y_e\ne 0}\frac{\max\limits_{i\in V}\frac{1}{\deg_i}|\sum_{e\ni i}y_e|}{\min\limits_{\vec z:\sum_{e\ni i}z_e= 0,\forall i}\|\vec y+\vec z\|_\infty}.$$
\end{example}

\begin{example}
Given a simple graph $G$, for any edge $e=\{i,j\}$, we let $\epsilon_{ie}\in\{-1,1\}$ be such that $\epsilon_{ie}=-\epsilon_{je}$, which indeed assigns an  orientation on $G$.  
Then, the  maxcut of  $G$ has the following equivalent  continuous formulation:
$$\max\limits_{x\ne 0}\frac{\sum_{i\sim j}|x_i-x_j|}{\|\vec x\|_\infty}=\max\limits_{y\ne 0}\frac{\sum_{i\in V}|\sum_{e\ni i} \epsilon_{ie}y_e|}{\|\vec y\|_\infty}.$$
\end{example}

\begin{example}
For a simple and connected graph $G$, its Cheeger constant equals 
$$\min\limits_{x\ne\mathrm{const}}\frac{\sum_{i\sim j}|x_i-x_j|}{\min\limits_{t\in R}\sum_{i\in V}\deg_i|x_i+t|}=\min\limits_{\vec y:\exists i\text{ s.t. }\sum_{e\ni i}\epsilon_{ie}y_e\ne 0} \frac{\max_{i\in V}\frac{1}{\deg_i}|\sum_{e\ni i} \epsilon_{ie}y_e|}{\min\limits_{\vec z:\sum_{e\ni i}\epsilon_{ie}z_e= 0,\forall i}\|\vec y+\vec z\|_\infty}.$$
\end{example}

{ \linespread{0.95} 
}

\end{document}